\documentclass[reqno,fullline,10pt]{amsart}
\usepackage{amsthm}
\usepackage{tikz}
\usetikzlibrary{positioning}
\usepackage{amssymb}
\usepackage{mathrsfs}
\usepackage{caption}
\usepackage{pgfplots} 
\usetikzlibrary{patterns} 
\usepackage{amsmath} 
\usepackage{pgfplots} 
\usetikzlibrary{patterns}
\usepackage{booktabs} 
\usepackage{booktabs} 
\usepackage{amssymb, amsmath}
\usepackage{amsthm, amsfonts,mathrsfs}
\usepackage[T1]{fontenc}
\usepackage{times}
\usepackage{mathrsfs}
\usepackage{booktabs}
\usepackage{epsfig}
\usepackage{color}
\usepackage[all]{xypic}
\usepackage{amssymb}
\usepackage{cases}
\usepackage{subeqnarray}
\allowdisplaybreaks[4]
\newtheorem{theorem}{Theorem}[section]

\newtheorem{definition}{Definition}[section]
\newtheorem{lemma}{Lemma}[section]
\newtheorem{proposition}{Proposition}[section]
\newtheorem{remark}{Remark}[section]
\newtheorem{example}{Example}[section]
\allowdisplaybreaks
\numberwithin{equation}{section}

\begin{document}

\title[Critical curve for semilinear Euler-Poisson-Darboux-Tricomi equations]{Critical curve for weakly coupled system of  semilinear  Euler-Poisson-Darboux-Tricomi equations}

\author[Y.- Q. Li]{Yuequn Li}
\address[Yuequn Li]{
School of Mathematical Sciences, Nanjing Normal University, Nanjing 210023, China}
\email{yqli1214@163.com}

\author[Fei Guo]{Fei Guo}
\address[Fei Guo]{
School of Mathematical Sciences and Key Laboratory for NSLSCS,
Ministry of Education, Nanjing Normal University,
Nanjing 210023, China}
\email{guof@njnu.edu.cn}


\begin{abstract}
This paper investigates a weakly coupled system of semilinear Euler-Poisson-Darboux-Tricomi equations (EPDTS) with power-type nonlinear terms.
More precisely, in the case where the damping terms dominate over the mass terms, the critical curve in the $p-q$ plane that delineates the threshold between global existence and blow-up for the EPDTS is given by
\begin{equation*}
\Gamma_m(n,p,q,\beta_1,\beta_2)=0,
\end{equation*}
where $\Gamma_m$ is defined by (\ref{gammam}).
Through the construction of new test functions, the blow-up problem is addressed when $\Gamma_m(n,p,q,\beta_1,\beta_2)\geq0$.  Based on the $(L^1\cap L^2)-L^2$  estimates of  the  solution to  the corresponding linear equation established in our previous work \cite{LiGuo2025}, we derive the global existence of solutions with  small initial data when $\Gamma_m(n,p,q,\beta_1,\beta_2)<0$, provided that the damping terms prevail over the mass terms.
\end{abstract}

\maketitle



\vskip 0.2cm


\section{Introduction}\label{s1}
This paper investigates the critical curve problem for the weakly coupled system of  semilinear  Euler-Poisson-Darboux-Tricomi equations
\begin{equation}\label{eqs}
  \left\{
\begin{aligned}
&\partial_t^2u-t^{2m}\Delta u+\frac{\mu_1}{t}\partial_tu+\frac{\nu_1^2}{t^2}u=|v|^p,&\quad& t>1, \  x\in \mathbb{R}^n,\\
&\partial_t^2v-t^{2m}\Delta v+\frac{\mu_2}{t}\partial_tv+\frac{\nu_2^2}{t^2}v=|u|^q,&\quad& t>1,\  x\in \mathbb{R}^n, \\
&(u, \partial_tu, v, \partial_tv)(1,x)= (u_0,u_1,v_0,v_1)(x),               &\quad& x\in\mathbb{R}^n,
\end{aligned}
 \right.
  \end{equation}
where $m>-1$, $n\geq1$ denotes the space dimension, $\mu_i,\nu_i^2$, $ i=1,2$ are  nonnegative numbers and $p,q>1$. We find that there exists a curve related to $\mu_i,\nu_i$,  $i=1,2$  in $p-q$ plane, which can be used to distinguish the blow-up phenomenon and the global existence of solutions to (\ref{eqs}).

The critical exponent for the classical semilinear wave equation
\begin{equation}\label{str}
\partial_t^2u-\Delta u=\vert u\vert^p
\end{equation}
is the well-known Strauss exponent $p_S(n)$, where $n\geq1$ denotes the space dimension. Namely, local solutions with small initial data  exist globally when $p>p_S(n)$ and the blow-up phenomenon occurs if $1<p\leq p_S(n)$ no matter how small the initial data may be.  Here $p_S(1)=+\infty$, and for $n\geq2$, $p_S(n)$ is the positive root of the quadratic equation
\begin{equation}\label{strausseq}
(n-1)p^2-(n+1)p-2=0,
 \end{equation}
see \cite{De1997, GeLiSo1997, Glassey1981,  IkSoWa2019,John1979, Strauss1981, Ta2001, Zh2006, Zhou2007, Zhou2014} for  details. Following  the basic resolution of the Strauss  exponent problem, many scholars \cite{ AgKuTa2000,  De1997, DeGeMi1997, DeMi1998, GeTaZh2006, IkSoWa2019,KuTaWa2012}  have  studied the weakly coupled system of semilinear wave equations
\begin{equation}\label{weak1}
  \left\{
\begin{aligned}
&\partial_t^2u-\Delta u=|v|^p,&\quad& t>0, \  x\in \mathbb{R}^n,\\
&\partial_t^2v-\Delta v=|u|^q,&\quad& t>0,\  x\in \mathbb{R}^n, \\
&(u,\partial_tu,v,\partial_tv)(0,x)= (u_0,u_1,v_0,v_1)(x),               &\quad& x\in\mathbb{R}^n
\end{aligned}
 \right.
  \end{equation}
with $p,q>1$. The critical curve for (\ref{weak1}) in the $p-q$ plane is described by $\Gamma_W(n,p,q)=0$, where
\begin{equation}\label{curve1}
\Gamma_W(n,p,q):=\max\Big\{\frac{p+2+q^{-1}}{pq-1},\frac{q+2+p^{-1}}{pq-1}\Big\}-\frac{n-1}{2}.
\end{equation}
Roughly speaking, the small data global solutions to (\ref{weak1}) exist when $\Gamma_W(n,p,q)<0$ , while if  $\Gamma_W(n,p,q)\geq0$, solutions  blow up in finite time.
Observe that $\Gamma_W(n,p,p)=0$ if and only if (\ref{strausseq}) holds, so the critical curve problem for (\ref{weak1}) can be regarded as a generalization of the Strauss exponent problem.
However,  the study of coupled system (\ref{weak1}) is not a simple generalization of the results of its single equation (\ref{str}). On the contrary, the critical curve described by $\Gamma_W(n,p,q)\geq0$ expands the blow-up range compared to the single case: $1<p\leq p_S(n), 1<q\leq p_S(n)$.  In fact, note that
\begin{equation}\label{1leqpleqpsn}
1<p\leq p_S(n)\iff \frac{1+p^{-1}}{p-1}\geq\frac{n-1}{2}
\end{equation}
and
\begin{equation}\label{2leqpleqpsn}
\max\Big\{\frac{p+2+q^{-1}}{pq-1},\frac{q+2+p^{-1}}{pq-1}\Big\}\leq \max\Big\{\frac{1+p^{-1}}{p-1},\frac{1+q^{-1}}{q-1}\Big\}
\end{equation}
hold,  where the equalities  in (\ref{1leqpleqpsn}) and (\ref{2leqpleqpsn})  hold  if and only if $p=q$.  This observation indicates that there may exist some $p$ or $q$ greater than $p_S(n)$ such that  $\Gamma_{W}(n,p,q)\geq0$ (the blow-up region). For example, for $n=3$, choosing $p=2<p_S(3)=1+\sqrt{2}<q=2.5$, a simple calculation yields $\Gamma_{W}(3,2,2.5)=0.25>0$.

Consider the weakly coupled system of semilinear damped wave equations
\begin{equation}\label{weak2}
  \left\{
\begin{aligned}
&\partial_t^2u-\Delta u+b(t)\partial_tu=|v|^p,&\quad& t>0, \  x\in \mathbb{R}^n,\ p>1,\\
&\partial_t^2v-\Delta v+b(t)\partial_tv=|u|^q,&\quad& t>0,\  x\in \mathbb{R}^n,\ q>1,  \\
&(u,v,\partial_tu,\partial_tv)(0,x)= (u_0,u_1,v_0,v_1)(x),               &\quad& x\in\mathbb{R}^n.
\end{aligned}
 \right.
  \end{equation}
If $b(t)=1,$ in contrast to (\ref{weak1}),  its critical curve is described by $\Gamma_{DW}(n,p,q)=0$, where \begin{equation*}
\Gamma_{DW}(n,p,q):=\max\Big\{\frac{p+1}{pq-1},\frac{q+1}{pq-1}\Big\}-\frac{n}{2},
\end{equation*}
highlighting  that the  appearance of the damping terms $\partial_tu, \partial_tv$ has a notable impact on the coupled system  (\ref{weak2}).  About the details, one can refer to \cite{Nara2011, Nara2009, NishWaka2014,NiWa2015, OgTa2010, SunWang2007}. Note $\Gamma_{DW}(n,p,p)=0\iff p=p_F(n)$, where $p_F(n):=1+\frac{2}{n}$ denotes the Fujita exponent, which is the critical exponent for the heat equation $\partial_tu-\Delta u=|u|^p$ and the single damped wave equation
\begin{equation}\label{singlewavedamped}
\partial_t^2u-\Delta u+\partial_tu=|u|^p,
\end{equation}
one can see \cite{Fujita1969, IkTa2005,ToYo2001, Zh2001}. Hence
we can claim that the critical curve $\Gamma_{DW}(n,p,q)=0$ generalizes the critical exponent for (\ref{singlewavedamped}).
Furthermore, in the case of general effective  damping,  i.e.,  $b(t)=(1+t)^{-\beta}$, $\beta\in(-1,1)$,   the authors \cite{NiWa2015} showed that  the critical curve  is still described by $\Gamma_{DW}(n,p,q)=0$. This explains why the damping term is referred to as \textit{effective} in this case. In addition, the authors   \cite{Mo2018, MoRe2018} established relevant results of  global solutions under different power nonlinearities  or  additional regularity assumptions on the initial data.

For the weakly coupled system of semilinear wave equations with scale-invariant damping and mass
\begin{equation}\label{weak3}
  \left\{
\begin{aligned}
&\partial_t^2u-\Delta u+\frac{\mu_1}{t}\partial_tu+\frac{\nu_1^2}{t^2}u=|v|^p,&\quad& t>1, \  x\in \mathbb{R}^n,\ p>1,\\
&\partial_t^2v-\Delta v+\frac{\mu_2}{t}\partial_tv+\frac{\nu_2^2}{t^2}v=|u|^q,&\quad& t>1,\  x\in \mathbb{R}^n,\ q>1,  \\
&(u,v,\partial_tu,\partial_tv)(1,x)= (u_0,u_1,v_0,v_1)(x),               &\quad& x\in\mathbb{R}^n.
\end{aligned}
 \right.
  \end{equation}
Chen-Palmieri \cite{ChenPa2019}  showed  that,  in the case of
\begin{equation}\label{deltai}
\delta_i:=(\mu_i-1)^2-4\nu_i^2\geq(n+1)^2,\ i=1,2,
\end{equation}
the critical curve for (\ref{weak3}) is  given by $\Gamma_0(n,p,q,\beta_1,\beta_2)=0$, where
\begin{equation}\label{E}
\Gamma_0(n,p,q,\beta_1,\beta_2):=\max\Big\{\frac{p+1}{pq-1}-\frac{\beta_1-1}{2},\frac{q+1}{pq-1}-\frac{\beta_2-1}{2}\Big\}-\frac{n}{2},
\end{equation}
where $\beta_i=\frac{\mu_i+1-\sqrt{\delta_i}}{2}, i=1,2$.
Clearly, the study of the critical curve for  (\ref{weak3}) is based on the corresponding single power-type semilinear wave equation
\begin{equation}\label{single}
\partial_t^2u-\Delta u+\frac{\mu}{t}\partial_tu+\frac{\nu^2}{t^2}u=|u|^p.
\end{equation}
Since the linear equation of (\ref{single}) is invariant under the transformation  $\tilde{u}(\lambda t,\lambda x)$, the damping term $\frac{\mu}{t}\partial_tu$ and the mass term $\frac{\nu^2}{t^2}u$  can be viewed as having the same scaling.  Following Palmieri's analysis,  this balance induces an interaction where the relative sizes of $\mu$ and $\nu^2$  play a decisive role in determining whether the damping  term or the mass term prevails in (\ref{single}). In fact, the magnitude of the number $\delta:=(\mu-1)^2-4\nu^2$ can characterize this interaction. Specifically, Nascimento-Palmieri-Reissig  \cite{PaRe2017,Pa2018,PaRe2018} proved that for any $\delta\geq0,$ the local (in time) solutions to (\ref{single}) blow up  in finite time when $1<p\leq p_F(n+\frac{\mu-1-\sqrt{\delta}}{2})$  by  means of  the test function approach. They  then established the decay estimates of  solutions to the corresponding linear equation of (\ref{single}) through the application of Fourier analysis techniques, thus gave that the  solutions  to (\ref{single}) exist globally if  $p>p_F(n+\frac{\mu-1-\sqrt{\delta}}{2})$ for any $\delta\geq(n+2\sigma-1)^2$, where $\sigma$ is a positive number describing the  regularity of the initial data,  and $p_F$ is Fujita exponent mentioned above.   Roughly speaking,  when the damping term dominates over the mass term, in other words, $\delta$  is   relatively large, (\ref{single}) is parabolic-like as  $p_F(n+\frac{\mu-1-\sqrt{\delta}}{2})$ is its critical exponent.  In \cite{Tu2019}, Palmieri-Tu also proved the blow-up phenomenon if $1<p\leq p_S(n+\mu)$ for any $\delta\geq0$.  Palmieri-Reissig  \cite{PaRe2019}  proved that the solution to (\ref{single}) also blows up in finite time when $1<p\leq\max\{p_F(n+\frac{\mu-1-\sqrt\delta}{2}),p_S(n+\mu)\}$.   Moreover, for other ranges of $\delta$, partial results can be found in \cite{DaPa2021,PaRe2017,Pa2018,PaRe2019,PaRe2018,Tu2019}.  From the above results, it is clear that the value of $\delta$ plays a significant role  in the interplay between the Strauss exponent, the Fujita exponent, and  their balance.  It is also worth mentioning that the critical curve for  (\ref{weak3}) is   not merely  a straightforward extension of the corresponding results for the single case (\ref{single}), and this perspective  is explained comprehensively by Chen-Palmieri in \cite{ChenPa2019}.

In our recent paper  \cite{LiGuo2025}, we studied the single Euler-Poisson-Darboux-Tricomi equation with a power nonlinearity
\begin{equation}\label{single2}
  \left\{
\begin{aligned}
&\partial_t^2u-t^{2m}\Delta u+\frac{\mu}{t}\partial_tu+\frac{\nu^2}{t^2}u=|u|^p,&\quad& t>1, \ m>-1, \\
&u(1,x)= u_0(x), \ \partial_tu(1,x)= u_1(x),               &\quad& x\in\mathbb{R}^n.
\end{aligned}
 \right.
  \end{equation}
By applying Fourier analysis methods, the corresponding linear problem
\begin{equation}\label{single21}
\partial_t^2u-t^{2m}\Delta u+\frac{\mu}{t}\partial_tu+\frac{\nu^2}{t^2}u=0
\end{equation}
associated with (\ref{single2}) is reduced to  a Bessel equation in the frequency space.  Using the properties of Bessel functions, we obtained the $(L^1 \cap L^2) -L^2$ estimates of  the  solutions to  (\ref{single21}).   On this basis, we established  the global existence of solutions to (\ref{single2})  by Duhamel's principle and contraction mapping principle,   provided that $p>p_{F}((m+1)n+\frac{\mu-1-\sqrt\delta}{2})$ in the case of $\delta\geq(m+1)^2(n+2\sigma-1)^2$, where $\sigma>0$.
 Combining   this with the blow-up result established by Palmieri \cite{Pa2025}, namely,  solutions blow up in finite time if $1<p\leq p_{F}((m+1)n+\frac{\mu-1-\sqrt\delta}{2})$ for any $\delta\geq0$, we conclude that $p_{F}((m+1)n+\frac{\mu-1-\sqrt\delta}{2})$ is the critical exponent of (\ref{single}) if $\delta$ is large to a certain extent. As shown in \cite{LiGuo2025,Pa2025}, $\delta$  continues to serve as the measure of the interaction between the damping term and mass term in (\ref{single2}), playing a decisive role in identifying the critical exponent that distinguish between the  blow-up phenomenon and  the global existence of  solutions.

Comparing  the  blow-up results of the wave equation (\ref{single}) and its weakly coupled system (\ref{weak3}), we can infer by analogy  that the blow-up result for  (\ref{eqs}) is unlikely to be a simple rectangular region in the $p-q$ plane described by
\begin{equation*}
1<p\leq p_F\big((m+1)n+\frac{\mu_1-1-\sqrt{\delta_1}}{2}\big), \ \ 1<q\leq p_F\big((m+1)n+\frac{\mu_2-1-\sqrt{\delta_2}}{2}\big).
\end{equation*}
Instead, it should be determined by the interplay among the coefficients $\mu_i,\nu_i$$(i=1,2)$  of the damping terms, the mass terms, as well as the parameter $m$. Building on the relevant results established for the single equation (\ref{single2}), this paper focuses on studying the critical curve for (\ref{eqs}). Regarding the blow-up phenomenon, it is clear that the Gellerstedt operator $\partial_t^2-t^{2m}\Delta$ is different from the wave operator $\partial_t^2-\Delta$,  which compels us to construct new test functions  tailored for the system  (\ref{eqs}). Hence we construct  new test functions that depend on $m$ by the smooth cutoff function to prove the blow-up result when $\Gamma_m(n,p,q,\beta_1,\beta_2)\geq0$, where
\begin{equation}\label{gammam}
\Gamma_m(n,p,q,\beta_1,\beta_2):=\max\Big\{\frac{p+1}{pq-1}-\frac{\beta_1-1}{2},\frac{q+1}{pq-1}-\frac{\beta_2-1}{2}\Big\}-\frac{(m+1)n}{2}
\end{equation}
and $\beta_i,  i=1,2$ are defined by (\ref{E}). Conversely, we expect that if
\begin{equation}\label{converse}
\Gamma_m(n,p,q,\beta_1,\beta_2)<0,
 \end{equation}
the   small data solutions  to (\ref{eqs}) exist globally.  Based on  the $(L^1\cap L^2)-L^2$ estimates of  the  solutions to   the corresponding linear equation (\ref{single21}) provided in our previous work \cite{LiGuo2025}, we establish the global existence of  solutions to  (\ref{eqs}) for some large $\delta_i:=(\mu_i-1)^2-4\nu^2$,  $i=1,2$  in six cases by distinguishing the  necessary conditions of (\ref{converse}) and the regularity assumptions of the initial data.  According to the seven theorems presented in this paper, we  conclude that
$\Gamma_m(n,p,q,\beta_1,\beta_2)=0$
represents the critical curve for (\ref{eqs}) in the $p-q$ plane when $\delta_i, i=1,2$ are appropriately large.
Through our main results, we  also conclude that the study of the weakly coupled system (\ref{eqs}) is not merely a straightforward extension of the results for its single equation (\ref{single2}).  This point is explained in detail in Remark \ref{noeasy}.

\noindent\textbf{Notations.}

(1) $a \lesssim b$  denotes that  there exists a constant
$C>0$ such that $a\leq Cb$. 

(2) $\vert D\vert^\sigma$ denotes the pseudo-differential operators with symbol $\vert \xi\vert^\sigma$. $H^{\sigma}_p{(\mathbb{R}^n)}$, $\dot{H}^{\sigma}_p{(\mathbb{R}^n)}$ mean the non-homogeneous Sobolev space and homogeneous Sobolev space respectively,  equipped with the norm $\Vert f\Vert_{H^{\sigma}_p{(\mathbb{R}^n)}}=\Vert(1+\vert D\vert^2)^{\frac{\sigma}{2}}f\Vert_{L^p(\mathbb{R}^n)}$, $\Vert f\Vert_{\dot{H}^{\sigma}_p{(\mathbb{R}^n)}}=\Vert\vert D\vert^{\sigma}f\Vert_{L^p(\mathbb{R}^n)}$. For simplicity in writing, we omit $\mathbb{R}^n$.


(3) $\lceil \cdot\rceil$ is the ceiling function, i.e., $\lceil x\rceil:=\min\{k\in\mathbb{Z}: x\leq k\}$; $[\cdot]_+$  is the positive part function, i.e., $[x]_+:=\max\{x,0\}$.

(4) $B_r$ represents the  ball centered at the origin with radius $r$, that is $B_r=\{x\in\mathbb{R}^n:\vert x\vert\leq r\}$.

This paper is organized as follows. In Section \ref{section2}, we provide the definition of the energy solution to (\ref{eqs}) and present the main results.  Section \ref{section3blowup} is devoted to give the blow-up result, that is Theorem \ref{theorem1}.  In Section \ref{section4high},  we list the results established in our recent work \cite{LiGuo2025},  then we establish the  global well-posedness with high regularity of initial data,  i.e., Theorems \ref{theorem11}-\ref{theorem13}.  And we  establish the  global well-posedness with low regularity of initial data in Section 5,  i.e., Theorems \ref{theorem21}-\ref{theorem23}. In the appendix, we provide a brief outline of the proof of  the local  (in time) existence of solutions to (\ref{eqs}), namely,  Proposition \ref{localexistence}.

\section{Main results}\label{section2}

In this section, we present the blow-up result and global existence results for (\ref{eqs}), thereby determining its critical curve.
We first give the definition of energy solution to (\ref{eqs}). Let us introduce the space
\begin{equation}\label{norm1}
D^\sigma=\left\{
\begin{aligned}
&\bigl(L^1\cap H^\sigma\bigr)\times\bigl(L^1\cap H^{\sigma-1}\bigr),&\quad &\sigma\geq1,\\
&\bigl(L^1\cap H^\sigma\bigr)\times\bigl(L^1\cap L^2\bigr), &\quad &\sigma\in (0,1)
\end{aligned}
 \right.
\end{equation}
with the norm
\begin{equation}\label{norm2}
\Vert(f,g)\Vert_{D^\sigma}=\left\{
\begin{aligned}
&\Vert f\Vert_{L^1}+\Vert f\Vert_{H^\sigma}+\Vert g\Vert_{L^1}+\Vert g\Vert_{H^{\sigma-1}},&\quad&\sigma\geq1,\\
&\Vert f\Vert_{L^1}+\Vert f\Vert_{H^\sigma}+\Vert g\Vert_{L^1}+\Vert g\Vert_{L^{2}}, &\quad &\sigma\in (0,1),
\end{aligned}
 \right.
\end{equation}
and $D^1:=D.$

\begin{definition}\label{defenergysolution}
Let $(u_0,u_1,v_0,v_1)\in D\times D$ be compactly supported with
\begin{equation}\label{compact1}
\text{supp} (u_0,u_1,v_0,v_1)\subset B_M(0)
\end{equation}
for some $M>0$.
We say that  $(u,v)$ is an energy solution to (\ref{eqs}) on $[1,T)$ if
\begin{align}
&(u,v)\in \Big(C\bigl([1,T); H^1(\mathbb{R}^n)\bigr) \cap C^1\bigl([1,T); L^2(\mathbb{R}^n)\bigr)\Big)^2,\label{solutiondefu}
\end{align}
and it satisfies  the support property
\begin{equation}\label{compact2}
\text{supp} (u,v)(t,\cdot)\subset B_{\phi_m(t)-\phi_m(1)+M}, \ \phi_m(t)=\frac{t^{m+1}}{m+1}, \  \text{for any}\ t\in(1,T)
\end{equation}
as well as the following integral equalities
\begin{equation}\label{def3}
\resizebox{   0.912\hsize}{!}{$  \begin{split}
&\iint_{[1,T)\times\mathbb{R}^n}u(t,x)\Bigl(\partial_t^2\varPhi_1(t,x)-t^{2m}\Delta\varPhi_1(t,x)
-\partial_t\bigl(\frac{\mu_1}{t}\varPhi_1(t,x)\bigr)+\frac{\nu_1^2}{t^2}\varPhi_1(t,x)\Bigr)dxdt\\
&=\int_{\mathbb{R}^n}\Bigl(-u_0(x)\partial_t\varPhi_1(1,x)+\bigl(\mu_1 u_0(x)+u_1(x)\bigr)\varPhi_1(1,x)\Bigr)dx\\
&\quad+\iint_{[1,T)\times\mathbb{R}^n}\vert v(t,x)\vert^p\varPhi_1(t,x)dxdt
\end{split}     $}
\end{equation}
and
\begin{equation}\label{def32}
\resizebox{   0.912\hsize}{!}{$  \begin{split}
&\iint_{[1,T)\times\mathbb{R}^n}v(t,x)\Bigl(\partial_t^2\varPhi_2(t,x)-t^{2m}\Delta\varPhi_2(t,x)
-\partial_t\bigl(\frac{\mu_2}{t}\varPhi_2(t,x)\bigr)+\frac{\nu_2^2}{t^2}\varPhi_2(t,x)\Bigr)dxdt\\
&=\int_{\mathbb{R}^n}\Bigl(-v_0(x)\partial_t\varPhi_2(1,x)+\bigl(\mu_2 v_0(x)+v_1(x)\bigr)\varPhi_2(1,x)\Bigr)dx\\
&\quad+\iint_{[1,T)\times\mathbb{R}^n}\vert u(t,x)\vert^q\varPhi_2(t,x)dxdt,
\end{split}     $}
\end{equation}
for any  $\varPhi_1(t,x),\varPhi_2(t,x)\in C_0^{\infty}\bigl([1,T)\times \mathbb{R}^n\bigr)$.
\end{definition}
\begin{remark}
The support property (\ref{compact2}) of the solution $(u,v)$ reflects the finite propagation speed of the wave, which has been explained in \cite{He2017II, He2017I, Pa2025}.
\end{remark}

\begin{proposition}{\bf{\text{(Local existence).}}}\label{localexistence}
Let $n\geq1, m>-1$, $\mu_i,\nu_i^2\geq0$, $i=1,2$ such that $\delta_1,\delta_2>0$, where $\delta_i,i=1,2$ are defined by (\ref{deltai}).  Suppose the  initial data $(u_0,u_1,v_0,v_1)\in D^\sigma\times D^\sigma$ with some $\sigma>0$ and satisfy the support condition (\ref{compact1}).
\begin{equation*}
\text{If}\ \ \sigma\geq1,\ \text{we assume}\ p,q>\lceil \sigma\rceil\ \text{and}\ \left
\{
\begin{aligned}
&1< p,q\leq1+\frac{2}{n-2\sigma}, \ &n>2\sigma,\\
&1< p,q<+\infty,\ & n\leq2\sigma;
\end{aligned}
\right.
\end{equation*}
\begin{equation*}
\text{if}\ \ 0<\sigma<1,\ \text{we assume}\ \left
\{
\begin{aligned}
&1< p,q\leq\frac{n}{n-2\sigma}, \ &n>2\sigma,\\
&1< p,q<+\infty,\ & n\leq2\sigma.
\end{aligned}
\right.
\end{equation*}
Then there exist a $T>1$ and a unique solution
\begin{equation*}
(u,v)\in
\left
\{
\begin{aligned}
&\Big(C\big([1,T);H^\sigma\big)\cap C^1\big([1,T);H^{\sigma-1}\big)\Big)^2,\ & \sigma\geq1,\\
&\Big(C\big([1,T);H^\sigma\big)\Big)^2,\ & 0<\sigma<1
\end{aligned}
\right.
\end{equation*}
to (\ref{eqs}). Moreover,  $(u,v)$ satisfies the support property (\ref{compact2}) and the  integral equalities (\ref{def3}), (\ref{def32}).
\end{proposition}


\begin{theorem}{\bf{\text{(Blow-up).}}}\label{theorem1}
Let $n\geq1,m>-1$, $\nu_1^2,\nu_2^2\geq0$ and
$$
\left\{
\begin{aligned}
&\mu_1,\mu_2>1,&\ &\text{if}\ -1<m<0,\\
&\mu_1,\mu_2>0, &\ &\text{if}\ m\geq0
\end{aligned}
\right.
$$ such that $\delta_1,\delta_2>0$,  where
$\delta_1,\delta_2$ are defined by (\ref{deltai}).
Suppose that  $(u_0,u_1,v_0,v_1)\in D\times D$ are  initial data satisfying the support condition (\ref{compact2}) and
\begin{equation}\label{integral>0}
\begin{aligned}
&\liminf\limits_{R\rightarrow+\infty}\int_{\vert x\vert<R}\Bigl( \frac{\mu_1-1+\sqrt{\delta_1}}{2}u_0(x)+u_1(x)\Bigr)dx>0,\\
&\liminf\limits_{R\rightarrow+\infty}\int_{\vert x\vert<R}\Bigl( \frac{\mu_2-1+\sqrt{\delta_2}}{2}v_0(x)+v_1(x)\Bigr)dx>0.
\end{aligned}
\end{equation}
If $p,q>1$ satisfy
\begin{align}
&p>\frac{1+\beta_1}{2m+1+\beta_2},\ q>\frac{1+\beta_2}{2m+1+\beta_1}\label{keycon2}
\end{align}
and
\begin{align}
&\Gamma_m(n,p,q,\beta_1,\beta_2)\geq0,\label{keycon1}
\end{align}
then the energy solution $(u,v)$ to (\ref{eqs}) blows up in finite time, where $\Gamma_m(n,p,q,\beta_1,\beta_2)$ is defined by (\ref{gammam}).
\end{theorem}

\begin{remark}
$\mu_i>1, i=1,2$ when $m\in(-1,0)$ are technical conditions introduced  to ensure that $2m+1+\beta_i>0,i=1,2$ in (\ref{keycon2}).
We illustrate the non-emptiness of the conditions imposed in Theorem  \ref{theorem1}  with relevant examples. In other words, we claim that the solution to (\ref{eqs})  blows up in the case of Examples \ref{ex1}-\ref{ex2}.
\end{remark}
\begin{example}\label{ex1}
In the case of $m>0$, for example, $m=0.1, n=2,  \mu_1=3, \nu_1=0.25,\mu_2=0.5,\nu_2=0.125$,   a straightforward  calculation reveals $\delta_1=\frac{15}{4}$, $\delta_2=\frac{3}{16}$, $\beta_1=2-\frac{\sqrt{15}}{4}$, $\beta_2=\frac{3}{4}-\frac{\sqrt{3}}{8}$.  Choosing  $p=2.1>\frac{1+\beta_1}{2m+1+\beta_2}=\frac{3-\frac{\sqrt{15}}{4}}{1.95-\frac{\sqrt{3}}{8}}\approx1.172,q=2.2>
\frac{1+\beta_2}{2m+1+\beta_1}=\frac{1.75-\frac{\sqrt{3}}{8}}{3.2-\frac{\sqrt{15}}{4}}\approx0.687$, we obtain   $\Gamma_m(n,p,q,\beta_1,\beta_2)\approx0.017>0$.
\end{example}

\begin{example}\label{ex2}
In the case of $m\in(-1,0)$, such as $m=-0.6, n=3,\mu_1=2,\nu_1=0.125,\mu_2=1.8,\nu_2=0.25$,  then $\delta_1=\frac{15}{16},\delta_2=0.39,\beta_1=\frac{3}{2}-\frac{\sqrt{15}}{8},\beta_2=1.4-\frac{\sqrt{0.39}}{2}$. Choosing $p=2.3>\frac{1+\beta_1}{2m+1+\beta_2}\approx2.271$ and $q=2.7>\frac{1+\beta_2}{2m+1+\beta_1}\approx2.559$, we have $\Gamma_m(n,p,q,\beta_1,\beta_2)\approx0.066>0$ .
\end{example}


After establishing the blow-up result, we anticipate that
\begin{equation}\label{criticalcurve}
\Gamma_m(n,p,q,\beta_1,\beta_2)=0
\end{equation}
is the critical curve for (\ref{eqs}). Let
\begin{equation}\label{tildepq}
\begin{aligned}
&\tilde{p}:=\tilde{p}(m,n,\beta_1,\beta_2)=\frac{(m+1)n+\beta_1+1}{(m+1)n+\beta_2-1}
=1+\frac{2+\beta_1-\beta_2}{(m+1)n+\beta_2-1},\\
&\tilde{q}:=\tilde{q}(m,n,\beta_1,\beta_2)=\frac{(m+1)n+\beta_2+1}{(m+1)n+\beta_1-1}
=1+\frac{2+\beta_2-\beta_1}{(m+1)n+\beta_1-1}.
\end{aligned}
\end{equation}
By a straightforward algebraic computation, we find that
\begin{equation}\label{imply}
p\leq\tilde{p}\ \text{and} \ q\leq\tilde{q}\implies \Gamma_m(n,p,q,\beta_1,\beta_2)\geq0,
\end{equation}
Hence the necessary conditions for $ \Gamma_m(n,p,q,\beta_1,\beta_2)<0$
are as follows
\begin{align}
&p>\tilde{p}\ \ \text{and} \ \ q>\tilde{q};\label{glocase1}\\
\text{or,}\ \ &p\leq\tilde{p}\ \ \text{and} \ \ q>\tilde{q};\label{glocase2}\\
\text{or,}\ \ &p>\tilde{p}\ \ \text{and} \ \ q\leq\tilde{q}\label{glocase3}.
\end{align}
In other words, only the above three cases (\ref{glocase1})-(\ref{glocase3}) can result in $\Gamma_m(n,p,q,\beta_1,\beta_2)<0$.
Next, we categorize and present the global existence  results  according to  the different ranges of $\sigma$.

For $(u_0,u_1,v_0,v_1)\in D^\sigma\times D^\sigma$ with  $\sigma\geq1$, we have the three global existence results.

\begin{theorem}{\bf{\text{(Global existence for $p>\tilde{p}$ and $q>\tilde{q}$ (I) ).}}}\label{theorem11}
Let  $n\geq1, \sigma\geq1, m>-1$, $\mu_1,$ $\mu_2>1, \nu_1^2,\nu_2^2\geq0$ such that both $\delta_1$ and $\delta_2\geq (m+1)^2(n+2\sigma-1)^2$. Let $p,q>1$ with $p,q>\lceil \sigma\rceil$, and assume that if $n>2\sigma$, it holds that $p,q\leq 1+\frac{2}{n-2\sigma}$. For
$p>\tilde{p}$ and $q>\tilde{q}$, there exists a small constant  $\varepsilon>0$ such that for any initial data $(u_0,u_1,v_0,v_1)\in D^\sigma\times D^\sigma$ satisfying  the support condition (\ref{compact1}),  if $\Vert(u_0,u_1)\Vert_{D^\sigma}+\Vert(v_0,v_1)\Vert_{D^\sigma}\leq\varepsilon$, then  (\ref{eqs}) has a  unique global solution
\begin{equation*}
(u,v)\in \Big(C\bigl([1,\infty); H^\sigma(\mathbb{R}^n)\bigr) \cap C^1\bigl([1,\infty); H^{\sigma-1}(\mathbb{R}^n)\bigr)\Big)^2,
\end{equation*}
and  $(u,v)$ satisfies the support property
\begin{equation}\label{supportedcondition}
\text{supp} (u,v)(t,\cdot)\subset B_{\phi_m(t)-\phi_m(1)+R}, \ \  \text{for any}\ t\in(1,\infty),
\end{equation}
where $\phi_m(t)$ is defined by (\ref{compact2}) and  $\tilde{p}, \tilde{q}$ are defined by (\ref{tildepq}). Furthermore, the solution ($u,v$) satisfies the decay estimates
\begin{align}
&\Vert u(t,\cdot)\Vert_{L^2}\lesssim t^{-(m+1)\frac{n}{2}+\frac{\sqrt{\delta_1}-\mu_1+1}{2}}\big(\Vert(u_0,u_1)\Vert_{D^\sigma}
+\Vert(v_0,v_1)\Vert_{D^\sigma}\bigr),\label{decay1}\\
&\Vert u(t,\cdot)\Vert_{\dot{H}^\sigma}\lesssim t^{-(m+1)(\sigma+\frac{n}{2})+\frac{\sqrt{\delta_1}-\mu_1+1}{2}}\ell_1(t)\big(\Vert(u_0,u_1)\Vert_{D^\sigma}
+\Vert(v_0,v_1)\Vert_{D^\sigma}\bigr),\label{decay2}\\
&\Vert \partial_tu(t,\cdot)\Vert_{L^2}\lesssim t^{m-(m+1)(1+\frac{n}{2})+\frac{\sqrt{\delta_1}-\mu_1+1}{2}}\big(\Vert(u_0,u_1)\Vert_{D^\sigma}
+\Vert(v_0,v_1)\Vert_{D^\sigma}\bigr),\label{decay3}\\
&\Vert \partial_tu(t,\cdot)\Vert_{\dot{H}^{\sigma-1}}\lesssim t^{m-(m+1)(\sigma+\frac{n}{2})+\frac{\sqrt{\delta_1}-\mu_1+1}{2}}\ell_1(t)\big(\Vert(u_0,u_1)\Vert_{D^\sigma}
+\Vert(v_0,v_1)\Vert_{D^\sigma}\bigr),\label{decay4}\\
&\Vert v(t,\cdot)\Vert_{L^2}\lesssim t^{-(m+1)\frac{n}{2}+\frac{\sqrt{\delta_2}-\mu_2+1}{2}}\big(\Vert(u_0,u_1)\Vert_{D^\sigma}
+\Vert(v_0,v_1)\Vert_{D^\sigma}\bigr),\label{decay5}\\
&\Vert v(t,\cdot)\Vert_{\dot{H}^\sigma}\lesssim t^{-(m+1)(\sigma+\frac{n}{2})+\frac{\sqrt{\delta_2}-\mu_2+1}{2}}\ell_2(t)\big(\Vert(u_0,u_1)\Vert_{D^\sigma}
+\Vert(v_0,v_1)\Vert_{D^\sigma}\bigr),\label{decay6}\\
&\Vert \partial_tv(t,\cdot)\Vert_{L^2}\lesssim t^{m-(m+1)(1+\frac{n}{2})+\frac{\sqrt{\delta_2}-\mu_2+1}{2}}\big(\Vert(u_0,u_1)\Vert_{D^\sigma}
+\Vert(v_0,v_1)\Vert_{D^\sigma}\bigr),\label{decay7}\\
&\Vert \partial_tv(t,\cdot)\Vert_{\dot{H}^{\sigma-1}}\lesssim t^{m-(m+1)(\sigma+\frac{n}{2})+\frac{\sqrt{\delta_2}-\mu_2+1}{2}}\ell_2(t)\big(\Vert(u_0,u_1)\Vert_{D^\sigma}
+\Vert(v_0,v_1)\Vert_{D^\sigma}\bigr),\label{decay8}
\end{align}
where
\begin{equation}\label{l(t)}
\ell_i(t)=\left\{
\begin{aligned}
&1,\ &\text{if}\ \delta_i>\bigl((m+1)(n+2\sigma-1)\bigr)^2,\\
&\bigl(1+\log{t}\bigr)^{\frac{1}{2}},\ &\text{if}\ \delta_i=\bigl((m+1)(n+2\sigma-1)\bigr)^2,
\end{aligned}
\right.
\ \ i=1,2.
\end{equation}
\end{theorem}
\begin{remark}
The assumptions $\mu_1,\mu_2>1$ are technical conditions imposed to guarantee that the denominators $(m+1)n+\beta_1-1$ and $(m+1)n+\beta_2-1$ in the definitions of $\tilde{p}$ and $\tilde{q}$ remain positive, thereby ensuring the validity of the integral exponents (\ref{integralexponent1}) and (\ref{integralexponent2}). It should be emphasized that since our analysis focuses on the case where $\delta_1,\delta_2$  are relatively large, imposing $\mu_1,\mu_2>1$ is naturally  reasonable.
\end{remark}
We need to provide an example to demonstrate the validity of the conditions in Theorem \ref{theorem11}.
\begin{example}\label{re2.2}
For $\sigma=1.1, m=0.4,n=3>2\sigma,\mu_1=15,\nu_1=6,\mu_2=8.7,\nu_2=0.25$, we calculate that $\delta_1=52,\delta_2=59.04>34.5744=(m+1)^2(n+2\sigma-1)^2,\beta_1=\frac{16-\sqrt{52}}{2},\beta_2=
\frac{9.7-\sqrt{59.04}}{2}, \tilde{p}=\frac{5.2+\frac{16-\sqrt{52}}{2}}{3.2+\frac{9.7-\sqrt{59.04}}{2}}\approx2.28,
\tilde{q}=\frac{5.2+\frac{9.7-\sqrt{59.04}}{2}}{3.2+\frac{16-\sqrt{52}}{2}}\approx0.82,
1+\frac{2}{n-2\sigma}=3.5$. Then for any $2=\lceil1.1\rceil<\frac{5.2+\frac{16-\sqrt{52}}{2}}{3.2+\frac{9.7-\sqrt{59.04}}{2}}<p\leq 3.5, 2=\lceil1.1\rceil< q\leq3.5$, (\ref{eqs}) has a unique global solution.
\end{example}
\begin{remark}\label{decayexplation}
The decay indices of the solution to (\ref{eqs}) shown in Theorem \ref{theorem11} are actually the same as that of the solution to the corresponding linear equation (\ref{lineareq}), which can be referred to in propositions \ref{linearestimate}-\ref{unlinearestimate}.
\end{remark}

\begin{theorem}{\bf{\text{(Global existence for $p\leq\tilde{p}$ and $q>\tilde{q}$  (I)) .}}}\label{theorem12}
Let  $n\geq1, \sigma\geq1, m>-1,$ $\mu_1,$ $\mu_2>1, \nu_1^2,\nu_2^2\geq0$ such that $\delta_1\geq (m+1)^2(n+2\sigma-1)^2, \delta_2>(m+1)^2(n+2\sigma-1)^2$. Let $p,q>1$ with $p,q>\lceil \sigma\rceil$,  and assume that  if $n>2\sigma$, $p,q\leq 1+\frac{2}{n-2\sigma}$.  For
$p\leq\tilde{p}$ and $q>\tilde{q}$ satisfying
\begin{equation}\label{th4con}
\frac{q+1}{pq-1}<\frac{(m+1)n+\beta_2-1}{2},
\end{equation}
there exists a small constant $\varepsilon>0$ such that for any  initial data $(u_0,u_1,v_0,v_1)\in D^\sigma\times D^\sigma$ satisfying  the support condition (\ref{compact1}), and if $\Vert(u_0,u_1)\Vert_{D^\sigma}+\Vert(v_0,v_1)\Vert_{D^\sigma}\leq\varepsilon$,
then (\ref{eqs}) admits  a unique global solution $$(u,v)\in \Big(C\bigl([1,\infty); H^\sigma(\mathbb{R}^n)\bigr)\cap C^1\bigl([1,\infty);  H^{\sigma-1}(\mathbb{R}^n)\bigr)\Big)^2,$$ and ($u,v$) satisfies the support property (\ref{supportedcondition}), where $\tilde{p}, \tilde{q}$ are defined by (\ref{tildepq}). Moreover, the solution $v$ satisfies the decay estimates (\ref{decay5})- (\ref{decay8}), and $u$ satisfies the estimates
\begin{align}
&\Vert u(t,\cdot)\Vert_{L^2}\lesssim t^{-(m+1)\frac{n}{2}+\frac{\sqrt{\delta_1}-\mu_1+1}{2}+\alpha_1}\big(\Vert(u_0,u_1)\Vert_{D^\sigma}
+\Vert(v_0,v_1)\Vert_{D^\sigma}\bigr),\label{decay9}\\
&\Vert u(t,\cdot)\Vert_{\dot{H}^\sigma}\lesssim t^{-(m+1)(\sigma+\frac{n}{2})+\frac{\sqrt{\delta_1}-\mu_1+1}{2}+\alpha_1}\ell_1(t)\big(\Vert(u_0,u_1)\Vert_{D^\sigma}
+\Vert(v_0,v_1)\Vert_{D^\sigma}\bigr),\label{decay10}\\
&\Vert \partial_tu(t,\cdot)\Vert_{L^2}\lesssim t^{m-(m+1)(1+\frac{n}{2})+\frac{\sqrt{\delta_1}-\mu_1+1}{2}+\alpha_1}\big(\Vert(u_0,u_1)\Vert_{D^\sigma}
+\Vert(v_0,v_1)\Vert_{D^\sigma}\bigr),\notag\\
&\Vert \partial_tu(t,\cdot)\Vert_{\dot{H}^{\sigma-1}}\lesssim t^{m-(m+1)(\sigma+\frac{n}{2})+\frac{\sqrt{\delta_1}-\mu_1+1}{2}+\alpha_1}\ell_1(t)\big(\Vert(u_0,u_1)\Vert_{D^\sigma}
+\Vert(v_0,v_1)\Vert_{D^\sigma}\bigr),\notag
\end{align}
where $\ell_1(t)$ is defined by (\ref{l(t)}) and
 \begin{equation}\label{gamma1}
\alpha_1=\left\{
\begin{aligned}
&\big((m+1)n+\beta_2-1\big)(\tilde{p}-p), &\quad& \text{if}\ \ p<\tilde{p}, \\
&\epsilon, &\quad&\text{if}\ \ p=\tilde{p}
\end{aligned}
\right.
\end{equation}
with $\epsilon> 0$ being an sufficiently  small constant.
\end{theorem}
\begin{remark}
In the proof of Theorem \ref{theorem12} shown in Section \ref{proofthm12}, we observe  that when $p\leq\tilde{p}$, the integral corresponding to index (\ref{nonintegralindex}), i.e., $\int_1^t\tau^{(m+1)n+\beta_1+(-(m+1)n-\beta_2+1)p}d\tau$,  cannot be controlled by a convergent improper integral.  Therefore, it is necessary to rely on the technical condition (\ref{th4con}) to ensure the integrability of another improper integral (\ref{ingreinfty}). However, the cost of this approach is the emergence of a non-decaying factor $t^{\alpha_1}$ in the estimate of $u(t,\cdot)$.
\end{remark}
It is still necessary to provide an example to illustrate the reasonableness of  conditions stated in Theorem \ref{theorem12}.
\begin{example}\label{re2.3}
For $\sigma=1.4$, $m=0.3, n=4>2\sigma, \mu_1=20, \mu_2=9, \nu_1=8.5, \nu_2=1$, we calculate that $\delta_1=72, \delta_2=60>(m+1)^2(n+2\sigma-1)^2=56.8516, \beta_1=\frac{21-\sqrt{72}}{2}, \beta_2=\frac{10-\sqrt{60}}{2},\tilde{p}=\frac{6.2+\frac{21-\sqrt{72}}{2}}
{4.2+\frac{10-\sqrt{60}}{2}}\approx 2.339,\tilde{q}=\frac{6.2+\frac{10-\sqrt{60}}{2}}{4.2+\frac{21-\sqrt{72}}{2}}\approx 0.7007$. Choosing  $p=2.1<\tilde{p},q=2.2>\tilde{q},$  we have $p,q>\lceil 1.4\rceil=2$, $p,q<1+\frac{2}{n-2\sigma}\approx2.667$ and  $\frac{q+1}{pq-1}=\frac{3.2}{3.62}\approx 0.884<\frac{(m+1)n+\beta_2-1}{2}\approx2.6635$. Hence  the solution ($u,v$) to (\ref{eqs}) exists globally for $p=2.1,q=2.2$.
\end{example}

\begin{theorem}{\bf{\text{(Global  existence for $p>\tilde{p}$ and $q\leq\tilde{q}$ (I) ).}}}\label{theorem13}
Let  $n\geq1, \sigma\geq1, m>-1, \mu_1,$ $\mu_2>1, \nu_1,\nu_2\geq0$ such that $\delta_1> (m+1)^2(n+2\sigma-1)^2,\delta_2\geq(m+1)^2(n+2\sigma-1)^2$. Let $p,q>1$ with $p,q>\lceil \sigma\rceil$,  and assume that  if $n>2\sigma$, $p,q\leq 1+\frac{2}{n-2\sigma}$. For
$p>\tilde{p}$ and $q\leq\tilde{q}$ satisfying
\begin{equation}\label{th4con}
\frac{p+1}{pq-1}<\frac{(m+1)n+\beta_1-1}{2},
\end{equation}
there exists a small constant $\varepsilon>0$ such that for any initial data  $(u_0,u_1,v_0,v_1)\in D^\sigma\times D^\sigma$ satisfying  the support condition (\ref{compact1}), and  if $\Vert(u_0,u_1)\Vert_{D^\sigma}+\Vert(v_0,v_1)\Vert_{D^\sigma}\leq\varepsilon$,
then  (\ref{eqs})  admits a unique global solution $$(u,v)\in \Big(C\bigl([1,\infty); H^\sigma(\mathbb{R}^n)\bigr) \cap C^1\bigl([1,\infty); H^{\sigma-1}(\mathbb{R}^n)\bigr)\Big)^2, $$ and ($u,v$) satisfies the support property (\ref{supportedcondition}), where $\tilde{p}, \tilde{q}$ are defined by (\ref{tildepq}). Moreover, the solution $u$  satisfies the decay estimates (\ref{decay1})- (\ref{decay4}),  and $v$ satisfies the estimates
\begin{align}
&\Vert v(t,\cdot)\Vert_{L^2}\lesssim t^{-(m+1)\frac{n}{2}+\frac{\sqrt{\delta_2}-\mu_2+1}{2}+\alpha_2}\big(\Vert(u_0,u_1)\Vert_{D^\sigma}
+\Vert(v_0,v_1)\Vert_{D^\sigma}\bigr)\label{decay11}\\
&\Vert v(t,\cdot)\Vert_{\dot{H}^\sigma}\lesssim t^{-(m+1)(\sigma+\frac{n}{2})+\frac{\sqrt{\delta_2}-\mu_2+1}{2}+\alpha_2}\ell_2(t)\big(\Vert(u_0,u_1)\Vert_{D^\sigma}
+\Vert(v_0,v_1)\Vert_{D^\sigma}\bigr),\label{decay12}\\
&\Vert \partial_tv(t,\cdot)\Vert_{L^2}\lesssim t^{m-(m+1)(1+\frac{n}{2})+\frac{\sqrt{\delta_2}-\mu_2+1}{2}+\alpha_2}\big(\Vert(u_0,u_1)\Vert_{D^\sigma}
+\Vert(v_0,v_1)\Vert_{D^\sigma}\bigr),\notag\\
&\Vert \partial_tv(t,\cdot)\Vert_{\dot{H}^{\sigma-1}}\lesssim t^{m-(m+1)(\sigma+\frac{n}{2})+\frac{\sqrt{\delta_2}-\mu_2+1}{2}+\alpha_2}\ell_2(t)\big(\Vert(u_0,u_1)\Vert_{D^\sigma}
+\Vert(v_0,v_1)\Vert_{D^\sigma}\bigr),\notag
\end{align}
where $\ell_2(t)$ is defined by (\ref{l(t)}) and
\begin{equation}\label{gamma2}
\alpha_2=\left\{
\begin{aligned}
&\big((m+1)n+\beta_1-1\big)(\tilde{q}-q), &\quad& \text{if}\ \ q<\tilde{q}, \\
&\epsilon, &\quad&\text{if}\ \ q=\tilde{q}
\end{aligned}
\right.
\end{equation}
with $\epsilon> 0$ being an sufficiently  small constant.
\end{theorem}
\begin{remark}

By swapping the positions of $p$ and $q$  in Theorem \ref{theorem12}, we can directly derive Theorem \ref{theorem13}. Therefore, we can also refer to Theorem \ref{theorem13} as being dual to Theorem \ref{theorem12}. Furthermore, it can be  straightforward to provide examples to illustrate the reasonableness of the various conditions outlined in Theorem \ref{theorem13}, and we will not repeat them here.
\end{remark}

For  $(u_0,u_1,v_0,v_1)\in D^\sigma\times D^\sigma$ with  $0<\sigma<1$, we also establish  three global existence results.

\begin{theorem}{\bf{\text{(Global existence for $p>\tilde{p}$ and $q>\tilde{q}$ (II) ).}}}\label{theorem21} Let $n\geq1, \sigma\in(0,1)$, $m>-1$, $\mu_1, \mu_2>1, \nu_1^2, \nu_2^2\geq0$ such that both $\delta_1$ and $\delta_2\geq(m+1)^2(n+2\sigma-1)^2$. Let $p,q>1$, and when $n>2\sigma$, we assume $p,q\leq\frac{n}{n-2\sigma}$.  For $p>\tilde{p}$ and $q>\tilde{q}$,  there exists a small constant $\varepsilon>0$ such that for any  initial data $(u_0,u_1,v_0,v_1)\in D^\sigma\times D^\sigma$ satisfying  the support condition (\ref{compact1}), and if $\Vert(u_0,u_1)\Vert_{D^\sigma}+\Vert(v_0,v_1)\Vert_{D^\sigma}\leq\varepsilon$, then (\ref{eqs}) admits a unique global solution $(u,v)\in \Big(C\bigl([1,\infty); H^\sigma(\mathbb{R}^n)\bigr)\Big)^2$ and  ($u,v$) satisfies the support property (\ref{supportedcondition}). Furthermore, the solution ($u,v$) satisfies the decay estimates (\ref{decay1}), (\ref{decay2}), (\ref{decay5}) and (\ref{decay6}).
\end{theorem}

\begin{theorem}{\bf{\text{(Global existence for $p\leq\tilde{p}$ and $q>\tilde{q}$ (II) ).}}}\label{theorem22}
Let $n\geq1, \sigma\in(0,1)$, $m>-1$, $\mu_1,\mu_2>1$, $\nu_1^2,\nu_2^2\geq0$, $m>-1$ such that $\delta_1\geq(m+1)^2(n+2\sigma-1)^2$ and $\delta_2>(m+1)^2(n+2\sigma-1)^2$. Let $p,q>1$,  and when $n>2\sigma$, we assume $p,q\leq\frac{n}{n-2\sigma}$. For $p\leq\tilde{p}$ and $q>\tilde{q}$ satisfying
\begin{equation*}
\frac{q+1}{pq-1}<\frac{(m+1)n+\beta_2-1}{2},
\end{equation*}
there exists a small constant $\varepsilon>0$ such that for any initial data  $(u_0,u_1,v_0,v_1)\in D^\sigma\times D^\sigma$ satisfying  the support condition (\ref{compact1}),  and if $\Vert(u_0,u_1)\Vert_{D^\sigma}+\Vert(v_0,v_1)\Vert_{D^\sigma}\leq\varepsilon$, then
(\ref{eqs}) admits  a unique global solution $(u,v)\in \Big(C\bigl([1,\infty); H^\sigma(\mathbb{R}^n)\bigr)\Big)^2$  and ($u,v$) satisfies the support property (\ref{supportedcondition}).  Furthermore, the solution $u$  satisfies the estimates (\ref{decay9}), (\ref{decay10}) and $v$ satisfies the  decay  estimates (\ref{decay5}), (\ref{decay6}).
\end{theorem}
\begin{theorem}{\bf{\text{(Global  existence for $p>\tilde{p}$ and $q\leq\tilde{q}$ (II)).}}}\label{theorem23}
Let $n\geq1, \sigma\in(0,1)$, $m>-1$, $\mu_1,\mu_2>1$, $\nu_1^2,\nu_2^2\geq0$, $m>-1$ such that $\delta_1>(m+1)^2(n+2\sigma-1)^2$ and $\delta_2\geq(m+1)^2(n+2\sigma-1)^2$. Let $p,q>1$,  and when $n>2\sigma$, we assume $p,q\leq\frac{n}{n-2\sigma}$. For $p>\tilde{p}$ and $q\leq\tilde{q}$ satisfying
\begin{equation*}
\frac{p+1}{pq-1}<\frac{(m+1)n+\beta_1-1}{2},
\end{equation*}
there exists a small constant  $\varepsilon>0$ such that for any initial data  $(u_0,u_1,v_0,v_1)\in D^\sigma\times D^\sigma$ satisfying  the support condition (\ref{compact1}),  and if $\Vert(u_0,u_1)\Vert_{D^\sigma}+\Vert(v_0,v_1)\Vert_{D^\sigma}\leq\varepsilon$,
then (\ref{eqs}) admits  a unique global solution $(u,v)\in \Big(C\bigl([1,\infty); H^\sigma(\mathbb{R}^n)\bigr)\Big)^2$ and ($u,v$) satisfies the support property (\ref{supportedcondition}). Furthermore, the solution $u$ satisfies the decay estimates (\ref{decay1}), (\ref{decay2}) and   $v$ satisfies the estimates (\ref{decay11}), (\ref{decay12}).
\end{theorem}

\begin{remark}
Regarding the reasonableness of the conditions in Theorems \ref{theorem21}-\ref{theorem23}, some examples similar to Examples \ref{re2.2}-\ref{re2.3} can be provided; however, we will not elaborate on them here.
\end{remark}
\begin{remark}
The regularity of   initial data   considered in Theorems \ref{theorem21}-\ref{theorem23} is lower than that in Theorems \ref{theorem11}-\ref{theorem13}. Inspired by \cite{PaRe2018}, we consider different function spaces for for the two cases of high and low regularity of the initial data,  as shown in (\ref{normxt}) and (\ref{def51functionspace}), respectively. This is precisely why the requirements for $p,q$ in Theorems \ref{theorem21}-\ref{theorem23} differ from those in  Theorems \ref{theorem11}-\ref{theorem13}. In terms of the proof process, the two cases are similar, therefore, we provide a relatively concise argument for the proofs of Theorems \ref{theorem21}-\ref{theorem23} in Section \ref{section5}.
\end{remark}


In the following remark, we emphasize that the results for the weakly  coupled system  (\ref{eqs}) are not just a simple generalization  of those for the corresponding single equation  (\ref{single2}).

\begin{remark}\label{noeasy}
We summarize the results of Theorems \ref{theorem1}- \ref{theorem23} in the following relationship diagram
\begin{equation*}\label{pic}
\resizebox{0.85\hsize}{!}{$
\begin{aligned}
   &\bullet p\leq\tilde{p}\
   \text{and}\ \ q\leq\tilde{q}\ \implies\ \ \Gamma_m(n,p,q,\beta_1,\beta_2)\geq0\implies \text{Blow-up result},\\
    &\bullet \Gamma_m(n,p,q,\beta_1,\beta_2)<0\implies\left.
   \begin{array}{l}
    p>\tilde{p}\ \text{and} \ q>\tilde{q},\ \text{or}  \\
    p\leq\tilde{p}\  \text{and} \ q>\tilde{q},\ \text{or}\\
    p>\tilde{p}\ \text{and} \ q\leq\tilde{q}
   \end{array}
  \right\}\small\overset{\text{\text{large} $\delta_{1,2}$ }}\implies\text{Global existence.}
\end{aligned}$}
\end{equation*}
Therefore, we can conclude that when the damping terms dominate over the mass terms, the critical curve of  (\ref{eqs}) is described by the curve
\begin{equation*}
\Gamma_m(n,p,q,\beta_1,\beta_2)=0
\end{equation*}
in the $p-q$ plane.
Moreover,  from this relationship diagram we see $\Gamma_m(n,p,q,\beta_1,\beta_2)\geq0$  can be equivalently expressed as
$p\leq\tilde{p}, q\leq\tilde{q}$, while $\Gamma_m(n,p,q,\beta_1,\beta_2)<0$  can be alternatively expressed as: $p>\tilde{p}, q>\tilde{q},$ or
$p\leq\tilde{p}, q>\tilde{q},$ or $p>\tilde{p}, q\leq\tilde{q}$.

Next, we focus on analysing the blow-up result, i.e., Theorem \ref{theorem1}, to illustrate that the main results (Theorems \ref{theorem1}-\ref{theorem23}) are not merely a straightforward extension of those for the single equation (\ref{single2}).
For convenience, denote
\begin{equation*}
p_F^1:=p_F\big((m+1)n+\frac{\mu_1-1-\sqrt{\delta_1}}{2}\big), \ \ p_F^2:=p_F\big((m+1)n+\frac{\mu_2-1-\sqrt{\delta_2}}{2}\big).
\end{equation*}
From 
 (\ref{tildepq}) 
it is clear that $\tilde{p},\tilde{q}$  can be regarded as  "perturbations" of the Fujita index. Indeed, we see
$\tilde{p}=p_F^2, \ \tilde{q}=p_F^1$
hold when $\beta_1=\beta_2$.
By the analysis in Section \ref{s1}, we know when $p\leq p_F((m+1)n+\frac{\mu-1-\sqrt{\delta}}{2})$, the solution to (\ref{single2}) will blow up.  Consequently,  if the result  was directly extended, the blow-up region for (\ref{eqs}) would be described by $p\leq p_F^1,q\leq p_F^2$ (or $p\leq p_F^2,q\leq p_F^1$),  however, the blow-up condition $\Gamma_m(n,p,q,\beta_1,\beta_2)\geq0$ does not coincide with  this,  except in the special case where $\beta_1=\beta_2$.
We present the blow-up region of (\ref{eqs}) in the form of a diagram (Diagram 1).
\begin{figure}[htbp]
\begingroup
\renewcommand{\figurename}{}   
\renewcommand{\thefigure}{}    
\renewcommand{\figurename}{\relax}  
\captionsetup{labelsep=none}
\setlength{\abovecaptionskip}{0pt}
\setlength{\belowcaptionskip}{0pt}
    \centering
\begin{tikzpicture}\label{blowupregion}
    \begin{axis}[
        axis lines=middle, 
        xlabel={$p$}, ylabel={$q$}, 
        xmin=0, xmax=4.5, 
        ymin=0, ymax=4.5, 
        xtick=\empty, ytick=\empty, 
        width=10cm, height=10cm, 
        axis line style={->}, 
        enlargelimits=false, 
        clip=false 
    ]
        \draw[black, fill=black] (axis cs:0, 0) circle (2pt);
        \node[below left] at (axis cs:0, 0) {\textcolor{black}{0}};

        \draw[fill=black] (axis cs:1,0) circle (0.5pt);
        \draw[black] (axis cs:1, 0) circle (2pt); 
        \node[below] at (axis cs:1, 0) {\textcolor{black}{$\max\{1,\frac{1+\beta_1}{2m+1+\beta_2}\}$}};
        \draw[black, fill=black] (axis cs:3, 0) circle (2pt); 
        \node[below] at (axis cs:3, 0) {\textcolor{black}{$\tilde{p}=p_F^2$}};

                \draw[fill=black] (axis cs:0,1) circle (0.5pt);
        \draw[black,] (axis cs:0, 1) circle (2pt); 
        \node[left] at (axis cs:0, 1) {\textcolor{black}{$\max\{1,\frac{1+\beta_2}{2m+1+\beta_1}\}$}};
        \draw[black, fill=black] (axis cs:0, 3) circle (2pt); 
        \node[left] at (axis cs:0, 3) {\textcolor{black}{$\tilde{q}=p_F^1$}};

        \addplot[pattern=north east lines, pattern color=black, draw=none] coordinates {
            (1,1) (3,1) (3,3) (1,3) (1,1)
        };

        \draw[black,dashed] (axis cs:1,1) -- (axis cs:3,1);
        \draw[black, thick] (axis cs:3,1) -- (axis cs:3,3);
        \draw[black, thick] (axis cs:3,3) -- (axis cs:1,3);
        \draw[black,  dashed] (axis cs:1,3) -- (axis cs:1,1);

        \draw[dashed, gray] (axis cs:1,1) -- (axis cs:1,0);
        \draw[dashed, gray] (axis cs:1,1) -- (axis cs:0,1);
        \draw[dashed, gray] (axis cs:3,1) -- (axis cs:3,0);
        \draw[dashed, gray] (axis cs:1,3) -- (axis cs:0,3);

        \draw[->, thick, black] (axis cs:2.55,2.75) -- (axis cs:4,2.75) node[anchor=west] {\large \textcolor{black}{Case1: $\beta_1 = \beta_2$}};

        \draw[blue, fill=blue] (axis cs:3.5, 0) circle (2pt); 
        \node[below] at (axis cs:3.5, 0) {\textcolor{blue}{$\tilde{p}$}};

        \draw[blue, fill=blue] (axis cs:0, 2.25) circle (2pt); 
        \node[left] at (axis cs:0, 2.25) {\textcolor{blue}{$\tilde{q}$}};

        \addplot[pattern=north west lines, pattern color=blue, draw=none] coordinates {
            (1,1) (3.5,1) (3.5,2.25) (1,2.25) (1,1)
        };

        \draw[blue, dashed] (axis cs:3,1) -- (axis cs:3.5,1);
        \draw[blue, thick] (axis cs:3.5,1) -- (axis cs:3.5,2.25);
        \draw[blue, thick] (axis cs:3.5,2.25) -- (axis cs:1,2.25);

        \draw[dashed, blue] (axis cs:3.5,2.25) -- (axis cs:3.5,0);
        \draw[dashed, blue] (axis cs:3.5,2.25) -- (axis cs:0,2.25);

        \draw[->, thick, blue] (axis cs:3.2,2) -- (axis cs:4,2) node[anchor=west] {\large \textcolor{blue}{Case2: $\beta_1 > \beta_2$}};

        \draw[green!70!black, fill=green!70!black] (axis cs:2.25, 0) circle (2pt); 
        \node[below] at (axis cs:2.25, 0) {\textcolor{green!70!black}{$\tilde{p}$}};

        \draw[green!70!black, fill=green!70!black] (axis cs:0, 3.5) circle (2pt); 
        \node[left] at (axis cs:0, 3.5) {\textcolor{green!70!black}{$\tilde{q}$}};

        \addplot[pattern=north east lines, pattern color=green!70!black, draw=none] coordinates {
            (1,1) (2.25,1) (2.25,3.5) (1,3.5) (1,1)
        };

        \draw[green!70!black,dashed] (axis cs:1,3) -- (axis cs:1,3.5);
        \draw[green!70!black, thick] (axis cs:1,3.5) -- (axis cs:2.25,3.5);
        \draw[green!70!black, thick] (axis cs:2.25,3.5) -- (axis cs:2.25,1);

        \draw[dashed, green!70!black] (axis cs:2.25,3.5) -- (axis cs:2.25,0);
        \draw[dashed, green!70!black] (axis cs:2.25,3.5) -- (axis cs:0,3.5);

        \draw[->, thick, green!70!black] (axis cs:1.75,3.25) -- (axis cs:4,3.25) node[anchor=west] {\large \textcolor{green!70!black}{Case3: $\beta_1 < \beta_2$}};

    \end{axis}
\end{tikzpicture}
 \caption{Diagram 1}
    \label{fig:blowupregion}
    \endgroup
\end{figure}
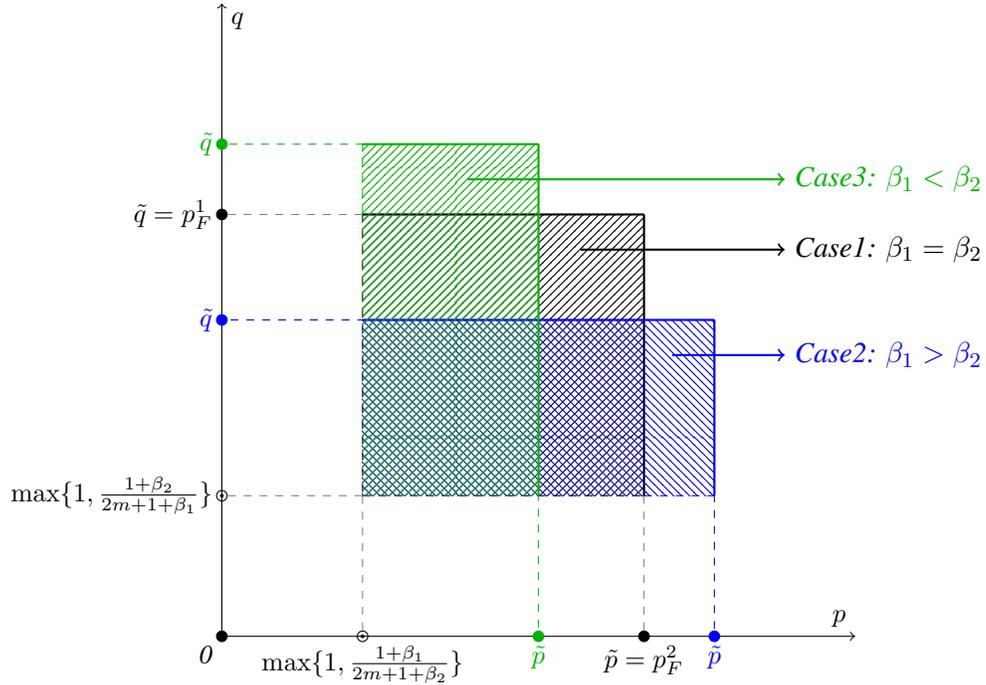

Diagram 1 clearly shows that the blow-up result for (\ref{eqs}) can be regarded as a direct extension of its single equation (\ref{single2}) only when $\beta_1=\beta_2$. In contrast, if $\beta_1\neq\beta_2$, the blow-up region of (\ref{eqs}) undergoes a "shift,"  which more prominently reflects the complex interplay of mutual constraints and balance between the two damping terms and the two mass terms.

\end{remark}

\section{Blow-up result}\label{section3blowup}
The main tool for proving Theorem \ref{theorem1} is the test function. Referring to the methods in \cite{ChenPa2019,PaRe2017}, we find that the presence of $m$ in Gellerstedt operator $\partial_t^2-t^{2m}\Delta$ compels us to seek new test functions. As will be seen in the following proof, the test functions  we use are   closely related to $m$, highlighting the differences between the Gellerstedt operator and the standard  wave operator $\partial_t^2-\Delta$.

Let us assume by contradiction that  $(u,v)$ is a global energy solution to (\ref{eqs}), that is $T=\infty$ in (\ref{solutiondefu}). Note that $h_i(t)=t^{\beta_i},\beta_i=\frac{\mu_i+1-\sqrt{\delta_i}}{2}, i=1,2$ satisfy
\begin{equation}\label{eqhi}
h_i^{\prime\prime}(t)-\frac{\mu_i}{t}h_i^{\prime}(t)+\frac{\mu_i+\nu_i^2}{t^2}h_i(t)=0,\ i=1,2.
\end{equation}
Multiplying both sides of the first equation in (\ref{eqs}) by $h_1(t)$ and the second by $h_2(t)$,  we have the following equations in divergence form
\begin{equation}\label{divergence}
  \left\{
\begin{aligned}
&\partial_{t}^2(h_1u)-\Delta(t^{2m}h_1u)+\partial_t(\frac{\mu_1}{t}h_1u-2h_1^{\prime}u)=h_1\vert v\vert^p,\\
&\partial_{t}^2(h_2v)-\Delta(t^{2m}h_2v)+\partial_t(\frac{\mu_2}{t}h_2v-2h_2^{\prime}v)=h_2\vert u\vert^q.
\end{aligned}
 \right.
  \end{equation}
Choose a non-increasing function $\lambda(t) \in C_0^\infty\big([0,\infty)\big)$, a radially symmetric function $\phi(x) \in C_0^\infty(\mathbb{R}^n)$ such that $\phi$ is non-increasing with respect to $|x|$, and satisfy
\begin{equation}\label{lambdaphi}
\begin{aligned}
&\lambda(t)=1 \ \text{on} \ [0, \frac{1}{2}],\  \text{supp} \lambda(t)\subset [0,1];\\
&\phi(x)=1 \ \text{on} \ B_{\frac{1}{2}}, \ \text{supp}  \phi(x)\subset B_1;\\
&\vert\lambda^{\prime}(t)\vert,\ \vert\lambda^{\prime\prime}(t)\vert\lesssim\lambda(t)^{\frac{1}{r}},\ \ \ \vert\Delta\phi\vert\lesssim\phi^{\frac{1}{r}}\ \ \text{for any}\ r>1.
\end{aligned}
\end{equation}
Regarding the existences of $\lambda$ and $\phi$,  one can  refer to \cite{PaRe2017}.
For $d, R\geq1,$ let $\psi_{d, R}(t,x)=\lambda\bigl(\frac{t-1}{d}\bigr)\phi\bigl(\frac{x}{R^{m+1}}\bigr)$.

Substituting  $\varPhi_1(t,x)=h_1(t)\psi_{d, R}(t,x)$ into (\ref{def3}) and using  (\ref{eqhi}), we have
\begin{equation*}
\begin{aligned}
I_{d,R}:&=\int_1^{d+1}\int_{B_{R^{m+1}}}h_1(t)\psi_{d, R}(t,x)\vert v(t,x)\vert^pdxdt\\
        &=-\underbrace{\int_{B_{R^{m+1}}}\bigl(\frac{\mu_1-1+\sqrt{\delta_1}}{2}u_0(x)+u_1(x)\bigr)\phi(\frac{x}{R^{m+1}})dx}_{K_0}\\
        &\quad+\underbrace{\int_{\frac{d}{2}+1}^{d+1}\int_{B_{R^{m+1}}}u(t,x)h_1(t)\partial_t^2\psi_{d, R}(t,x)dxdt}_{K_1}\\
        &\quad+\underbrace{\int_{\frac{d}{2}+1}^{d+1}\int_{B_{R^{m+1}}}u(t,x)\bigl(2h_1^{\prime}(t)-\frac{\mu_1}{t}h_1(t)\bigr)\partial_t\psi_{d, R}(t,x)dxdt}_{K_2}\\
        &\quad-\underbrace{\int_1^{d+1}\int_{B_{R^{m+1}}\textbackslash B_{\frac{R^{m+1}}{2}}}u(t,x)t^{2m}h_1(t)\Delta\psi_{d, R}(t,x)dxdt}_{K_3}.
\end{aligned}
\end{equation*}
In view of  (\ref{integral>0}), there exists a $R_1>0$ such that for any $R\geq R_1,$  we have $K_0>0$, so
\begin{equation}\label{idr}
I_{d,R}<K_1+K_2-K_3\  \ \text{as}\ \  R\geq R_1.
\end{equation}

Now we estimate $K_1$, $K_2$, and $K_3$ separately. Denote by $p'$, $q'$ the conjugate numbers of $p,q$,  respectively, i.e.,  $\frac{1}{p}+\frac{1}{p'}=1$ and $\frac{1}{q}+\frac{1}{q'}=1$.
Using H$\ddot{\text{o}}$lder's inequality and the properties (\ref{lambdaphi}) of $\lambda(t)$ and $\phi(x)$ implies that
\begin{align}
&\vert K_1\vert
\leq d^{-2}\Biggl(\int_{\frac{d}{2}+1}^{d+1}\int_{B_{R^{m+1}}}h_2(t)\vert u(t,x)\vert^q\Big\vert\lambda^{\prime\prime}\bigl(\frac{t-1}{d}\bigr)\phi\bigl(\frac{x}{R^{m+1}}\bigr)\Big\vert^qdxdt\Biggl)^{\frac{1}{q}}\notag\\
&\quad\quad\quad\quad\quad\quad\times\Biggl(\int_{\frac{d}{2}+1}^{d+1}\int_{B_{R^{m+1}}}h_1^{q'}(t)h_2^{1-q'}(t)dxdt\Biggl)^{\frac{1}{q'}}\notag\\
&\lesssim d^{-2}(d+1)^{\beta_1+\beta_2(\frac{1}{q'}-1)+\frac{1}{q'}}R^{\frac{(m+1)n}{q'}}\Bigl(\int_{\frac{d}{2}+1}^{d+1}\int_{B_{R^{m+1}}}h_2(t)\vert u(t,x)\vert^q\psi_{d,R}(t,x)dxdt\Bigl)^{\frac{1}{q}}\notag\\
&\lesssim (d+1)^{-2+\beta_1+\beta_2(\frac{1}{q'}-1)+\frac{1}{q'}}R^{\frac{(m+1)n}{q'}}\widetilde{J}_{d,R}^{\frac{1}{q}},\label{K1}
\end{align}
where
\begin{equation*}
\widetilde{J}_{d,R}=\int_{\frac{d}{2}+1}^{d+1}\int_{B_{R^{m+1}}}h_2(t)\vert u(t,x)\vert^q\psi_{d,R}(t,x)dxdt.
\end{equation*}
Note that $2h_1^{\prime}(t)-\frac{\mu_1}{t}h_1(t)=(1-\sqrt{\delta_1})t^{-1}h_1(t)$, so
\begin{align}
\vert K_2\vert 
&\lesssim d^{-1}\Biggr(\int_{\frac{d}{2}+1}^{d+1}\int_{B_{R^{m+1}}}h_2(t)\vert u(t,x)\vert^q\Big\vert\lambda^\prime(\frac{t-1}{d})\phi\big(\frac{x}{R^{m+1}}\big)\Big\vert^qdxdt\Biggr)^{\frac{1}{q}}\notag\\
&\quad\quad\quad\times\Big(\int_{\frac{d}{2}+1}^{d+1}\int_{B_{R^{m+1}}}t^{-q'}h_1^{q'}(t)h_2^{1-q'}(t)dxdt\Big)^{\frac{1}{q'}}\notag\\
&\lesssim (d+1)^{-2+\beta_1+\beta_2(\frac{1}{q'}-1)+\frac{1}{q'}}R^{\frac{(m+1)n}{q'}}\widetilde{J}_{d,R}^{\frac{1}{q}}.\label{K2}
\end{align}
For $K_3$, the integration interval with respect to $t$  is $[1,d+1]$,   which differs from that of  $K_1,K_2,$ where the integration interval is $[\frac{d}{2}+1,d+1]$.   Consequently, we need  the condition (\ref{keycon2}) to   ensure  $2mq'+\beta_1q'+\beta_2(1-q')>-1$, i.e., $$\int_1^{d+1}t^{2mq'+\beta_1q'+\beta_2(1-q')}dt\lesssim(d+1)^{2mq'+\beta_1q'+\beta_2(1-q')+1}. $$ Then
\begin{align}
\vert K_3\vert 
&\lesssim R^{-2(m+1)}\Big(\int_1^{d+1}\int_{B_{R^{m+1}}\textbackslash B_{\frac{R^{m+1}}{2}}}t^{2mq'}h_1^{q'}(t)h_2^{1-q'}(t)dxdt\Big)^{\frac{1}{q'}}\notag\\
&\ \  \times\Big(\int_1^{d+1}\int_{B_{R^{m+1}}\textbackslash B_{\frac{R^{m+1}}{2}}}h_2(t)\vert u(t,x)\vert^q\Big\vert\lambda\bigl(\frac{t-1}{d}\bigr)\Delta\phi\bigl(\frac{x}{R^{m+1}}\bigr)\Bigr\vert^q dxdt\Bigr)^{\frac{1}{q}}\notag\\
&\lesssim (d+1)^{2m+\beta_1+\beta_2(\frac{1}{q'}-1)+\frac{1}{q'}}R^{-2(m+1)+\frac{(m+1)n}{q'}}\widehat{J}^{\frac{1}{q}}_{d,R},\label{K3}
\end{align}
where
\begin{equation*}
\widehat{J}_{d,R}=\int_{1}^{d+1}\int_{B_{R^{m+1}}\textbackslash B_{\frac{R^{m+1}}{2}}}h_2(t)\vert u(t,x)\vert^q\psi_{d,R}(t,x)dxdt.
\end{equation*}
By (\ref{idr})-(\ref{K3}), we get the estimate of $I_{d,R}$  as
\begin{align}
I_{d,R}\lesssim& (d+1)^{-2+\beta_1+\beta_2(\frac{1}{q'}-1)+\frac{1}{q'}}R^{\frac{(m+1)n}{q'}}\widetilde{J}_{d,R}^{\frac{1}{q}}\notag\\
&+(d+1)^{2m+\beta_1+\beta_2(\frac{1}{q'}-1)+\frac{1}{q'}}R^{-2(m+1)+\frac{(m+1)n}{q'}}\widehat{J}^{\frac{1}{q}}_{d,R},\
\ \text{for any}\ R\geq R_1.\label{idr0}
\end{align}

Substituting $\varPhi_2(t,x)=h_2(t)\psi_{d, R}(t,x)$ into (\ref{def32}) and using (\ref{eqhi}) yields
\begin{equation*}
\begin{aligned}
J_{d,R}:&=\int_1^{d+1}\int_{B_{R^{m+1}}}h_2(t)\psi_{d, R}(t,x)\vert u(t,x)\vert^qdxdt\\
        &=-\underbrace{\int_{B_{R^{m+1}}}\bigl(\frac{\mu_2-1+\sqrt{\delta_2}}{2}v_0(x)+v_1(x)\bigr)\phi(\frac{x}{R^{m+1}})dx}_{L_0}\\
        &\quad+\underbrace{\int_{\frac{d}{2}+1}^{d+1}\int_{B_{R^{m+1}}}v(t,x)h_2(t)\partial_t^2\psi_{d, R}(t,x)dxdt}_{L_1}\\
        &\quad+\underbrace{\int_{\frac{d}{2}+1}^{d+1}\int_{B_{R^{m+1}}}v(t,x)\bigl(2h_2^{\prime}(t)-\frac{\mu_2}{t}h_2(t)\bigr)\partial_t\psi_{d, R}(t,x)dxdt}_{L_2}\\
        &\quad-\underbrace{\int_1^{d+1}\int_{B_{R^{m+1}}\textbackslash B_{\frac{R^{m+1}}{2}}}v(t,x)t^{2m}h_2(t)\Delta\psi_{d, R}(t,x)dxdt}_{L_3}.
\end{aligned}
\end{equation*}
By (\ref{integral>0}),  there exists a $R_2>0$ such that $L_0>0$ holds for any $R\geq R_2,$   so
\begin{equation}\label{idj}
J_{d,R}<L_1+L_2-L_3 \ \ \text{as}\ \ R\geq R_2.
\end{equation}

By employing the same  techniques as in the estimates of $K_1,K_2,K_3$, we can derive the estimates of $L_1,L_2,L_3$.
\begin{align}
&\vert L_1\vert 
\leq d^{-2}\Biggl(\int_{\frac{d}{2}+1}^{d+1}\int_{B_{R^{m+1}}}h_1(t)\vert v(t,x)\vert^p\Big\vert\lambda^{\prime\prime}\bigl(\frac{t-1}{d}\bigr)\phi\bigl(\frac{x}{R^{m+1}}\bigr)\Big\vert^pdxdt\Biggl)^{\frac{1}{p}}\notag\\
&\quad\quad\quad\quad\quad\quad\times\Biggl(\int_{\frac{d}{2}+1}^{d+1}\int_{B_{R^{m+1}}}h_2^{p'}(t)h_1^{1-p'}(t)dxdt\Biggl)^{\frac{1}{p'}}\notag\\
&\lesssim d^{-2}(d+1)^{\beta_1(\frac{1}{p'}-1)+\beta_2+\frac{1}{p'}}R^{\frac{(m+1)n}{p'}}\Bigl(\int_{\frac{d}{2}+1}^{d+1}\int_{B_{R^{m+1}}}h_1(t)\vert v(t,x)\vert^p\psi_{d,R}(t,x)dxdt\Bigl)^{\frac{1}{p}}\notag\\
&\lesssim (d+1)^{-2+\beta_1(\frac{1}{p'}-1)+\beta_2+\frac{1}{p'}}R^{\frac{(m+1)n}{p'}}\widetilde{I}_{d,R}^{\frac{1}{p}},\label{L1}
\end{align}
where
\begin{equation*}
\widetilde{I}_{d,R}=\int_{\frac{d}{2}+1}^{d+1}\int_{B_{R^{m+1}}(0)}h_1(t)\vert v(t,x)\vert^p\psi_{d,R}(t,x)dxdt.
\end{equation*}
Since $2h_2^{\prime}(t)-\frac{\mu_2}{t}h_2(t)=(1-\sqrt{\delta_2})t^{-1}h_2(t)$,  it follows that
\begin{align}
&\vert L_2\vert 
\lesssim d^{-1}\Biggr(\int_{\frac{d}{2}+1}^{d+1}\int_{B_{R^{m+1}}}h_1(t)\vert v(t,x)\vert^p\Big\vert\lambda^\prime(\frac{t-1}{d})\phi\big(\frac{x}{R^{m+1}}\big)\Big\vert^pdxdt\Biggr)^{\frac{1}{p}}\notag \\
&\quad\quad\quad\quad\quad\quad\times\Big(\int_{\frac{d}{2}+1}^{d+1}\int_{B_{R^{m+1}}}t^{-p'}h_2^{p'}(t)h_1^{1-p'}(t)dxdt\Big)^{\frac{1}{p'}}\notag \\
&\lesssim (d+1)^{-2+\beta_1(\frac{1}{p'}-1)+\beta_2+\frac{1}{p'}}R^{\frac{(m+1)n}{p'}}\widetilde{I}_{d,R}^{\frac{1}{p}}.\label{L2}
\end{align}
The condition (\ref{keycon2}) also  ensures that $2mp'+\beta_2p'+\beta_1(1-p')>-1$, which guarantees that
\begin{equation*}
\int_1^{d+1}t^{2mp'+\beta_2p'+\beta_1(1-p')}dt\lesssim(d+1)^{2mp'+\beta_2p'+\beta_1(1-p')+1}, \end{equation*}
then
\begin{align}
&\vert L_3\vert
\lesssim R^{-2(m+1)}\Big(\int_1^{d+1}\int_{B_{R^{m+1}}\textbackslash B_{\frac{R^{m+1}}{2}}}t^{2mp'}h_2^{p'}(t)h_1^{1-p'}(t)dxdt\Big)^{\frac{1}{p'}}\notag\\
&\times\Big(\int_1^{d+1}\int_{B_{R^{m+1}}\textbackslash B_{\frac{R^{m+1}}{2}}}h_1(t)\vert v(t,x)\vert^p\Big\vert\lambda\bigl(\frac{t-1}{d}\bigr)\Delta\phi\bigl(\frac{x}{R^{m+1}}\bigr)\Bigr\vert^p dxdt\Bigr)^{\frac{1}{p}}\notag\\
&\lesssim (d+1)^{2m+\beta_1(\frac{1}{p'}-1)+\beta_2+\frac{1}{p'}}R^{-2(m+1)+\frac{(m+1)n}{p'}}\widehat{I}^{\frac{1}{p}}_{d,R},\label{L3}
\end{align}
where
\begin{equation*}
\widehat{I}_{d,R}=\int_{1}^{d+1}\int_{B_{R^{m+1}}\textbackslash B_{\frac{R^{m+1}}{2}}}h_1(t)\vert v(t,x)\vert^p\psi_{d,R}(t,x)dxdt.
\end{equation*}
Substituting the estimates (\ref{L1}), (\ref{L2}), (\ref{L3}) of $L_1$, $L_2$ and $L_3$ into (\ref{idj}) gives that
\begin{equation}\label{jdr0}
\begin{aligned}
J_{d,R}\lesssim & (d+1)^{-2+\beta_1(\frac{1}{p'}-1)+\beta_2+\frac{1}{p'}}R^{\frac{(m+1)n}{p'}}\widetilde{I}_{d,R}^{\frac{1}{p}}\\
&+(d+1)^{2m+\beta_1(\frac{1}{p'}-1)+\beta_2+\frac{1}{p'}}R^{-2(m+1)+\frac{(m+1)n}{p'}}\widehat{I}^{\frac{1}{p}}_{d,R}
\end{aligned}
\end{equation}
holds for any $R\geq R_2$.

Let $d=R\geq\max\{1,R_1,R_2\}$, from (\ref{idr0}) and (\ref{jdr0}), we get
\begin{align}
&I_{R,R}\lesssim(R+1)^{-2+\beta_1+\beta_2(\frac{1}{q'}-1)+\frac{1}{q'}+\frac{(m+1)n}{q'}}
\big(\widetilde{J}_{R,R}^{\frac{1}{q}}+\widehat{J}^{\frac{1}{q}}_{R,R}\big),\label{IJRR1}\\
&J_{R,R}\lesssim(R+1)^{-2+\beta_2+\beta_1(\frac{1}{p'}-1)+\frac{1}{p'}+\frac{(m+1)n}{p'}}
\big(\widetilde{I}_{R,R}^{\frac{1}{p}}+\widehat{I}^{\frac{1}{p}}_{R,R}\big).\label{IJRR2}
\end{align}
Note that $\widetilde{I}_{R,R},\widehat{I}_{R,R}\leq I_{R,R}$ and $\widetilde{J}_{R,R},\widehat{J}_{R,R}\leq J_{R,R}$.  Using  (\ref{IJRR1}) and (\ref{IJRR2}),  we obtain
\begin{equation}\label{IRReq}
\begin{aligned}
&I_{R,R}\lesssim(R+1)^{-2+\beta_1+\beta_2(\frac{1}{q'}-1)+\frac{1}{q'}+\frac{(m+1)n}{q'}}J_{R,R}^{\frac{1}{q}}\\
&\lesssim (R+1)^{-2+\beta_1+\beta_2(\frac{1}{q'}-1)+\frac{1}{q'}+\frac{(m+1)n}{q'}+\frac{1}{q}
\big(-2+\beta_1(\frac{1}{p'}-1)+\beta_2+\frac{1}{p'}+\frac{(m+1)n}{p'}\big)}I_{R,R}^{\frac{1}{pq}}\\
&=(R+1)^{-2-\frac{2}{q}+(1-\frac{1}{pq})\big((m+1)n+\beta_1+1\big)}I_{R,R}^{\frac{1}{pq}}
\end{aligned}
\end{equation}
and
\begin{equation}\label{JRReq}
\begin{aligned}
&J_{R,R}\lesssim(R+1)^{-2+\beta_2+\beta_1(\frac{1}{p'}-1)+\frac{1}{p'}+\frac{(m+1)n}{p'}}I_{R,R}^{\frac{1}{p}}\\
&\lesssim (R+1)^{-2+\beta_2+\beta_1(\frac{1}{p'}-1)+\frac{1}{p'}+\frac{(m+1)n}{p'}+\frac{1}{p}\big(-2+\beta_1+\beta_2(\frac{1}{q'}-1)+\frac{1}{q'}
+\frac{(m+1)n}{q'}\big)}J_{R,R}^{\frac{1}{pq}}\\
&=(R+1)^{-2-\frac{2}{p}+(1-\frac{1}{pq})\big((m+1)n+\beta_2+1\big)}J_{R,R}^{\frac{1}{pq}}.
\end{aligned}
\end{equation}
In other words,
\begin{align}
&I_{R,R}\lesssim (R+1)^{-2p\frac{q+1}{pq-1}+(m+1)n+\beta_1+1},\label{Irreq}\\
&J_{R,R}\lesssim (R+1)^{-2q\frac{p+1}{pq-1}+(m+1)n+\beta_2+1}.\label{Jrreq}
\end{align}
If $-2p\frac{q+1}{pq-1}+(m+1)n+\beta_1+1<0$, i.e., $\frac{p+1}{pq-1}-\frac{\beta_1-1}{2}>\frac{(m+1)n}{2}$, by (\ref{Irreq}), $\lim\limits_{R\rightarrow+\infty}I_{R,R}=0$ holds.  Then
\begin{equation}\label{contheorem}
\resizebox{0.913\hsize}{!}{$\begin{aligned}
\lim\limits_{R\rightarrow+\infty}\int_1^{R+1}\int_{B_{R^{m+1}}}h_1(t)\vert v(t,x)&\vert^p\psi_{R, R}(t,x)dxdt=\int_1^{\infty}\int_{R^{n}}h_1(t)\vert v(t,x)\vert^pdxdt=0,
\end{aligned}$}
\end{equation}
which yields $v=0$ almost everywhere. This contradicts the assumption (\ref{integral>0}) on ($v_0, v_1$).
If $-2p\frac{q+1}{pq-1}+(m+1)n+\beta_1+1=0$, i.e., $\frac{p+1}{pq-1}-\frac{\beta_1-1}{2}=\frac{(m+1)n}{2}$, by (\ref{Irreq}), $\lim\limits_{R\rightarrow+\infty}I_{R,R}\leq C$ holds. Then
\begin{equation*}
\begin{aligned}
&\lim\limits_{R\rightarrow+\infty}\int_1^{R+1}\int_{B_{R^{m+1}}}h_1(t)\vert v(t,x)\vert^p\psi_{R, R}(t,x)dxdt=\int_1^{\infty}\int_{R^{n}}h_1(t)\vert v(t,x)\vert^pdxdt\leq C,
\end{aligned}
\end{equation*}
which means $h_1(t)\vert v(t,x)\vert^p\in L^1([1,\infty)\times\mathbb{R}^n)$. Consequently, employing the dominated convergence theorem and the definitions of $\widetilde{I}_{R,R}$, $\widehat{I}_{R,R}$ gives that
\begin{equation*}
\lim\limits_{R\rightarrow+\infty}\widetilde{I}_{R,R}=\lim\limits_{R\rightarrow+\infty}\widehat{I}_{R,R}=0.
\end{equation*}
From (\ref{IJRR1}), we have
$\lim\limits_{R\rightarrow+\infty}I_{R,R}=0$, so (\ref{contheorem}) holds in this case and we derive the same contradiction.

By applying (\ref{Jrreq}) and the same argument above, we can get $u=0$ almost everywhere under the condition $\frac{q+1}{pq-1}-\frac{\beta_2-1}{2}\geq\frac{(m+1)n}{2}$, which contradicts the assumption (\ref{integral>0}) on ($u_0, u_1$).

In summary, assuming  $T=\infty$ leads to the contradiction under the condition
\begin{equation}\label{blowupconditionproof}
\frac{p+1}{pq-1}-\frac{\beta_1-1}{2}\geq\frac{(m+1)n}{2}\ \ \ \text{or}\ \ \ \frac{q+1}{pq-1}-\frac{\beta_2-1}{2}\geq\frac{(m+1)n}{2},
\end{equation}
provided $p,q$ fulfill (\ref{keycon2}),  which means the energy solution $(u,v)$ cannot exist globally. Since the  relations  (\ref{blowupconditionproof}) on $p, q$ are equivalent
\begin{equation*}
\Gamma_m(n,p,q,\beta_1,\beta_2)\geq0,
\end{equation*}
the proof is completed.

\section{Global existence of solutions with high regularity of initial data}\label{section4high}
In this section, we aim to establish the global existence of  solutions with high regularity of initial data, i.e., Theorems \ref{theorem11} - \ref{theorem13}. The key tool is the estimates of the solution and its derivatives for the corresponding linear equation of (\ref{eqs}), which were precisely established in  our recent study \cite{LiGuo2025}.  These results will be outlined in Section \ref{sec4.1} without repeating the proofs. Based on these estimates, we can construct the appropriate space and the associated solution operator, then apply Duhamel's principle along with the contraction mapping principle to derive the global existence.
\subsection{The estimates for the corresponding linear equation}\label{sec4.1}
For the corresponding linear equation of (\ref{eqs})
\begin{equation}\label{lineareq}
  \left\{
\begin{aligned}
&\partial_t^2u-t^{2m}\Delta u+\frac{\mu}{t}\partial_tu+\frac{\nu^2}{t^2}u=0,\ &\quad &t>\tau\geq1,\\
&u(\tau,x)= f(x), \ \partial_tu(\tau,x)= g(x),               &\quad &x\in\mathbb{R}^n,
\end{aligned}
 \right.
  \end{equation}
we established the $(L^1\cap L^2)-L^2$ estimates in \cite{LiGuo2025} and we list the results.
\begin{proposition}\label{linearestimate}
Let $\sigma>0$, $\delta=(\mu-1)^2-4\nu^2>0$ and $(f,g)\in D^\sigma$, then for any $\kappa\in[0,\sigma]$, the solution $u$ to (\ref{lineareq}) with $\tau=1$ satisfies
\begin{equation}\label{linearesu}
\begin{aligned}
\Vert &u(t,\cdot)\Vert_{\dot{H}^\kappa}\lesssim\Vert (f,g)\Vert_{D^\kappa}
\left\{
\begin{aligned}
&t^{-\frac{\mu+m}{2}},\ &\text{if}\ \kappa>\frac{\sqrt\delta}{2(m+1)}+\frac{1}{2}-\frac{n}{2},\\
&t^{-\frac{\mu+m}{2}}\bigl(1+\log{t}\bigr)^{\frac{1}{2}},\ &\text{if}\ \kappa=\frac{\sqrt\delta}{2(m+1)}+\frac{1}{2}-\frac{n}{2},\\
&t^{-(m+1)(\kappa+\frac{n}{2})+\frac{\sqrt\delta-\mu+1}{2}},\ &\text{if}\ \kappa<\frac{\sqrt\delta}{2(m+1)}+\frac{1}{2}-\frac{n}{2}.
\end{aligned}
\right.
\end{aligned}
\end{equation}
Moreover, for any $\kappa\in[1,\sigma],$
\begin{equation}\label{linearesut}
\resizebox{0.926\hsize}{!}{$\begin{aligned}
\Vert &\partial_tu(t,\cdot)\Vert_{\dot{H}^{\kappa-1}}\lesssim\Vert (f,g)\Vert_{D^\kappa}
\left\{
\begin{aligned}
&t^{m-\frac{\mu+m}{2}},\ &\text{if}\ \kappa>\frac{\sqrt\delta}{2(m+1)}+\frac{1}{2}-\frac{n}{2},\\
&t^{m-\frac{\mu+m}{2}}\bigl(1+\log{t}\bigr)^{\frac{1}{2}},\ &\text{if}\ \kappa=\frac{\sqrt\delta}{2(m+1)}+\frac{1}{2}-\frac{n}{2},\\
&t^{m-(m+1)(\kappa+\frac{n}{2})+\frac{\sqrt\delta-\mu+1}{2}},\ &\text{if}\ \kappa<\frac{\sqrt\delta}{2(m+1)}+\frac{1}{2}-\frac{n}{2}.
\end{aligned}\right.
\end{aligned}$}
\end{equation}

\end{proposition}

\begin{proposition}\label{unlinearestimate}
Let $\delta=(\mu-1)^2-4\nu^2>0$, $f=0$, $g\in H^{[\sigma-1]_+}\cap L^1$ with $\sigma>0$, then for any $\kappa\in[0,\sigma]$, the solution $u$ to (\ref{lineareq}) satisfies

\begin{equation}\label{u0=0u}
\begin{aligned}
\Vert u(t,\cdot)&\Vert_{\dot{H}^\kappa}\lesssim\bigl(\Vert g\Vert_{L^1}+\tau^{(m+1)(\frac{n}{2}+[\kappa-1]_+)}\Vert g\Vert_{\dot{H}^{[\kappa-1]_+}}\bigr)\\[5pt]
                &\times
                \left\{
                \begin{aligned}
                &t^{-\frac{\mu+m}{2}}\tau^{-(m+1)(\kappa+\frac{n}{2})+\frac{\mu+m}{2}+1},\ &\text{if}\ \kappa>\frac{\sqrt\delta}{2(m+1)}+\frac{1}{2}-\frac{n}{2},\\
                &t^{-\frac{\mu+m}{2}}\tau^{-\frac{\sqrt\delta-\mu-1}{2}}\bigl(1+\log\frac{t}{\tau}\bigr)^{\frac{1}{2}},\ &\text{if}\ \kappa=\frac{\sqrt\delta}{2(m+1)}+\frac{1}{2}-\frac{n}{2},\\
                &t^{-(m+1)(\kappa+\frac{n}{2})+\frac{\sqrt\delta-\mu+1}{2}}\tau^{-\frac{\sqrt\delta-\mu-1}{2}},\ &\text{if}\ \kappa<\frac{\sqrt\delta}{2(m+1)}+\frac{1}{2}-\frac{n}{2}.\\
                \end{aligned}
                \right.
\end{aligned}
\end{equation}
Moreover, for any $\kappa\in [1,\sigma]$, 
\begin{align}
&\Vert \partial_tu(t,\cdot)\Vert_{\dot{H}^{\kappa-1}}\lesssim\bigl(\Vert g\Vert_{L^1}+\tau^{(m+1)(\frac{n}{2}+\kappa-1)}\Vert g\Vert_{\dot{H}^{\kappa-1}}+\tau^{(m+1)(\frac{n}{2}+[\kappa-2]_+)}\Vert g\Vert_{\dot{H}^{[\kappa-2]_+}}\bigr)\nonumber\\[5pt]
                &\quad\quad\times
                \left\{
                \begin{aligned}\label{u0=0ut}
                &t^{m-\frac{\mu+m}{2}}\tau^{-(m+1)(\kappa+\frac{n}{2})+\frac{\mu+m}{2}+1},\ &\text{if}\ \kappa>\frac{\sqrt\delta}{2(m+1)}+\frac{1}{2}-\frac{n}{2},\\
                &t^{m-\frac{\mu+m}{2}}\tau^{-\frac{\sqrt\delta-\mu-1}{2}}\bigl(1+\log\frac{t}{\tau}\bigr)^{\frac{1}{2}},\ &\text{if}\ \kappa=\frac{\sqrt\delta}{2(m+1)}+\frac{1}{2}-\frac{n}{2},\\
                &t^{m-(m+1)(\kappa+\frac{n}{2})+\frac{\sqrt\delta-\mu+1}{2}}\tau^{-\frac{\sqrt\delta-\mu-1}{2}},\ &\text{if}\ \kappa<\frac{\sqrt\delta}{2(m+1)}+\frac{1}{2}-\frac{n}{2}.\\
                \end{aligned}
                \right.
\end{align}

\end{proposition}

\subsection{The proof of Theorem \ref{theorem11}}\label{prooftheorem2}
Denote by $E_0^{\ \mu,\nu}(t,\tau,x)$ and $E_1^{\ \mu,\nu}(t,\tau,x)$ the  fundamental solutions to (\ref{lineareq}) with the initial data $(f,g)=(\delta_0,0)$ and $(0,\delta_0)$, respectively, where $\delta_0$ is the Dirac function. By Duhamel's principle, the solution to  (\ref{eqs}) can be expressed as
\begin{equation}\label{expressuv}
\resizebox{0.925\hsize}{!}{$
\begin{aligned}
&u(t,x)=E_0^{\ \mu_1,\nu_1}(t,1,x)\ast u_0(x)+E_1^{\ \mu_1,\nu_1}(t,1,x)\ast u_1(x)+\int_1^tE_1^{\ \mu_1,\nu_1}(t,\tau,x)\ast\vert v(\tau,x)\vert^pd\tau,\\
&v(t,x)=E_0^{\ \mu_2,\nu_2}(t,1,x)\ast v_0(x)+E_1^{\ \mu_2,\nu_2}(t,1,x)\ast v_1(x)+\int_1^tE_1^{\ \mu_2,\nu_2}(t,\tau,x)\ast\vert u(\tau,x)\vert^qd\tau,
\end{aligned}$}
\end{equation}
where $\ast$ denotes the convolution with respect to $x$. For convenience, let
\begin{equation}\label{ulvlh1h2}
\begin{aligned}
&u^l(t,x)=E_0^{\ \mu_1,\nu_1}(t,1,x)\ast u_0(x)+E_1^{\ \mu_1,\nu_1}(t,1,x)\ast u_1(x),\\
&v^l(t,x)=E_0^{\ \mu_2,\nu_2}(t,1,x)\ast v_0(x)+E_1^{\ \mu_2,\nu_2}(t,1,x)\ast v_1(x),\\
&H_1(v)(t,x)=\int_1^tE_1^{\ \mu_1,\nu_1}(t,\tau,x)\ast\vert v(\tau,x)\vert^pd\tau,\\
&H_2(u)(t,x)=\int_1^tE_1^{\ \mu_2,\nu_2}(t,\tau,x)\ast\vert u(\tau,x)\vert^qd\tau.
\end{aligned}
\end{equation}
For the sake of clarity, let us introduce  some notations and expressions. $\ell_i(t) \ (i=1,2)$ are shown as (\ref{l(t)}),
if $\sigma>1$, let
\begin{align}
M_1(t,u)=&t^{-\frac{\sqrt{\delta_1}-\mu_1+1}{2}+(m+1)\frac{n}{2}}\Vert u(t,\cdot)\Vert_{L^2}\notag\\
              &+t^{-\frac{\sqrt{\delta_1}-\mu_1+1}{2}+(m+1)(\sigma+\frac{n}{2})}\ell_1^{-1}(t)\Vert u(t,\cdot)\Vert_{\dot{H}^\sigma}\notag\\
                               &+t^{-m-\frac{\sqrt{\delta_1}-\mu_1+1}{2}+(m+1)(1+\frac{n}{2})}\Vert \partial_tu(t,\cdot)\Vert_{L^2}\label{M1}\\
                               &+t^{-m-\frac{\sqrt{\delta_1}-\mu_1+1}{2}+(m+1)(\sigma+\frac{n}{2})}\ell_1^{-1}(t)\Vert \partial_tu(t,\cdot)\Vert_{\dot{H}^{\sigma-1}},\notag\\
M_2(t,v)=&t^{-\frac{\sqrt{\delta_2}-\mu_2+1}{2}+(m+1)\frac{n}{2}}\Vert v(t,\cdot)\Vert_{L^2}\notag\\
                               &+t^{-\frac{\sqrt{\delta_2}-\mu_2+1}{2}+(m+1)(\sigma+\frac{n}{2})}\ell_2^{-1}(t)\Vert v(t,\cdot)\Vert_{\dot{H}^\sigma}\notag\\
                               &+t^{-m-\frac{\sqrt{\delta_2}-\mu_2+1}{2}+(m+1)(1+\frac{n}{2})}\Vert \partial_tv(t,\cdot)\Vert_{L^2}\label{M2}\\
                               &+t^{-m-\frac{\sqrt{\delta_2}-\mu_2+1}{2}+(m+1)(\sigma+\frac{n}{2})}\ell_2^{-1}(t)\Vert \partial_tv(t,\cdot)\Vert_{\dot{H}^{\sigma-1}};\notag
\end{align}
while if $\sigma=1$,
\begin{align}
M_1(t,u)=&t^{-\frac{\sqrt{\delta_1}-\mu_1+1}{2}+(m+1)\frac{n}{2}}\Vert u(t,\cdot)\Vert_{L^2}\notag\\
                               &+t^{-\frac{\sqrt{\delta_1}-\mu_1+1}{2}+(m+1)(\sigma+\frac{n}{2})}\ell_1^{-1}(t)\Vert u(t,\cdot)\Vert_{\dot{H}^\sigma}\notag\\
                               &+t^{-m-\frac{\sqrt{\delta_1}-\mu_1+1}{2}+(m+1)(\sigma+\frac{n}{2})}\ell_1^{-1}(t)\Vert \partial_tu(t,\cdot)\Vert_{\dot{H}^{\sigma-1}},\label{M11}\\
M_2(t,v)=&t^{-\frac{\sqrt{\delta_2}-\mu_2+1}{2}+(m+1)\frac{n}{2}}\Vert v(t,\cdot)\Vert_{L^2}\notag\\
                               &+t^{-\frac{\sqrt{\delta_2}-\mu_2+1}{2}+(m+1)(\sigma+\frac{n}{2})}\ell_2^{-1}(t)\Vert v(t,\cdot)\Vert_{\dot{H}^\sigma}\label{M21}\\
                               &+t^{-m-\frac{\sqrt{\delta_2}-\mu_2+1}{2}+(m+1)(\sigma+\frac{n}{2})}\ell_2^{-1}(t)\Vert \partial_tv(t,\cdot)\Vert_{\dot{H}^{\sigma-1}}.\notag
\end{align}

For $\sigma\geq1$  and $T>1$, define  the function space
\begin{equation*}
\begin{aligned}
&\mathcal{X}(T):=\Big\{(u,v)\in(C\bigl([1,T];H^\sigma\bigr)\cap C^1\bigl([1,T];H^{\sigma-1}\bigr)\Big)^2\\
&\quad\quad\quad\quad\quad\quad \text{such that}\ \ \text{supp}\big(u(t,\cdot),v(t,\cdot)\big)\subset B_{\phi_m(t)-\phi_m(1)+M}\Big\}
\end{aligned}
\end{equation*}
equipped with the norm
\begin{equation}\label{normxt}
\Vert(u,v)\Vert_{\mathcal{X}(T)}= \sup\limits_{t\in[1,T]}\big(t^{-\alpha_1}M_1(t,u)+t^{-\alpha_2}M_2(t,v)\big),
\end{equation}
where $\phi_m(t)$ is defined by (\ref{compact2}) and  $\alpha_1, \alpha_2$ are given by
\begin{equation}\label{alpha1}
\alpha_1=\left\{
\begin{aligned}
&\big((m+1)n+\beta_2-1\big)(\tilde{p}-p),&\quad& \text{if}\ \ p<\tilde{p}, \\
&\epsilon, &\quad&\text{if}\ \ p=\tilde{p},\\
&0,&\quad&\text{if}\ \ p>\tilde{p},
\end{aligned}
\right.
\end{equation}
\begin{equation}\label{alpha2}
\alpha_2=\left\{
\begin{aligned}
&\big((m+1)n+\beta_1-1\big)(\tilde{q}-q),&\quad& \text{if}\ \ q<\tilde{q}, \\
&\epsilon, &\quad&\text{if}\ \ q=\tilde{q},\\
&0,&\quad&\text{if}\ \ q>\tilde{q},
\end{aligned}
\right.
\end{equation}
where $\epsilon>0$ is  sufficiently small. Specially, $\alpha_1=\alpha_2=0$ under the assumptions of Theorem \ref{theorem11}.

Based on the  representations  (\ref{expressuv}) of the solution to (\ref{eqs}),  we define the operator $\mathcal{N}$ by
\begin{equation}\label{operator}
\mathcal{N}(u,v):=\big(u^l+H_1(v),v^l+H_2(u)\big)
\end{equation}
and introduce the subset of $\mathcal{X}(T)$
\begin{equation*}
X(T,K)=\bigl\{(u,v)\in\mathcal{X}(T):\Vert (u,v)\Vert_{\mathcal{X}(T)}\leq K\bigr\},
\end{equation*}
where $u^l,v^l,H_1(v),H_2(u)$ are defined in (\ref{ulvlh1h2}) and $K$ is a positive constant to be determined. Then we have
\begin{proposition}\label{keypro}
Under the conditions of Theorem \ref{theorem11},
there exists a constant $C>0$ such that for any $T>1$  and any $(u,v), (\tilde{u},\tilde{v})\in \mathcal{X}(T)$,
\begin{align}
&\Vert \mathcal{N}(u,v)\Vert_{\mathcal{X}(T)}\leq C\bigl(\Vert(u_0,u_1)\Vert_{D^\sigma}+\Vert(v_0,v_1)\Vert_{D^\sigma}\big)+C\big(\Vert (u,v)\Vert_{\mathcal{X}{(T)}}^p+\Vert (u,v)\Vert_{\mathcal{X}{(T)}}^q\bigr),\label{key1}\\
&\Vert \mathcal{N}(u,v)-\mathcal{N}(\tilde{u},\tilde{v})\Vert_{\mathcal{X}(T)}\leq C \Vert (u,v)-(\tilde{u},\tilde{v})\Vert_{\mathcal{X}(T)}\notag\\
&\quad\quad\quad\quad\quad\quad\quad\quad\times\bigl(\Vert (u,v)\Vert_{\mathcal{X}(T)}^{p-1}+\Vert (u,v)\Vert_{\mathcal{X}(T)}^{q-1}+\Vert (\tilde{u},\tilde{v})\Vert_{\mathcal{X}(T)}^{p-1}+\Vert (\tilde{u},\tilde{v})\Vert_{\mathcal{X}(T)}^{q-1}\bigr).\label{key2}
\end{align}
\end{proposition}

The proofs of (\ref{key1}) and (\ref{key2}) are rather lengthy, so we  present them in detail in the following two sections. To conclude this section, we use Proposition \ref{keypro} to establish the global existence of solutions for $\sigma\geq1$, thereby completing  the proof of Theorem \ref{theorem11}.
\begin{proof}[Proof of Theorem \ref{theorem11}]
Choose $K=3C\Vert(u_0,u_1)\Vert_{D^\sigma}+3C\Vert(v_0,v_1)\Vert_{D^\sigma}$ in $X(T,K)$, where  $C$ is the constant mentioned in Proposition \ref{keypro}.  Then for any $(u,v),(\tilde{u},\tilde{v})\in \mathcal{X}(T,K)$, it follows from (\ref{key1}) that
\begin{align}
&\Vert \mathcal{N}(u,v)\Vert_{\mathcal{X}(T)} 
\leq C(\Vert(u_0,u_1)\Vert_{D^\sigma}+\Vert(v_0,v_1)\Vert_{D^\sigma})\notag\\
&\times \Big(1+(3C)^p\big(\Vert (u_0,u_1)\Vert_{D^\sigma}+\Vert (v_0,v_1)\Vert_{D^\sigma}\big)^{p-1}+(3C)^q\big(\Vert (u_0,u_1)\Vert_{D^\sigma}+\Vert (v_0,v_1)\Vert_{D^\sigma}\big)^{q-1}\Big)\notag\\
&\leq 3C(\Vert (u_0,u_1)\Vert_{D^\sigma}+\Vert (v_0,v_1)\Vert_{D^\sigma})=K\label{keyyy1}
\end{align}
provided that  $\Vert(u_0,u_1)\Vert_{D^\sigma}+\Vert(v_0,v_1)\Vert_{D^\sigma}\leq\min\big\{\bigl(\frac{1}{3^pC^p}\bigr)^{\frac{1}{p-1}}, \bigl(\frac{1}{3^qC^q}\bigr)^{\frac{1}{q-1}}\big\}$. Moreover, by (\ref{key2}), we have
\begin{align}
&\Vert \mathcal{N}(u,v)-\mathcal{N}(\tilde{u},\tilde{v})\Vert_{\mathcal{X}(T)}
\leq \Vert (u,v)-(\tilde{u},\tilde{v})\Vert_{\mathcal{X}(T)}\notag\\
&\ \times2C\Big(\big(3C(\Vert (u_0,u_1)\Vert_{D^\sigma}+\Vert (v_0,v_1)\Vert_{D^\sigma})\big)^{p-1}+\big(3C(\Vert (u_0,u_1)\Vert_{D^\sigma}+\Vert (v_0,v_1)\Vert_{D^\sigma})\big)^{q-1}\Big)\notag\\
&\leq \frac{1}{2}\Vert (u,v)-(\tilde{u},\tilde{v})\Vert_{\mathcal{X}(T)}\label{keyyy2}
\end{align}
as long as $\Vert(u_0,u_1)\Vert_{D^\sigma}+\Vert(v_0,v_1)\Vert_{D^\sigma}\leq\min\big\{
\frac{1}{8C^p3^{p-1}},\frac{1}{8C^q3^{q-1}}\big\}$.

Therefore, if
\begin{equation*}
\resizebox{0.95\hsize}{!}{$
\Vert(u_0,u_1)\Vert_{D^\sigma}+\Vert(v_0,v_1)\Vert_{D^\sigma}\leq\varepsilon:= \min\big\{\bigl(\frac{1}{3^pC^p}\bigr)^{\frac{1}{p-1}}, \bigl(\frac{1}{3^qC^q}\bigr)^{\frac{1}{q-1}},\frac{1}{8C^p3^{p-1}},\frac{1}{8C^q3^{q-1}} \big\}$},
\end{equation*}
from (\ref{keyyy1}) and (\ref{keyyy2}), we see that  $\mathcal{N}$ is a contraction mapping from$X(T,K)$ into itself. According to the  contraction mapping principle, there exists a unique $(u,v)\in X(T,K)$ such that $\mathcal{N}(u,v)=(u,v)$, i.e., $(u,v)$ is the unique solution to (\ref{eqs}). Furthermore, due to the constant $C$ does not depend on the choice of $T$,  $(u,v)$ is  actually  the global solution. The choice of $K$ gives that the solution $(u,v)$ satisfies $\Vert(u,v)\Vert_{\mathcal{X}(t)}\leq 3C\big(\Vert(u_0,u_1)\Vert_{D^\sigma}+\Vert(v_0,v_1)\Vert_{D^\sigma}\big)$ for any $t>1$, which yields the validness of the estimates shown in Theorem \ref{theorem11}.
\end{proof}

\subsubsection{The proof of (\ref{key1})}\label{proofkey1}
By the definition (\ref{operator}) of $\mathcal{N},$ we  now  provide the estimate for ($u^l,v^l$). In the case of $\delta_1,\delta_2>(m+1)^2(n+2\sigma-1)^2$ with $\sigma>1$, by Proposition \ref{linearestimate}, we have
\begin{equation}\label{ull2hsigmaetc}
\begin{aligned}
&\Vert u^l(t,\cdot)\Vert_{L^2}\lesssim t^{-(m+1)\frac{n}{2}+\frac{\sqrt{\delta_1}-\mu_1+1}{2}}\Vert (u_0,u_1)\Vert_{L^2},\\
&\Vert u^l(t,\cdot)\Vert_{\dot{H}^\sigma}\lesssim t^{-(m+1)(\sigma+\frac{n}{2})+\frac{\sqrt{\delta_1}-\mu_1+1}{2}}\Vert (u_0,u_1)\Vert_{D^\sigma},\\
&\Vert \partial_tu^l(t,\cdot)\Vert_{L^2}\lesssim t^{m-(m+1)(1+\frac{n}{2})+\frac{\sqrt{\delta_1}-\mu_1+1}{2}}\Vert (u_0,u_1)\Vert_{D^1},\\
&\Vert \partial_tu^l(t,\cdot)\Vert_{\dot{H}^{\sigma-1}}\lesssim t^{m-(m+1)(\sigma+\frac{n}{2})+\frac{\sqrt{\delta_1}-\mu_1+1}{2}}\Vert (u_0,u_1)\Vert_{D^\sigma}
\end{aligned}
\end{equation}
and
\begin{equation}\label{vll2hsigmaetc}
\begin{aligned}
&\Vert v^l(t,\cdot)\Vert_{L^2}\lesssim t^{-(m+1)\frac{n}{2}+\frac{\sqrt{\delta_2}-\mu_2+1}{2}}\Vert (v_0,v_1)\Vert_{L^2},\\
&\Vert v^l(t,\cdot)\Vert_{\dot{H}^\sigma}\lesssim t^{-(m+1)(\sigma+\frac{n}{2})+\frac{\sqrt{\delta_2}-\mu_2+1}{2}}\Vert (v_0,v_1)\Vert_{D^\sigma},\\
&\Vert \partial_tv^l(t,\cdot)\Vert_{L^2}\lesssim t^{m-(m+1)(1+\frac{n}{2})+\frac{\sqrt{\delta_2}-\mu_2+1}{2}}\Vert (v_0,v_1)\Vert_{D^1},\\
&\Vert \partial_tv^l(t,\cdot)\Vert_{\dot{H}^{\sigma-1}}\lesssim t^{m-(m+1)(\sigma+\frac{n}{2})+\frac{\sqrt{\delta_2}-\mu_2+1}{2}}\Vert (v_0,v_1)\Vert_{D^\sigma}.
\end{aligned}
\end{equation}
By the definitions (\ref{M1}), (\ref{M2}) of $M_1(t,u^l)$, $M_2(t,v^l)$, we get
\begin{equation*}
M_1(t,u^l)+M_2(t,v^l)\lesssim \Vert (u_0,u_1)\Vert_{D^\sigma}+\Vert (v_0,v_1)\Vert_{D^\sigma},
\end{equation*}
which yields $\Vert(u^l,v^l)\Vert_{\mathcal{X}(T)}\lesssim \Vert (u_0,u_1)\Vert_{D^\sigma}+\Vert (v_0,v_1)\Vert_{D^\sigma}$ by (\ref{normxt}).

In the same way, we can deal with the remaining cases
\begin{align}
&\text{Case 2:}\ \delta_1>\bigl((m+1)(n+2\sigma-1)\bigr)^2,\ \delta_2=\bigl((m+1)(n+2\sigma-1)\bigr)^2,\ \sigma>1, \label{bigcase2}\\
&\text{Case 3:}\ \delta_1=\bigl((m+1)(n+2\sigma-1)\bigr)^2, \ \delta_2>\bigl((m+1)(n+2\sigma-1)\bigr)^2,\ \sigma>1, \label{bigcase3}\\
&\text{Case 4:}\ \delta_1=\bigl((m+1)(n+2\sigma-1)\bigr)^2, \ \delta_2=\bigl((m+1)(n+2\sigma-1)\bigr)^2,\ \sigma>1, \label{bigcase4}\\
&\text{Case 5:}\ \delta_1>\bigl((m+1)(n+2\sigma-1)\bigr)^2,\ \delta_2>\bigl((m+1)(n+2\sigma-1)\bigr)^2,\ \sigma=1,  \label{bigcase5}\\
&\text{Case 6:}\ \delta_1>\bigl((m+1)(n+2\sigma-1)\bigr)^2, \ \delta_2=\bigl((m+1)(n+2\sigma-1)\bigr)^2,\ \sigma=1, \label{bigcase6}\\
&\text{Case 7:}\ \delta_1=\bigl((m+1)(n+2\sigma-1)\bigr)^2, \ \delta_2>\bigl((m+1)(n+2\sigma-1)\bigr)^2,\ \sigma=1, \label{bigcase7}\\
&\text{Case 8:}\ \delta_1=\bigl((m+1)(n+2\sigma-1)\bigr)^2, \ \delta_2=\bigl((m+1)(n+2\sigma-1)\bigr)^2,\ \sigma=1. \label{bigcase8}
\end{align}

Combining the above eight cases gives the estimates of $(u^l,v^l)$ as
\begin{equation}\label{ulvlestimate}
\Vert(u^l,v^l)\Vert_{\mathcal{X}(T)}\lesssim \Vert (u_0,u_1)\Vert_{D^\sigma}+\Vert (v_0,v_1)\Vert_{D^\sigma}.
\end{equation}

In what follows, we will use Proposition \ref{unlinearestimate} and Duhamel's principle to establish the estimates for the nonlinear term, namely,
\begin{equation}\label{nones}
\Vert\big(H_1(v),H_2(u)\big)\Vert_{\mathcal{X}(T)}\lesssim \Vert (u,v)\Vert_{\mathcal{X}{(T)}}^p+\Vert (u,v)\Vert_{\mathcal{X}{(T)}}^q.
\end{equation}
From (\ref{ulvlestimate}) and (\ref{nones}), we have essentially  proved (\ref{key1}).

Before proving (\ref{nones}), let us  list some primary results  that will be used throughout the whole proof.

\begin{lemma}[\cite{LiGuo2025}]\label{coro4.3}
For  a function $w\in\mathcal{X}(T)$, $T>1$ and $1\leq\tau< t\leq T$, the following estimates can be derived

(1)
\begin{align}
&\Vert \vert w(\tau,\cdot)\vert^s\Vert_{\dot{H}^{\sigma-1}}\lesssim \Vert w(\tau,\cdot)\Vert_{L^{r_1}}^{s-1}\Vert\vert D\vert^{\sigma-1}w(\tau,\cdot)\Vert_{L^{r_2}},\ \sigma>1\ \text{and}\ s>\lceil\sigma-1\rceil\ (\geq1),\label{corol2}\\
&\Vert\vert w(\tau,\cdot)\vert^s\Vert_{\dot{H}^{\sigma-2}}\lesssim\Vert w(\tau,\cdot)\Vert_{L^{r_3}}^{s-1}\Vert\vert D\vert^{\sigma-2}w(\tau,\cdot)\Vert_{L^{r_4}},\ \sigma>2\  \text{and}\ s>\lceil\sigma-2\rceil\ (\geq1),\label{corol3}
\end{align}
where  $\frac{s-1}{r_1}+\frac{1}{r_2}=\frac{s-1}{r_3}+\frac{1}{r_4}=\frac{1}{2}$.

(2) For any $r\geq1$, it holds that
\begin{align}
&\Vert w(\tau,\cdot)\Vert_{L^r}\lesssim \Vert w(\tau,\cdot)\Vert_{\dot{H}^\sigma}^{\theta_1(r)}\Vert w(\tau,\cdot)\Vert_{L^2}^{1-\theta_1(r)},\label{corol1}\\
&\Vert \vert D\vert^{\sigma-1} w(\tau,\cdot)\Vert_{L^{r}}\lesssim\Vert w(\tau,\cdot)\Vert_{\dot{H}^\sigma}^{\theta_2(r)}\Vert w(\tau,\cdot)\Vert_{L^2}^{1-\theta_2(r)}\ \text{for} \ \sigma>1,\label{corol22}\\
&\Vert\vert D\vert^{\sigma-2}w(\tau,\cdot)\Vert_{L^{r}}\lesssim\Vert w(\tau,\cdot)\Vert_{\dot{H}^\sigma}^{\theta_3(r)}\Vert w(\tau,\cdot)\Vert_{L^2}^{1-\theta_3(r)}\ \text{for} \ \sigma>2\label{corol23},
\end{align}
where
\begin{align}
&\theta_1(r)=(\frac{1}{2}-\frac{1}{r})\frac{n}{\sigma}\in[0,1],\label{theta1}\\
&\theta_2(r)=\bigl(\frac{1}{2}-\frac{1}{r}+\frac{\sigma-1}{n}\bigr)\frac{n}{\sigma}\in[\frac{\sigma-1}{\sigma},1],\label{theta2}\\
&\theta_3(r)=(\frac{1}{2}-\frac{1}{r}+\frac{\sigma-2}{n})\frac{n}{\sigma}\in[\frac{\sigma-2}{\sigma},1].\label{theta3}
\end{align}
\end{lemma}
\begin{remark}
A detailed proof of Lemma \ref{coro4.3} was given in Corollary 3.3  \cite{LiGuo2025}, and therefore it will not be repeated here.
\end{remark}
\begin{lemma}\label{foures}
In the case of $\delta_1,\delta_2\geq(m+1)^2(n+2\sigma-1)^2$ with $\sigma\geq1$, $p,q>1$ , for any $1\leq\tau\leq t\leq T$ and any $(u,v),(\tilde{u},\tilde{v})\in \mathcal{X}(T)$, we have
\begin{align}
&\Vert \vert v(\tau,\cdot)\vert^p\Vert_{L^1}\lesssim \tau^{(m+1)n+(-(m+1)n-\beta_2+1)p}\ell_2(\tau)^{p\theta_1(2p)}M_2(\tau,v)^p,\label{vpl1}\\
&\Vert \vert v(\tau,\cdot)\vert^p\Vert_{L^2}\lesssim\tau^{(m+1)\frac{n}{2}+(-(m+1)n-\beta_2+1)p}\ell_2
(\tau)^{p\theta_1(2p)}M_2(\tau,v)^p,\label{vpl2}\\
&\Vert \vert v(\tau,\cdot)\vert^p\Vert_{\dot{H}^{\sigma-1}}\lesssim \tau^{(m+1)(\frac{n}{2}-\sigma+1)+(-(m+1)n-\beta_2+1)p}\ell_2
(\tau)^{\theta_1(r_1)(p-1)+\theta_2(r_2)}M_2(\tau,v)^{p},\label{vph-1}
\end{align}
and if $\sigma>2$,
\begin{align}
&\Vert \vert v(\tau,\cdot)\vert^p\Vert_{\dot{H}^{\sigma-2}}\lesssim \tau^{(m+1)(\frac{n}{2}-\sigma+2)+(-(m+1)n-\beta_2+1)p}\ell_2
(\tau)^{\theta_1(r_3)(p-1)+\theta_3(r_4)}M_2(\tau,v)^{p},\label{vph-2}
\end{align}
where
$\frac{p-1}{r_1}+\frac{1}{r_2}=\frac{p-1}{r_3}+\frac{1}{r_4}=\frac{1}{2}$, $\beta_2$ is defined by (\ref{E}).

As for $u(\tau,x)$, we have
\begin{align}
&\Vert \vert u(\tau,\cdot)\vert^q\Vert_{L^1}\lesssim \tau^{(m+1)n+(-(m+1)n-\beta_1+1)q}\ell_1(\tau)^{q\theta_1(2q)}M_1(\tau,u)^q,\label{uql11}\\
&\Vert \vert u(\tau,\cdot)\vert^q\Vert_{L^2}\lesssim\tau^{(m+1)\frac{n}{2}+(-(m+1)n-\beta_1+1)q}\ell_1
(\tau)^{q\theta_1(2q)}M_1(\tau,u)^q,\label{uql2}\\
&\Vert \vert u(\tau,\cdot)\vert^q\Vert_{\dot{H}^{\sigma-1}}\lesssim \tau^{(m+1)(\frac{n}{2}-\sigma+1)+(-(m+1)n-\beta_1+1)q}\ell_1
(\tau)^{\theta_1(r_5)(q-1)+\theta_2(r_6)}M_1(\tau,u)^{q},\label{uqh-1}
\end{align}
and if $\sigma>2$,
\begin{align}
&\Vert \vert u(\tau,\cdot)\vert^q\Vert_{\dot{H}^{\sigma-2}}\lesssim \tau^{(m+1)(\frac{n}{2}-\sigma+2)+(-(m+1)n-\beta_1+1)q}\ell_1
(\tau)^{\theta_1(r_7)(q-1)+\theta_3(r_8)}M_2(\tau,u)^{q}.\label{uqh-2}
\end{align}
where
$\frac{q-1}{r_5}+\frac{1}{r_6}=\frac{q-1}{r_7}+\frac{1}{r_8}=\frac{1}{2}$, $\beta_1$ is defined by (\ref{E}).
\end{lemma}
\begin{remark}\label{reasonablness}
We point out that the numbers $r_i (i=1,\cdots,8)$ in Lemma \ref{foures} are well-defined under the assumptions of Theorem \ref{theorem11}. About $p, r_1, r_2$, we need to ensure that
\begin{equation*}
\theta_1(2p)=(\frac{1}{2}-\frac{1}{2p})\frac{n}{\sigma}\in [0,1]\ \ \
\text{and} \ \
\left\{
\begin{aligned}
&\theta_1(r_1)=(\frac{1}{2}-\frac{1}{r_1})\frac{n}{\sigma}\in [0,1],\\
&\theta_2(r_2)=(\frac{1}{2}-\frac{1}{r_2}+\frac{\sigma-1}{n})\frac{n}{\sigma}\in [\frac{\sigma-1}{\sigma},1],
\end{aligned}
\right.
\end{equation*}
which are equal to \begin{equation}\label{condi22}
\left\{
\begin{aligned}
&p\geq1,\ &n\leq2\sigma,\\
&1\leq p\leq\frac{n}{n-2\sigma},\ & n>2\sigma,
\end{aligned}
\right.
\end{equation}

\begin{equation}\label{r1r2}
\left\{
\begin{aligned}
&\frac{1}{r_1}\in (0,\frac{1}{2}],\ &n\leq2\sigma,\\
&\frac{1}{r_1}\in[\frac{n-2\sigma}{2n},\frac{1}{2}],\ & n>2\sigma,
\end{aligned}
\right.
\quad\text{and}\
\left\{
\begin{aligned}
&\frac{1}{r_2}\in (0,\frac{1}{2}],\ & n\leq2,\\
&\frac{1}{r_2}\in[\frac{n-2}{2n},\frac{1}{2}],\ & n>2.
\end{aligned}
\right.
\end{equation}
Clearly, (\ref{condi22}) holds  under the assumptions regarding
$p$ in  Theorem \ref{theorem11}. Since  $\frac{p-1}{r_1}+\frac{1}{r_2}=\frac{1}{2}$  and $\sigma\geq1$ in (\ref{vph-1}), in order to ensure the validity of (\ref{r1r2}), it suffices to guarantee
\begin{equation*}
\frac{1}{2}\in
\left
\{
\begin{aligned}
&(0,\frac{p}{2}],\ &\quad& n\leq2,\\
&\big(\frac{1}{2}-\frac{1}{n},\frac{p}{2}\big],\ &\quad& 2<n\leq2\sigma,\\
&\big[\frac{(n-2\sigma)(p-1)+n-2}{2n},\frac{p}{2}\big], \ &\quad& n>2\sigma,
\end{aligned}
\right.
\end{equation*}
the first two cases are automatically satisfied for $p\geq1$,  and  the condition $p\leq 1+\frac{2}{n-2\sigma}\ (n>2\sigma)$ in Theorem \ref{theorem11} gives the validity of the third case. The existence of $r_5, r_6$ can be obtained through the same derivation process.

As for $r_3,r_4$, we need to ensure that $\theta_1(r_3)\in[0,1], \theta_3(r_4)\in[0,1]$,  which are equivalent to ensuring
\begin{equation}\label{r3r4}
\left\{
\begin{aligned}
&\frac{1}{r_3}\in (0,\frac{1}{2}],\ &n\leq2\sigma,\\
&\frac{1}{r_3}\in[\frac{n-2\sigma}{2n},\frac{1}{2}],\ & n>2\sigma,
\end{aligned}
\right.
\quad\text{and}\
\left\{
\begin{aligned}
&\frac{1}{r_4}\in (0,\frac{1}{2}],\ & n\leq4,\\
&\frac{1}{r_4}\in[\frac{n-4}{2n},\frac{1}{2}],\ & n>4.
\end{aligned}
\right.
\end{equation}
Note that $\sigma\geq2$ in (\ref{vph-2}) and $\frac{p-1}{r_3}+\frac{1}{r_4}=\frac{1}{2}$, it suffices to verify that
\begin{equation*}
\frac{1}{2}\in
\left
\{
\begin{aligned}
&(0,\frac{p}{2}],\ &\quad& n\leq4,\\
&\big(\frac{1}{2}-\frac{2}{n},\frac{p}{2}\big],\ &\quad& 4<n\leq2\sigma,\\
&\big[\frac{(n-2\sigma)(p-1)+n-4}{2n},\frac{p}{2}\big], \ &\quad& n>2\sigma.
\end{aligned}
\right.
\end{equation*}
Hence,   it is sufficient to guarantee $p\leq1+\frac{4}{n-2\sigma}$ (when $n>2\sigma$). Obviously, the condition $p\leq 1+\frac{2}{n-2\sigma}\ (n>2\sigma)$ in Theorem \ref{theorem11} can ensure this. The method of demonstrating the existence of $r_7, r_8$  is the same as that for $r_3, r_4$.

\end{remark}

\begin{proof}[Proof of Lemma \ref{foures}]
By (\ref{corol1}) in Lemma \ref{coro4.3} and the definition (\ref{M2}) of $M_2(\tau,v)$, we  obtain
\begin{align}
&\Vert \vert v(\tau,\cdot)\vert^p\Vert_{L^2}=\Vert v(\tau,\cdot)\Vert_{L^{2p}}^p\lesssim \Vert v(\tau,\cdot)\Vert_{\dot{H}^{\sigma}}^{p\theta_1(2p)}\Vert v(\tau,\cdot)\Vert_{\dot{L}^{2}}^{p(1-\theta_1(2p))}\notag\\
&\lesssim \big(\tau^{-(m+1)(\sigma+\frac{n}{2})+\frac{\sqrt{\delta_2}-\mu_2+1}{2}}\ell_2(\tau)
M_2(\tau,v)\big)^{p\theta_1(2p)}\big(\tau^{-(m+1)\frac{n}{2}+\frac{\sqrt{\delta_2}
-\mu_2+1}{2}}M_2(\tau,v)\big)^{p(1-\theta_1(2p))}\notag\\
&\quad=\tau^{\big(-(m+1)\frac{n}{2}+\frac{\sqrt{\delta_2}-\mu_2+1}{2}\big)p-(m+1)\sigma p\theta_1(2p)}\ell_2(\tau)^{p\theta_1(2p)}M_2(\tau,v)^p\notag\\
&\quad=\tau^{\frac{(m+1)n}{2}+(-(m+1)n-\beta_2+1)p}\ell_2(\tau)^{p\theta_1(2p)}M_2(\tau,v)^p.\notag
\end{align}
It follows from H$\ddot{\text{o}}$lder's inequality that
\begin{align}
&\Vert \vert v(\tau,\cdot)\vert^p\Vert_{L^1}\leq\Vert \vert v(\tau,\cdot)\vert^p\Vert_{L^{2}}\big(\int_{\vert x\vert\leq \phi_m(\tau)-\phi_m(1)+M}dx\big)^{\frac{1}{2}}\lesssim \tau^{\frac{(m+1)n}{2}}\Vert \vert v(\tau,\cdot)\vert\Vert_{L^{2}}\notag\\
&\lesssim\tau^{(m+1)n+(-(m+1)n-\beta_2+1)p}\ell_2
(\tau)^{p\theta_1(2p)}M_2(\tau,v)^p.\notag
\end{align}
By (\ref{corol2}), (\ref{corol1})-(\ref{corol22}) in Lemma \ref{coro4.3}, we have 
\begin{align}
&\Vert \vert v(\tau,\cdot)\vert^p\Vert_{\dot{H}^{\sigma-1}}\lesssim\Vert  v(\tau,\cdot)\Vert_{L^{r_1}}^{p-1}\Vert\vert D\vert^{\sigma-1}v(\tau,\cdot)\Vert_{L^{r_2}}\notag\\
&\lesssim \Vert  v(\tau,\cdot)\Vert_{\dot{H}^{\sigma}}^{\theta_1(r_1)(p-1)+\theta_2(r_2)}\Vert  v(\tau,\cdot)\Vert_{L^{2}}^{(1-\theta_1(r_1))(p-1)+1-\theta_2(r_2)}\notag\\
&\lesssim \tau^{\big(-(m+1)\frac{n}{2}+\frac{\sqrt{\delta_2}-\mu_2+1}{2}\big)p-(m+1)\sigma(\theta_1(r_1)(p-1)+\theta_2(r_2))}\ell_2
(\tau)^{\theta_1(r_1)(p-1)+\theta_2(r_2))}M_2(\tau,v)^{p}\notag\\
&=\tau^{(-(m+1)n-\beta_2+1)p+(m+1)(\frac{n}{2}-\sigma+1)}\ell_2
(\tau)^{\theta_1(r_1)(p-1)+\theta_2(r_2))}M_2(\tau,v)^{p};\notag
\end{align}
while if $\sigma>2$, using (\ref{M2}), (\ref{corol3})-(\ref{corol1}) and (\ref{corol23}), we immediately obtain (\ref{vph-2}). 

Using the same argument, (\ref{uql11})-(\ref{uqh-2}) can be derived from  Lemma \ref{coro4.3} and the definition (\ref{M1}) of $M_1(\tau,u)$, so we omit the details.
\end{proof}

With the preliminary results of Lemma \ref{foures}, we now proceed to prove (\ref{nones}).
\begin{proof}[Proof of (\ref{nones})]

Case 1: $\delta_1,\delta_2>\big((m+1)(n+2\sigma-1)\big)^2$ with $\sigma>1$. The condition  $p>\tilde{p}$ yields
\begin{equation}\label{integralexponent1}
\begin{aligned}
(m+1)n+\beta_1+(-(m+1)n-\beta_2+1)p<-1,
\end{aligned}
\end{equation}
then
\begin{equation}\label{infite1}
\begin{aligned}
&\int_1^t\tau^{(m+1)n+\beta_1+(-(m+1)n-\beta_2+1)p}d\tau\leq\int_1^\infty\tau^{(m+1)n+\beta_1
+(-(m+1)n-\beta_2+1)p}d\tau<\infty.
\end{aligned}
\end{equation}
By (\ref{u0=0u}) in Proposition \ref{unlinearestimate}, (\ref{vpl1})-(\ref{vpl2}), (\ref{infite1}) and the Duhamel's principle,  we have
\begin{equation}\label{esl2}
\begin{aligned}
&\Vert H_1(v)(t,\cdot)\Vert_{L^2}\lesssim \int_1^t\Vert E_1^{\mu_1,\nu_1}(t,\tau,\cdot)\ast\vert v(\tau,\cdot)\vert^p\Vert_{L^2}d\tau\\
&\lesssim t^{-(m+1)\frac{n}{2}+\frac{\sqrt{\delta_1}-\mu_1+1}{2}}\int_1^t\tau^
{-\frac{\sqrt{\delta_1}-\mu_1-1}{2}}\big(\Vert \vert v(\tau,\cdot)\vert^p\Vert_{L^1}+\tau^{\frac{(m+1)n}{2}}\Vert \vert v(\tau,\cdot)\vert^p\Vert_{L^2}\big)d\tau\\
&\lesssim t^{-(m+1)\frac{n}{2}+\frac{\sqrt{\delta_1}-\mu_1+1}{2}}\int_1^t\tau^
{(m+1)n+\beta_1+(-(m+1)n-\beta_2+1)p}M_2(\tau,v)^pd\tau\\
&\lesssim t^{-(m+1)\frac{n}{2}+\frac{\sqrt{\delta_1}-\mu_1+1}{2}}\Vert (u,v)\Vert_{\mathcal{X}(t)}^p\int_1^t\tau^
{(m+1)n+\beta_1+(-(m+1)n-\beta_2+1)p}d\tau\\
&\lesssim t^{-(m+1)\frac{n}{2}+\frac{\sqrt{\delta_1}-\mu_1+1}{2}}\Vert (u,v)\Vert_{\mathcal{X}(t)}^p.
\end{aligned}
\end{equation}
Using (\ref{u0=0ut}) and the same argument  as in deriving (\ref{esl2}), we get
\begin{equation}\label{partialtesl2}
\Vert\partial_tH_1(v)(t,\cdot)\Vert_{L^2}\lesssim t^{m-(m+1)(1+\frac{n}{2})+\frac{\sqrt{\delta_1}-\mu_1+1}{2}}\Vert (u,v)\Vert_{\mathcal{X}(t)}^p.
\end{equation}
Applying (\ref{u0=0u}), (\ref{vpl1})-(\ref{vph-1}) yields
\begin{align}
&\Vert H_1(v)(t,\cdot)\Vert_{\dot{H}^{\sigma}}\lesssim \int_1^t\Vert E_1^{\mu_1,\nu_1}(t,\tau,\cdot)\ast\vert v(t,\cdot)\vert^p\Vert_{\dot{H}^{\sigma}}d\tau\notag\\
&\lesssim t^{-(m+1)(\sigma+\frac{n}{2})+\frac{\sqrt{\delta_1}-\mu_1-1}{2}}
\int_1^t\tau^{-\frac{\sqrt{\delta_1}-\mu_1-1}{2}}\big(\Vert\vert v(\tau,\cdot)\vert^p\Vert_{L^1}+\tau^{(m+1)(\frac{n}{2}+\sigma-1)}\Vert\vert v(\tau,\cdot)\vert^p\Vert_{\dot{H}^{\sigma-1}}\big)d\tau\notag\\
&\lesssim t^{-(m+1)(\sigma+\frac{n}{2})+\frac{\sqrt{\delta_1}-\mu_1-1}{2}}\int_1^t\tau^
{(m+1)n+\beta_1+(-(m+1)n-\beta_2+1)p}M_2(\tau,v)^pd\tau\notag\\
&\lesssim t^{-(m+1)(\sigma+\frac{n}{2})+\frac{\sqrt{\delta_1}-\mu_1-1}{2}}\Vert (u,v)\Vert_{\mathcal{X}(t)}^p.\label{h1vhsigma}
\end{align}
To estimate $\Vert\partial_tH_1(v)(t,\cdot)\Vert_{\dot{H}^{\sigma-1}}$, if $\sigma>2$, then by (\ref{u0=0ut}), (\ref{vpl1})-(\ref{vph-2}), we have
\begin{align}
&\Vert\partial_tH_1(v)(t,\cdot)\Vert_{\dot{H}^{\sigma-1}}\lesssim t^{m-(m+1)(\sigma+\frac{n}{2})+\frac{\sqrt{\delta_1}-\mu_1-1}{2}}\int_1^t\tau^{-\frac
{\sqrt{\delta_1}-\mu_1-1}{2}}\big(\Vert\vert v(\tau,\cdot)\vert^p\Vert_{L^1}\notag\\
&\quad\quad\quad\quad+\tau^{(m+1)(\frac{n}{2}+\sigma-1)}\Vert\vert v(\tau,\cdot)\vert^p\Vert_{\dot{H}^{\sigma-1}}+\tau^{(m+1)(\frac{n}{2}+\sigma-2)}\Vert\vert v(\tau,\cdot)\vert^p\Vert_{\dot{H}^{\sigma-2}}\big)d\tau\notag\\
&\lesssim t^{m-(m+1)(\sigma+\frac{n}{2})+\frac{\sqrt{\delta_1}-\mu_1-1}{2}}\int_1^t\tau^
{(m+1)n+\beta_1+(-(m+1)n-\beta_2+1)p}M_2(\tau,v)^pd\tau\notag\\
&\lesssim t^{m-(m+1)(\sigma+\frac{n}{2})+\frac{\sqrt{\delta_1}-\mu_1-1}{2}}\Vert (u,v)\Vert_{\mathcal{X}(t)}^p;\label{partialtuh-11}
\end{align}
while for $1<\sigma\leq 2$, it follows from  (\ref{u0=0ut}), (\ref{vpl1})- (\ref{vph-1}) that
\begin{equation}\label{partialtuh-12}
\begin{aligned}
\Vert\partial_tH_1(v)(t,\cdot)\Vert_{\dot{H}^{\sigma-1}} \lesssim t^{m-(m+1)(\sigma+\frac{n}{2})+\frac{\sqrt{\delta_1}-\mu_1-1}{2}}\Vert (u,v)\Vert_{\mathcal{X}(t)}^p.
\end{aligned}
\end{equation}
Combining  (\ref{esl2})-(\ref{partialtuh-12}) and the definition (\ref{M1}) of $M_1(t,H_1(v))$, we get
\begin{equation}\label{m1h1v}
M_1(t,H_1(v))\lesssim \Vert (u,v)\Vert_{\mathcal{X}(t)}^p.
\end{equation}

Since $q>\tilde{q}$, it can be inferred  that
\begin{equation}\label{integralexponent2}
\begin{aligned}
(m+1)n+\beta_2+(-(m+1)n-\beta_1+1)q<-1,
\end{aligned}
\end{equation}
then
\begin{equation}\label{infite2}
\int_1^t\tau^{(m+1)n+\beta_2+(-(m+1)n-\beta_1+1)q}d\tau\leq\int_1^\infty\tau^{(m+1)n+\beta_2
+(-(m+1)n-\beta_1+1)p}d\tau<\infty.
\end{equation}
From (\ref{u0=0u}), (\ref{uql11})-(\ref{uql2}) and (\ref{infite2}), we conclude that
\begin{align*}
\resizebox{0.95\hsize}{!}{$
\begin{aligned}
&\Vert H_2(u)(t,\cdot)\Vert_{L^2}\lesssim t^{-(m+1)\frac{n}{2}+\frac{\sqrt{\delta_2}-\mu_2+1}{2}}\int_1^t\tau^{-\frac{\sqrt
{\delta_2}-\mu_2-1}{2}}\big(\Vert \vert u(\tau,\cdot)\vert^q\Vert_{L^1}+\tau^{(m+1)\frac{n}{2}}\Vert \vert u(\tau,\cdot)\vert^q\Vert_{L^2}\big)d\tau
\end{aligned}$}
\end{align*}
\begin{align}
&\lesssim t^{-(m+1)\frac{n}{2}+\frac{\sqrt{\delta_2}-\mu_2+1}{2}}\int_1^t\tau^{(m+1)n+\beta_2+
(-(m+1)n-\beta_1+1)q}M_1(\tau,u)^qd\tau\notag\\
&\lesssim t^{-(m+1)\frac{n}{2}+\frac{\sqrt{\delta_2}-\mu_2+1}{2}}\Vert (u,v)\Vert_{\mathcal{X}(t)}^q\int_1^t\tau^{(m+1)n+\beta_2+
(-(m+1)n-\beta_1+1)q}d\tau\notag\\
&\lesssim t^{-(m+1)\frac{n}{2}+\frac{\sqrt{\delta_2}-\mu_2+1}{2}}\Vert (u,v)\Vert_{\mathcal{X}(t)}^q.\label{h2ul2}
\end{align}
The estimates for $\Vert \partial_tH_2(u)(t,\cdot)\Vert_{L^2}$, $\Vert H_2(u)(t,\cdot)\Vert_{\dot{H}^{\sigma}}$, $\Vert \partial_tH_2(u)(t,\cdot)\Vert_{\dot{H}^{\sigma-1}}$ can be derived by using (\ref{u0=0u})-(\ref{u0=0ut}), (\ref{uql11})-(\ref{uqh-2}) and (\ref{infite2}), along with the same argument as applied in obtaining (\ref{partialtesl2})-(\ref{partialtuh-12}). Hence,  we present the results as follows and omit the details
\begin{align}
&\Vert\partial_tH_2(u)(t,\cdot)\Vert_{L^2}\lesssim t^{m-(m+1)(1+\frac{n}{2})+\frac{\sqrt{\delta_2}-\mu_2+1}{2}}\Vert (u,v)\Vert_{\mathcal{X}(t)}^q,\label{partialtvl2}\\
&\Vert H_2(u)(t,\cdot)\Vert_{\dot{H}^{\sigma}}\lesssim t^{-(m+1)(\sigma+\frac{n}{2})+\frac{\sqrt{\delta_2}-\mu_2+1}{2}}\Vert (u,v)\Vert_{\mathcal{X}(t)}^q,\label{vhsigma}\\
&\Vert \partial_tH_2(u)(t,\cdot)\Vert_{\dot{H}^{\sigma-1}}\lesssim
t^{m-(m+1)(\sigma+\frac{n}{2})+\frac{\sqrt{\delta_2}-\mu_2+1}{2}}\Vert (u,v)\Vert_{\mathcal{X}(t)}^q.\label{partialtvhsigma-1}
\end{align}
It follows (\ref{h2ul2})-(\ref{partialtvhsigma-1}) that
\begin{equation}\label{m2h2u}
M_2(t,H_2(u))\lesssim \Vert (u,v)\Vert_{\mathcal{X}(t)}^q.
\end{equation}
Therefore, we can immediately derive (\ref{nones}) by combining (\ref{m1h1v}) and (\ref{m2h2u}).

Case 2: $\delta_1>\bigl((m+1)(n+2\sigma-1)\bigr)^2,\ \delta_2=\bigl((m+1)(n+2\sigma-1)\bigr)^2,\ \sigma>1$.
Clearly, $\ell_1(\tau)=1$ and $\ell_2(\tau)=(1+\log \tau)^{\frac{1}{2}}$, so by Proposition \ref{unlinearestimate} and Lemma \ref{foures}, we have
\begin{align}
&\Vert H_1(v)(t,\cdot)\Vert_{L^2}\lesssim t^{-(m+1)\frac{n}{2}+\frac{\sqrt{\delta_1}-\mu_1+1}{2}}\notag\\
&\quad\quad\quad\quad \times\int_1^t\tau^{(m+1)n+\beta_1+
(-(m+1)n-\beta_2+1)p}(1+\log\tau)^{\frac{p\theta_1(2p)}{2}}M_2(\tau,u)^pd\tau,\label{case211}\\
&\Vert \partial_tH_1(v)(t,\cdot)\Vert_{L^2}\lesssim t^{m-(m+1)(1+\frac{n}{2})+\frac{\sqrt{\delta_1}-\mu_1+1}{2}}\notag\\
&\quad\quad\quad\quad\times\int_1^t\tau^{(m+1)n+\beta_1+
(-(m+1)n-\beta_2+1)p}(1+\log\tau)^{\frac{p\theta_1(2p)}{2}}M_2(\tau,v)^pd\tau,\label{case212}\\
&\Vert H_1(v)(t,\cdot)\Vert_{\dot{H}^\sigma}\lesssim t^{-(m+1)(\sigma+\frac{n}{2})+\frac{\sqrt{\delta_1}-\mu_1+1}{2}}\notag\\
&\ \ \times\int_1^t\tau^{(m+1)n+\beta_1+
(-(m+1)n-\beta_2+1)p}(1+\log\tau)^{\frac{p\theta_1(2p)+\theta_1(r_1)(p-1)
+\theta_2(r_2)}{2}}M_2(\tau,v)^pd\tau,\label{case213}
\end{align}
if $\sigma>2$,
\begin{equation}\label{case2141}
\begin{aligned}
&\Vert \partial_tH_1(v)(t,\cdot)\Vert_{\dot{H}^{\sigma-1}}\lesssim t^{m-(m+1)(\sigma+\frac{n}{2})+\frac{\sqrt{\delta_1}-\mu_1+1}{2}}\int_1^t\tau^{(m+1)n+\beta_1+
(-(m+1)n-\beta_2+1)p}\\
&\ \ \times(1+\log\tau)^{\frac{p\theta_1(p)+\theta_1(r_1)(p-1)+\theta_2(r_2)+\theta_1(r_3)(p-1)+\theta_3(r_4)}{2}}M_2(\tau,v)^pd\tau,
\end{aligned}
\end{equation}
while for $1<\sigma\leq 2$,
\begin{equation}\label{case2142}
\begin{aligned}
&\Vert \partial_tH_1(v)(t,\cdot)\Vert_{\dot{H}^{\sigma-1}}\lesssim t^{m-(m+1)(\sigma+\frac{n}{2})+\frac{\sqrt{\delta_1}-\mu_1+1}{2}}\int_1^t\tau^{(m+1)n+\beta_1+
(-(m+1)n-\beta_2+1)p}\\
&\quad\quad\quad\quad \quad\quad\quad\quad \times(1+\log\tau)^{\frac{p\theta_1(2p)+\theta_2(r_2)+\theta_1(r_1)(p-1)}{2}}
M_2(\tau,v)^pd\tau.
\end{aligned}
\end{equation}
Since $(m+1)n+\beta_1+(-(m+1)n-\beta_2+1)p<-1,$ we can set $(m+1)n+\beta_1+(-(m+1)n-\beta_2+1)p=-1-\gamma$. Then $\tau^{-\frac{\gamma}{2}}(1+\log{\tau})^{\frac{p\theta_1(2p)}{2}}\lesssim 1$ holds for large $\tau$, so
\begin{equation}\label{ingre}
\begin{aligned}
&\int_1^t\tau^{(m+1)n+\beta_1+
(-(m+1)n-\beta_2+1)p}(1+\log\tau)^{\frac{\theta_1(p)+\theta_1(2p)}{2}p}d\tau\lesssim
\int_1^\infty\tau^{-1-\frac{\gamma}{2}}d\tau<\infty.
\end{aligned}
\end{equation}
Thus from (\ref{case211}), we infer that
\begin{equation}\label{h1vl2}
\Vert H_1(v)(t,\cdot)\Vert_{L^2}\lesssim t^{-(m+1)\frac{n}{2}+\frac{\sqrt{\delta_1}-\mu_1+1}{2}}\Vert(u,v)\Vert_{\mathcal{X}(t)}^p.
\end{equation}
Similarly, by (\ref{case212})-(\ref{case2142}), we obtain
\begin{align}
&\Vert \partial_tH_1(v)(t,\cdot)\Vert_{L^2}\lesssim t^{m-(m+1)(1+\frac{n}{2})+\frac{\sqrt{\delta_1}-\mu_1+1}{2}}\Vert(u,v)\Vert_{\mathcal{X}(t)}
^p,\label{case2121}\\
&\Vert H_1(v)(t,\cdot)\Vert_{\dot{H}^\sigma}\lesssim t^{-(m+1)(\sigma+\frac{n}{2})+\frac{\sqrt{\delta_1}-\mu_1+1}{2}}
\Vert(u,v)\Vert_{\mathcal{X}(t)}^p,\label{case3111}\\
&\Vert \partial_tH_1(v)(t,\cdot)\Vert_{\dot{H}^{\sigma-1}}\lesssim t^{m-(m+1)(\sigma+\frac{n}{2})+\frac{\sqrt{\delta_1}-\mu_1+1}{2}}
\Vert(u,v)\Vert_{\mathcal{X}(t)}^p.\label{case4111}
\end{align}
Hence, (\ref{h1vl2})-(\ref{case4111}) show that
\begin{equation}\label{M1tv1}
M_1(t,H_1(v))\lesssim \Vert(u,v)\Vert_{\mathcal{X}(t)}^p.
\end{equation}

In view of $\delta_2=(m+1)^2(n+2\sigma-1)^2$ and $(m+1)n+\beta_2+(-(m+1)n-\beta_1+1)q<-1,$  we see that $\frac{\sqrt{\delta_2}}{2(m+1)}+\frac{1}{2}-\frac{n}{2}=\sigma>1$ and $\int_1^t\tau^{(m+1)n+\beta_2+
(-(m+1)n-\beta_1+1)q}d\tau<\infty$, then by Proposition \ref{linearestimate} and Lemma \ref{foures}, we have

\begin{align}
&\Vert H_2(u)(t,\cdot)\Vert_{L^2}\lesssim t^{-(m+1)\frac{n}{2}+\frac{\sqrt{\delta_2}-\mu_2+1}{2}}\int_1^t\tau^{(m+1)n+\beta_2+
(-(m+1)n-\beta_1+1)q}M_1(\tau,u)^qd\tau\notag\\
&\quad\quad\quad\quad\quad\lesssim t^{-(m+1)\frac{n}{2}+\frac{\sqrt{\delta_2}-\mu_2+1}{2}}\Vert(u,v)\Vert_{\mathcal{X}(t)}
^q,\label{case21}\\
&\Vert \partial_tH_2(u)(t,\cdot)\Vert_{L^2}
\lesssim t^{m-(m+1)(1+\frac{n}{2})+\frac{\sqrt{\delta_2}-\mu_2+1}{2}}\Vert(u,v)\Vert
_{\mathcal{X}(t)}^q,\label{case22}\\
&\Vert H_2(u)(t,\cdot)\Vert_{\dot{H}^\sigma}\lesssim t^{-\frac{\mu_2+m}{2}}\int_1^t(1+\log\frac{t}{\tau})^{\frac{1}{2}}
\tau^{(m+1)n+\beta_2+(-(m+1)n-\beta_1+1)q}M_1(\tau,u)^qd\tau\notag\\
&\quad\lesssim t^{-\frac{\mu_2+m}{2}}(1+\log t)^{\frac{1}{2}}\int_1^t\tau^{(m+1)n+\beta_2+(-(m+1)n-\beta_1+1)q}M_1(\tau,u)^qd\tau\notag\\
&\quad=t^{-(m+1)(\sigma+\frac{n}{2})+\frac{\sqrt{\delta_2}-\mu_2+1}{2}}(1+\log t)^{\frac{1}{2}}\int_1^t\tau^{(m+1)n+\beta_2+(-(m+1)n-\beta_1+1)q}
M_1(\tau,u)^qd\tau\notag\\
&\quad\lesssim t^{-(m+1)(\sigma+\frac{n}{2})+\frac{\sqrt{\delta_2}-\mu_2+1}{2}}(1+\log t)^{\frac{1}{2}}\Vert(u,v)\Vert
_{\mathcal{X}(t)}^q\label{case23}
\end{align}
and
\begin{equation}\label{case24}
\begin{aligned}
&\Vert \partial_tH_2(u)(t,\cdot)\Vert_{\dot{H}^{\sigma-1}}\lesssim t^{m-\frac{\mu_2+m}{2}}(1+\log t)^{\frac{1}{2}}\Vert(u,v)\Vert
_{\mathcal{X}(t)}^q\\
&\quad\quad\quad\quad\quad\quad=t^{m-(m+1)(\sigma+\frac{n}{2})+\frac{\sqrt{\delta_2}-\mu_2+1}{2}}(1+\log t)^{\frac{1}{2}}\Vert(u,v)\Vert
_{\mathcal{X}(t)}^q.
\end{aligned}
\end{equation}
It follows from (\ref{case21})-(\ref{case24}) that
\begin{equation}\label{M2tves2}
M_2(t,H_2(u))\lesssim \Vert (u,v)\Vert_{\mathcal{X}(t)}^q.
\end{equation}
Consequently, we conclude that (\ref{nones}) holds in this case by (\ref{M1tv1}) and (\ref{M2tves2}).

The remaining cases Case 3- Case 8 presented in (\ref{bigcase3})-(\ref{bigcase8})
can be handled in the same way as in Case 1 and Case 2. It is worth noting that  when $\sigma=1$, we should use  the definitions (\ref{M11}) and (\ref{M21}) of  $M_1(t,H_1(v))$ and $M_2(t,H_2(u))$ in place of (\ref{M1}) and (\ref{M2}), and we  will not elaborate further.
\end{proof}

\subsubsection{The proof of (\ref{key2})}
In Lemma 3.5  \cite{LiGuo2025}, we established the following preliminary results, which will be used throughout the whole proof of (\ref{key2}).
\begin{lemma}[ \cite{LiGuo2025}]\label{key2lemmaa}
For any $T>1$, $1\leq\tau\leq t\leq T$  and any $w, \tilde{w}\in \mathcal{X}(T)$, we have\\
(1) For any $s,r\geq1$,
\begin{equation}\label{u-tildeu}
\resizebox{0.91\hsize}{!}{$
\begin{aligned}
&\Vert\vert w(\tau,\cdot)\vert^s-\vert\tilde{w}(\tau,\cdot)\vert^s \Vert_{L^r}\lesssim\Vert w(\tau,\cdot)-\tilde{w}(\tau,\cdot)\Vert_{\dot{H}^\sigma}^{\theta_1(rs)}\Vert w(\tau,\cdot)-\tilde{w}(\tau,\cdot)\Vert_{L^2}^{1-\theta_1(rs)}\\
&\times\Big(\Vert w(\tau,\cdot)\Vert_{\dot{H}^\sigma}^{\theta_1(rs)(s-1)}\Vert w(\tau,\cdot)\Vert_{L^2}^{(1-\theta_1(rs))(s-1)}\Vert \tilde{w}(\tau,\cdot)\Vert_{\dot{H}^\sigma}^{\theta_1(rs)(s-1)}\Vert \tilde{w}(\tau,\cdot)\Vert_{L^2}^{(1-\theta_1(rs))(s-1)}\Bigr).
\end{aligned}$}
\end{equation}
(2)  If $\sigma>1$ and $s>\lceil \sigma-1\rceil+1\ (=\lceil \sigma\rceil\geq2)$,
\begin{align}
&\Vert\vert w(\tau,\cdot)\vert^s-\vert\tilde{w}(\tau,\cdot)\vert^s \Vert_{\dot{H}^{\sigma-1}}\lesssim\Vert w(\tau,\cdot)-\tilde{w}(\tau,\cdot)\Vert_{\dot{H}^\sigma}^{\theta_2(b_1)}\Vert w(\tau,\cdot)-\tilde{w}(\tau,\cdot)\Vert_{L^2}^{1-\theta_2(b_1)}\notag\\
&\times\Big(\Vert w(\tau,\cdot)\Vert_{\dot{H}^\sigma}^{(s-1)\theta_1((s-1)b_2)}\Vert w(\tau,\cdot)\Vert_{L^2}^{(s-1)(1-\theta_1((s-1)b_2))}\notag\\
&\quad\quad\quad+\Vert \tilde{w}(\tau,\cdot)\Vert_{\dot{H}^\sigma}^{(s-1)\theta_1((s-1)b_2)}\Vert \tilde{w}(\tau,\cdot)\Vert_{L^2}^{(s-1)(1-\theta_1((s-1)b_2))}\Bigr)\notag\\
&+\Vert w(\tau,\cdot)-\tilde{w}(\tau,\cdot)\Vert_{\dot{H}^\sigma}^{\theta_1(b_3)}\Vert w(\tau,\cdot)-\tilde{w}(\tau,\cdot)\Vert_{L^2}^{1-\theta_1(b_3)}\bigl(\Vert w(\tau,\cdot)\Vert_{\dot{H}^\sigma}+\Vert \tilde{w}(\tau,\cdot)\Vert_{\dot{H}^\sigma}\bigr)^{(s-2)\theta_1(b_5)+\theta_2(b_6)}\notag\\
&\quad\quad\times\bigl(\Vert w(\tau,\cdot)\Vert_{L^2}+\Vert \tilde{w}(\tau,\cdot)\Vert_{L^2}\bigr)^{(s-2)(1-\theta_1(b_5))+1-\theta_2(b_6)}.\label{u-tildeu2}
\end{align}
(3) If $\sigma>2$ and $s>\lceil\sigma-2\rceil+1\ (=\lceil\sigma-1\rceil\geq2)$,
\begin{align}
&\Vert\vert w(\tau,\cdot)\vert^s-\vert\tilde{w}(\tau,\cdot)\vert^s \Vert_{\dot{H}^{\sigma-2}}\lesssim\Vert w(\tau,\cdot)-\tilde{w}(\tau,\cdot)\Vert_{\dot{H}^\sigma}^{\theta_3(c_1)}\Vert w(\tau,\cdot)-\tilde{w}(\tau,\cdot)\Vert_{L^2}^{1-\theta_3(c_1)}\notag\\
&\quad\quad\times\Big(\Vert w(\tau,\cdot)\Vert_{\dot{H}^\sigma}^{(s-1)\theta_1((s-1)c_2)}\Vert w(\tau,\cdot)\Vert_{L^2}^{(s-1)(1-\theta_1((s-1)c_2))}\notag\\
&\quad\quad\quad\quad\quad\quad+\Vert \tilde{w}(\tau,\cdot)\Vert_{\dot{H}^\sigma}^{(s-1)\theta_1((s-1)c_2)}\Vert \tilde{w}(\tau,\cdot)\Vert_{L^2}^{(s-1)(1-\theta_1((s-1)c_2))}\Bigr)\notag\\
&+\Vert w(\tau,\cdot)-\tilde{w}(\tau,\cdot)\Vert_{\dot{H}^\sigma}^{\theta_1(c_3)}\Vert w(\tau,\cdot)-\tilde{w}(\tau,\cdot)\Vert_{L^2}^{1-\theta_1(c_3)}\times\bigl(\Vert w\Vert_{\dot{H}^\sigma}+\Vert \tilde{w}\Vert_{\dot{H}^\sigma}\bigr)^{(s-2)\theta_1(c_5)+\theta_3(c_6)}\notag\\
&\quad\quad\quad\times\bigl(\Vert w\Vert_{L^2}+\Vert \tilde{w}\Vert_{L^2}\bigr)^{(s-2)(1-\theta_1(c_5))+1-\theta_3(c_6)},\label{u-tildeu222}
\end{align}
where $\theta_1, \theta_2, \theta_3$ are defined in Lemma \ref{coro4.3}, and $b_i, c_i ,i=1,2,3,5,6$ satisfy
\begin{equation}
\begin{aligned}
&\frac{1}{b_1}+\frac{1}{b_2}=\frac{1}{2},\ \ \frac{1}{2}-\frac{1}{b_3}=\frac{s-2}{b_5}+\frac{1}{b_6},\\
&\frac{1}{c_1}+\frac{1}{c_2}=\frac{1}{2},\ \ \frac{1}{2}-\frac{1}{c_3}=\frac{s-2}{c_5}+\frac{1}{c_6}.
\end{aligned}
\end{equation}
\end{lemma}

\begin{lemma}\label{8cases}
In the case of $\delta_1,\delta_2\geq(m+1)^2(n+2\sigma-1)^2$ with $\sigma\geq1$,  for any $1\leq\tau\leq t\leq T$ and any $(u,v), (\tilde{u},\tilde{v})\in \mathcal{X}(T)$, we have
\begin{align}
&\Vert\vert v(\tau,\cdot)\vert^p-\vert \tilde{v}(\tau,\cdot)\vert^p\Vert_{L^1}\notag\\
&\lesssim\tau^{(m+1)n+(-(m+1)n-\beta_2+1)p}
\ell_2(\tau)^{p\theta_1(2p)}M_2(\tau,v-\tilde{v})
\big(M_2(\tau,v)^{p-1}+M_2(\tau,\tilde{v})^{p-1}\big),\label{8cases1}\\
&\Vert\vert v(\tau,\cdot)\vert^p-\vert \tilde{v}(\tau,\cdot)\vert^p\Vert_{L^2}\notag\\
&\lesssim\tau^{(m+1)\frac{n}{2}+(-(m+1)n-\beta_2+1)p}
\ell_2(\tau)^{p\theta_1(2p)}M_2(\tau,v-\tilde{v})
\big(M_2(\tau,v)^{p-1}+M_2(\tau,\tilde{v})^{p-1}\big);\label{8cases2}
\end{align}
For $p>\lceil \sigma\rceil\ (\geq2)$,
\begin{align}
&\Vert\vert v(\tau,\cdot)\vert^p-\vert \tilde{v}(\tau,\cdot)\vert^p\Vert_{\dot{H}^{\sigma-1}}\notag\\
&\lesssim\tau^{(m+1)
(\frac{n}{2}-\sigma+1)+(-(m+1)n-\beta_2+1)p}\ell_2(\tau)^{\theta_2(b_1)+(p-1)\theta_1((p-1)b_2)+\theta_1(b_3)+(p-2)
\theta_1(b_5)+\theta_2(b_6)}\notag\\
&\quad\times M_2(\tau,v-\tilde{v})
\big(M_2(\tau,v)^{p-1}+M_2(\tau,\tilde{v})^{p-1}\big);\label{8cases3}
\end{align}
If $\sigma>2$ and for $p>\lceil \sigma-1\rceil\ (\geq2)$,
\begin{align}
&\Vert\vert v(\tau,\cdot)\vert^p-\vert \tilde{v}(\tau,\cdot)\vert^p\Vert_{\dot{H}^{\sigma-2}}\notag\\
&\lesssim\tau^{(m+1)
(\frac{n}{2}-\sigma+2)+(-(m+1)n-\beta_2+1)p}\ell_2(\tau)^{\theta_3(c_1)+(p-1)\theta_1((p-1)c_2)+\theta_1(c_3)+(p-2)
\theta_1(c_5)+\theta_3(c_6)}\notag\\
&\quad\times M_2(\tau,v-\tilde{v})
\big(M_2(\tau,v)^{p-1}+M_2(\tau,\tilde{v})^{p-1}\big),\label{8cases4}
\end{align}
where $\frac{1}{b_1}+\frac{1}{b_2}=\frac{1}{2},\ \ \frac{1}{2}-\frac{1}{b_3}=\frac{p-2}{b_5}+\frac{1}{b_6}$, $\frac{1}{c_1}+\frac{1}{c_2}=\frac{1}{2},\ \ \frac{1}{2}-\frac{1}{c_3}=\frac{p-2}{c_5}+\frac{1}{c_6}$.
\begin{align}
&\Vert\vert u(\tau,\cdot)\vert^q-\vert \tilde{u}(\tau,\cdot)\vert^q\Vert_{L^1}\notag\\
&\lesssim\tau^{(m+1)n+(-(m+1)n-\beta_1+1)q}
\ell_1(\tau)^{q\theta_1(2q)}M_1(\tau,u-\tilde{u})
\big(M_1(\tau,u)^{q-1}+M_1(\tau,\tilde{u})^{q-1}\big),\label{8cases5}
\end{align}
\begin{align}
&\Vert\vert u(\tau,\cdot)\vert^q-\vert \tilde{u}(\tau,\cdot)\vert^q\Vert_{L^2}\notag\\
&\lesssim\tau^{(m+1)\frac{n}{2}+(-(m+1)n-\beta_1+1)q}
\ell_1(\tau)^{q\theta_1(2q)}M_1(\tau,u-\tilde{u})
\big(M_1(\tau,u)^{q-1}+M_1(\tau,\tilde{u})^{q-1}\big);\label{8cases6}
\end{align}
For $q>\lceil \sigma\rceil\ (\geq2)$,
\begin{align}
&\Vert\vert u(\tau,\cdot)\vert^q-\vert \tilde{u}(\tau,\cdot)\vert^q\Vert_{\dot{H}^{\sigma-1}}\notag\\
&\lesssim\tau^{(m+1)
(\frac{n}{2}-\sigma+1)+(-(m+1)n-\beta_1+1)q}\ell_1(\tau)^{\theta_2(d_1)+(q-1)\theta_1((q-1)d_2)+\theta_1(d_3)+(q-2)
\theta_1(d_5)+\theta_2(d_6)}\notag\\
&\quad\times M_1(\tau,u-\tilde{u})
\big(M_1(\tau,u)^{q-1}+M_2(\tau,\tilde{u})^{q-1}\big);\label{8cases7}
\end{align}
If $\sigma>2$ and  for   $q>\lceil \sigma-1\rceil\ (\geq2)$,
\begin{align}
&\Vert\vert u(\tau,\cdot)\vert^q-\vert \tilde{u}(\tau,\cdot)\vert^q\Vert_{\dot{H}^{\sigma-2}}\notag\\
&\lesssim\tau^{(m+1)
(\frac{n}{2}-\sigma+2)+(-(m+1)n-\beta_1+1)q}\ell_1(\tau)^{\theta_3(e_1)+(q-1)\theta_1((q-1)e_2)+\theta_1(e_3)+(q-2)
\theta_1(e_5)+\theta_3(e_6)}\notag\\
&\quad\times M_1(\tau,u-\tilde{u})
\big(M_1(\tau,u)^{q-1}+M_2(\tau,\tilde{u})^{q-1}\big),\label{8cases8}
\end{align}
where $\frac{1}{d_1}+\frac{1}{d_2}=\frac{1}{2},\ \ \frac{1}{2}-\frac{1}{d_3}=\frac{q-2}{d_5}+\frac{1}{d_6}$, $\frac{1}{e_1}+\frac{1}{e_2}=\frac{1}{2},\ \ \frac{1}{2}-\frac{1}{e_3}=\frac{q-2}{e_5}+\frac{1}{e_6}$.
\end{lemma}
\begin{remark}\label{existencesbbb}
Before proving Lemma \ref{8cases}, we need to show that, under the assumptions of Theorem \ref{theorem11}, the coefficients $b_i,c_i,d_i,e_i$ ($i=1,\cdots,6$) exist.  First, the parameter $p$  needs to satisfy $\theta_1(2p)\in[0,1],$  which is exactly the same as (\ref{condi22}). The parameters $b_1,b_2$ need to satisfy $\theta_1((p-1)b_2)\in[0,1],\theta_2(b_1)\in[\frac{\sigma-1}{\sigma},1]$, respectively, which is equivalent to
\begin{equation*}
\frac{1}{b_1}\in\left\{
\begin{aligned}
&(0,\frac{1}{2}],\ &n\leq2,\\
&[\frac{n-2}{2n},\frac{1}{2}],\ &n>2
\end{aligned}
\right.\
\text{and} \
\frac{1}{b_2}\in
\left\{
\begin{aligned}
&(0,\frac{p-1}{2}],\ &n\leq 2\sigma,\\
&[\frac{(p-1)(n-2\sigma)}{2n},\frac{p-1}{2}], \ & n> 2\sigma.
\end{aligned}
\right.
\end{equation*}
Since $\frac{1}{b_1}+\frac{1}{b_2}=\frac{1}{2}$ and $\sigma\geq1$, it suffices  to ensure that
\begin{equation*}
\frac{1}{2}\in\left\{
\begin{aligned}
&(0,\frac{p}{2}],\ &n\leq2,\\
&(\frac{n-2}{2n},\frac{p}{2}],\ &2<n\leq2\sigma,\\
&[\frac{(p-1)(n-2\sigma)+n-2}{2n},\frac{p}{2}],\ &n>2\sigma.
\end{aligned}
\right.
\end{equation*}
The condition $p\leq 1+\frac{2}{n-2\sigma}$ (if\ $n>2\sigma$) in Theorem \ref{theorem11}  ensures that the above expressions are satisfied. Regarding $b_3,b_5,b_6$, it  is necessary to ensure that
$\theta_1(b_3),\theta_1(b_5)\in[0,1],\theta_2(b_6)\in[\frac{\sigma-1}{\sigma},1]$. This is equivalent to simply ensuring that
\begin{equation}\label{b356}
\frac{1}{b_3}, \frac{1}{b_5}\in\left\{
\begin{aligned}
&(0,\frac{1}{2}],\ &n\leq2\sigma,\\
&[\frac{n-2\sigma}{2n},\frac{1}{2}],\ &n>2\sigma
\end{aligned}
\right.\
\
\text{and} \ \
\frac{1}{b_6}\in
\left\{
\begin{aligned}
&(0,\frac{1}{2}],\ &n\leq 2,\\
&[\frac{n-2}{2n},\frac{1}{2}], \ & n> 2.
\end{aligned}
\right.
\end{equation}
For convenience, let $\frac{1}{b_4}:=\frac{p-2}{b_5}+\frac{1}{b_6}$,  and note that  $\sigma\geq1$, the restrictions (\ref{b356}) on $b_i (i=3,5,6)$  turn into
\begin{equation}\label{b4}
\frac{1}{b_4}\in
\left
\{
\begin{aligned}
&(0,\frac{p-1}{2}],\ &n\leq 2,\\
&(\frac{n-2}{2n},\frac{p-1}{2}],\ &2<n\leq2\sigma,\\
&[\frac{(n-2\sigma)(p-2)+n-2}{2n},\frac{p-1}{2}],\ &n>2\sigma.
\end{aligned}
\right.
\end{equation}
Moreover, due to  $\frac{p-2}{b_5}+\frac{1}{b_6}=\frac{1}{2}-\frac{1}{b_3}$, together with  (\ref{b356}), we see $b_4$ also needs to satisfy
\begin{equation}\label{b42}
\frac{1}{b_4}\in\left\{\begin{aligned}&[0,\frac{1}{2}),\ &n\leq2\sigma,\\&[0,\frac{\sigma}{n}],\ &n>2\sigma. \end{aligned}\right.
\end{equation}
Hence, the existence  for $b_i, i=3,5,6$ reduce to ensuring that both (\ref{b4}) and (\ref{b42}) hold simultaneously, which means that we need to guarantee
\begin{equation}\label{vari}
\left\{
\begin{aligned}
&[0,\frac{1}{2})\cap (0,\frac{p-1}{2}]\neq \emptyset, \ & n\leq2,\\
&[0,\frac{1}{2})\cap (\frac{n-2}{2n},\frac{p-1}{2}]\neq \emptyset, \ & 2< n\leq2\sigma,\\
&[0,\frac{\sigma}{n}]\cap [\frac{(n-2\sigma)(p-2)+n-2}{2n},\frac{p-1}{2}]\neq \emptyset, \ &  n>2\sigma.
\end{aligned}
\right.
\end{equation}
The condition $p\leq 1+\frac{2}{n-2\sigma}$ (when\ $n>2\sigma$) ensures the validity of (\ref{vari}).
So far, we have completed the examination of the existence of $b_i, i=1,2,3,5,6$.

Under the assumption $\sigma\geq2$, we need to guarantee $\theta_1((p-1)c_2), \theta_1(c_3),\theta_1(c_5)\in[0,1]$, $\theta_3(c_1)\in[\frac{\sigma-2}{\sigma},1]$ and $\theta_3(c_6)\in[\frac{\sigma-2}{\sigma},1]$.  Repeating the argument for the existence of $b_i$, one can similarly deduce that the condition $p\leq1+\frac{4}{n-2\sigma}\ (\text{when}\ n>2\sigma)$ guarantees the existence of $c_i,i=1,2,3,5,6$.

The verification of the existence of $d_i,e_i,i=1,2,3,5,6$
follows the same process as above and will not be repeated here.
\end{remark}

\begin{proof}[Proof of Lemma \ref{8cases}]
By (\ref{u-tildeu}) in Lemma \ref{key2lemmaa} and the definitions of $M_2(\tau,v), M_2(\tau,\tilde{v}),M_2(\tau,v-\tilde{v})$, we have
\begin{equation*}
\resizebox{0.96\hsize}{!}{$
\begin{aligned}
&\Vert\vert v(\tau,\cdot)\vert^p-\vert \tilde{v}(\tau,\cdot)\vert^p\Vert_{L^2}\lesssim\Vert v(\tau,\cdot)-\tilde{v}(\tau,\cdot)\Vert_{\dot{H}^\sigma}^{\theta_1(2p)}\Vert v(\tau,\cdot)-\tilde{v}(\tau,\cdot)\Vert_{L^2}^{1-\theta_1(2p)}\\
&\times\Bigl(\Vert v(\tau,\cdot)\Vert_{\dot{H}^\sigma}^{\theta_1(2p)(p-1)}\Vert v(\tau,\cdot)\Vert_{L^2}^{(1-\theta_1(2p))(p-1)}+\Vert \tilde{v}(\tau,\cdot)\Vert_{\dot{H}^\sigma}^{\theta_1(2p)(p-1)}\Vert \tilde{v}(\tau,\cdot)\Vert_{L^2}^{(1-\theta_1(2p))(p-1)}\Bigr)\\
&\lesssim \tau^{(-\frac{(m+1)n}{2}+\frac{\sqrt{\delta_2}-\mu_2+1}{2})p-(m+1)\sigma p\theta_1(2p)}\ell_2(\tau)^{p\theta_1(2p)}M_2(\tau,v-\tilde{v})
\big(M_2(\tau,v)^{p-1}+M_2(\tau,\tilde{v})^{p-1}\big)\\
&=\tau^{\frac{(m+1)n}{2}+(-(m+1)n-\beta_2+1)p}\ell_2(\tau)^{p\theta_1(2p)}M_2(\tau,v-\tilde{v})
\big(M_2(\tau,v)^{p-1}+M_2(\tau,\tilde{v})^{p-1}\big).
\end{aligned}$}
\end{equation*}
By H$\ddot{\text{o}}$lder's inequality,  we have
\begin{align}
&\Vert\vert v(\tau,\cdot)\vert^p-\vert \tilde{v}(\tau,\cdot)\vert^p\Vert_{L^1}\leq\Vert\vert v(\tau,\cdot)\vert^p-\vert \tilde{v}(\tau,\cdot)\vert^p\Vert_{L^2}\big(\int_{\vert x\vert\leq \phi_m(\tau)-\phi_m(1)+M}dx\big)^{\frac{1}{2}}\notag\\
&\lesssim \tau^{\frac{(m+1)n}{2}}\Vert\vert v(\tau,\cdot)\vert^p-\vert \tilde{v}(\tau,\cdot)\vert^p\Vert_{L^2}\notag\\
&\lesssim \tau^{(m+1)n+(-(m+1)n-\beta_2+1)p}\ell_2(\tau)^{p\theta_1(2p)}M_2(\tau,v-\tilde{v})
\big(M_2(\tau,v)^{p-1}+M_2(\tau,\tilde{v})^{p-1}\big).\notag
\end{align}
The proofs of (\ref{8cases3})-(\ref{8cases8})  follow a similar approach so the details are omitted.
\end{proof}

Subsequently,  we  use Lemma \ref{8cases} along with Proposition \ref{unlinearestimate} to establish the proof of (\ref{key2}).

\begin{proof}[Proof of (\ref{key2})]
Case 1: $\delta_1,\delta_2>(m+1)^2(n+2\sigma-1)^2$ with $\sigma>1$. In this case, $\frac{\sqrt{\delta_1}}{2(m+1)}+\frac{1}{2}-\frac{n}{2}>\sigma>1$, using (\ref{u0=0u}) in Proposition \ref{unlinearestimate},  (\ref{8cases1})-(\ref{8cases2}) in Lemma \ref{8cases} and the integrability condition (\ref{infite1}), we obtain
\small{
\begin{align}
&\Vert H_1(v)(t,\cdot)-H_1(\tilde{v})(t,\cdot)\Vert_{L^2}\lesssim \int_1^t\Vert E_1^{\mu_1,\nu_1}(t,\tau,x)\ast\big(\vert v(\tau,x)\vert^p-\vert \tilde{v}(\tau,x)\vert^p\big)\Vert_{L^2}d\tau\notag\\
&\lesssim  t^{-(m+1)\frac{n}{2}+\frac{\sqrt{\delta_1}-\mu_1+1}{2}}\int_1^t\tau^{-\frac{\sqrt{\delta_1}
-\mu_1-1}{2}}\big(\Vert\vert v(\tau,\cdot)\vert^p-\vert \tilde{v}(\tau,\cdot)\vert^p\Vert_{L^1}d\tau\notag\\
&\quad\quad\quad\quad\quad\quad\quad\quad+\tau^{(m+1)\frac{n}{2}}\Vert\vert v(\tau,\cdot)\vert^p-\vert \tilde{v}(\tau,\cdot)\vert^p\Vert_{L^2}\big)d\tau\notag\\
&\lesssim t^{-(m+1)\frac{n}{2}+\frac{\sqrt{\delta_1}-\mu_1+1}{2}}\int_1^t\tau^{(m+1)n+\beta_1+(-(m+1)n-\beta_2+1)p} M_2(\tau,v-\tilde{v})\big(M_2(\tau,v)^{p-1}+M_2(\tau,\tilde{v})^{p-1}\big)d\tau\notag\\
&\lesssim t^{-(m+1)\frac{n}{2}+\frac{\sqrt{\delta_1}-\mu_1+1}{2}}\Vert (u,v)-(\tilde{u},\tilde{v})\Vert_{\mathcal{X}(t)}\bigl(\Vert (u,v)\Vert_{\mathcal{X}(t)}^{p-1}+\Vert (\tilde{u},\tilde{v})\Vert_{\mathcal{X}(t)}^{p-1}\bigr).\label{1case1}
\end{align}
}

In a similar manner,
\begin{align}
&\Vert \partial_t\big(H_1(v)(t,\cdot)-H_1(\tilde{v})(t,\cdot)\big)\Vert_{L^2}\notag\\
&\lesssim t^{m-(m+1)(1+\frac{n}{2})+\frac{\sqrt{\delta_1}-\mu_1+1}{2}}\Vert (u,v)-(\tilde{u},\tilde{v})\Vert_{\mathcal{X}(t)}\bigl(\Vert (u,v)\Vert_{\mathcal{X}(t)}^{p-1}+\Vert (\tilde{u},\tilde{v})\Vert_{\mathcal{X}(t)}^{p-1}\bigr).\label{1case2}
\end{align}
By (\ref{u0=0u}), (\ref{8cases1}) and (\ref{8cases3}), we get
\small{
\begin{align}
&\Vert H_1(v)(t,\cdot)-H_1(\tilde{v})(t,\cdot)\Vert_{\dot{H}^{\sigma}}\lesssim t^{-(m+1)(\sigma+\frac{n}{2})+\frac{\sqrt{\delta_1}-\mu_1+1}{2}}\int_1^t\tau^
{-\frac{\sqrt{\delta_1}-\mu_1-1}{2}}\Vert\vert v(\tau,\cdot)\vert^p-\vert\tilde{v}(\tau,\cdot)\vert^p\Vert_{L^1}d\tau\notag\\
&\quad\quad +t^{-(m+1)
(\sigma+\frac{n}{2})+\frac{\sqrt{\delta_1}-\mu_1+1}{2}}\int_1^t\tau^
{-\frac{\sqrt{\delta_1}-\mu_1-1}{2}+(m+1)(\frac{n}{2}+\sigma-1)}\Vert\vert v(\tau,\cdot)\vert^p-\vert \tilde{v}(\tau,\cdot)\vert^p\Vert_{\dot{H}^{\sigma-1}}d\tau\notag\\
&\lesssim t^{-(m+1)
(\sigma+\frac{n}{2})+\frac{\sqrt{\delta_1}-\mu_1+1}{2}}\int_1^t\tau^{(m+1)n+\beta_1+
(-(m+1)n-\beta_2+1)p} M_2(\tau,v-\tilde{v})\big(M_2(\tau,v)^{p-1}\notag\\
&\quad\quad\quad\quad\quad\quad\quad\quad\quad\quad\quad\quad\quad\quad\quad\quad\quad\quad\quad\quad\quad\quad\quad\quad\quad\quad\quad\quad\quad\quad\quad\quad+M_2(\tau,\tilde{v})^{p-1}\big)d\tau\notag\\
&\lesssim t^{-(m+1)
(\sigma+\frac{n}{2})+\frac{\sqrt{\delta_1}-\mu_1+1}{2}}\Vert (u,v)-(\tilde{u},\tilde{v})\Vert_{\mathcal{X}(t)}\bigl(\Vert (u,v)\Vert_{\mathcal{X}(t)}^{p-1}+\Vert (\tilde{u},\tilde{v})\Vert_{\mathcal{X}(t)}^{p-1}\bigr).\label{1case3}
\end{align}
}
If $\sigma>2$, using (\ref{u0=0ut}), (\ref{8cases1}), (\ref{8cases3})-(\ref{8cases4}) yields
\small{
\begin{align}
&\Vert\partial_t(H_1(v)(t,\cdot)-H_1(\tilde{v})(t,\cdot))\Vert_{\dot{H}^{\sigma-1}}\notag\\
&\lesssim t^{m-(m+1)
(\sigma+\frac{n}{2})+\frac{\sqrt{\delta_1}-\mu_1+1}{2}}\int_1^t\tau^{
-\frac{\sqrt{\delta_1}-\mu_1-1}{2}}\big(\Vert\vert v(\tau,\cdot)\vert^p-\vert\tilde{v}(\tau,\cdot)\vert^p\Vert_{L^1}d\tau\notag\\
&\quad+\tau^{(m+1)(\frac{n}{2}+\sigma-1)}\Vert\vert v(\tau,\cdot)\vert^p-\vert\tilde{v}(\tau,\cdot)\vert^p\Vert_{\dot{H}^{\sigma-1}}+\tau^{(m+1)(\frac{n}{2}+\sigma-2)}\Vert\vert v(\tau,\cdot)\vert^p-\vert\tilde{v}(\tau,\cdot)\vert^p\Vert_{\dot{H}^{\sigma-2}}\big)d\tau\notag\\
&\lesssim t^{m-(m+1)
(\sigma+\frac{n}{2})+\frac{\sqrt{\delta_1}-\mu_1+1}{2}}\Vert (u,v)-(\tilde{u},\tilde{v})\Vert_{\mathcal{X}(t)}\bigl(\Vert (u,v)\Vert_{\mathcal{X}(t)}^{p-1}+\Vert (\tilde{u},\tilde{v})\Vert_{\mathcal{X}(t)}^{p-1}\bigr),\label{1case41}
\end{align}
}
while if $1<\sigma<2$, by (\ref{u0=0ut}), (\ref{8cases1})-(\ref{8cases3}), we can also get
\begin{equation}\label{1case42}
\resizebox{0.79\hsize}{!}{$
\begin{aligned}
&\Vert\partial_t(H_1(v)(t,\cdot)-H_1(\tilde{v})(t,\cdot))\Vert_{\dot{H}^{\sigma-1}}\\
&\lesssim t^{m-(m+1)
(\sigma+\frac{n}{2})+\frac{\sqrt{\delta_1}-\mu_1+1}{2}}\Vert (u,v)-(\tilde{u},\tilde{v})\Vert_{\mathcal{X}(t)}\bigl(\Vert (u,v)\Vert_{\mathcal{X}(t)}^{p-1}+\Vert (\tilde{u},\tilde{v})\Vert_{\mathcal{X}(t)}^{p-1}\bigr).
\end{aligned}$}
\end{equation}
Hence, from (\ref{1case1})-(\ref{1case42}), we infer that
\begin{equation}\label{1casecon}
M_1(t,H_1(v)-H_1(\tilde{v}))\lesssim \Vert (u,v)-(\tilde{u},\tilde{v})\Vert_{\mathcal{X}(t)}\bigl(\Vert (u,v)\Vert_{\mathcal{X}(t)}^{p-1}+\Vert (\tilde{u},\tilde{v})\Vert_{\mathcal{X}(t)}^{p-1}\bigr).
\end{equation}

By (\ref{u0=0u}), (\ref{8cases5})-(\ref{8cases6}) and the integrability condition (\ref{infite2}), we get
\small{
\begin{align}
&\Vert H_2(u)(t,\cdot)-H_2(\tilde{u})(t,\cdot)\Vert_{L^2}\notag\\
&\lesssim t^{-(m+1)\frac{n}{2}+\frac{\sqrt{\delta_2}-\mu_2+1}{2}}\int_1^t\tau^
{-\frac{\sqrt{\delta_2}-\mu_2-1}{2}}\big(\Vert\vert u(\tau,\cdot)\vert^q-\vert \tilde{u}(\tau,\cdot)\vert^q\Vert_{L^1}d\tau\notag\\
&\quad\quad\quad\quad\quad\quad\quad\quad\quad\quad\quad\quad\quad\quad\quad+\tau^{(m+1)\frac{n}{2}}\Vert\vert u(\tau,\cdot)\vert^q-\vert \tilde{u}(\tau,\cdot)\vert^q\Vert_{L^2}\big)d\tau\notag\\
&\lesssim t^{-(m+1)\frac{n}{2}+\frac{\sqrt{\delta_2}-\mu_2+1}{2}}\int_1^t\tau^
{(m+1)n+\beta_2+(-(m+1)n-\beta_1+1)q} M_1(\tau,u-\tilde{u})\big(M_1(\tau,u)^{q-1}+
M_1(\tau,\tilde{u})^{q-1}\big)d\tau\notag\\
&\lesssim t^{-(m+1)\frac{n}{2}+\frac{\sqrt{\delta_2}-\mu_2+1}{2}}\Vert (u,v)-(\tilde{u},\tilde{v})\Vert_{\mathcal{X}(t)}\bigl(\Vert (u,v)\Vert_{\mathcal{X}(t)}^{q-1}+\Vert (\tilde{u},\tilde{v})\Vert_{\mathcal{X}(t)}^{q-1}\bigr).\label{2case1}
\end{align}
}
Following the same procedure, we arrive at the following  estimates
\begin{equation*}\label{2case2}
\resizebox{1.03\hsize}{!}{$
\begin{aligned}
&\Vert \partial_t\big(H_2(u)(t,\cdot)-H_2(\tilde{u})(t,\cdot)\big)\Vert_{L^2}\lesssim t^{m-(m+1)(1+\frac{n}{2})+\frac{\sqrt{\delta_2}-\mu_2+1}{2}}\Vert (u,v)-(\tilde{u},\tilde{v})\Vert_{\mathcal{X}(t)}\bigl(\Vert (u,v)\Vert_{\mathcal{X}(t)}^{q-1}+\Vert (\tilde{u},\tilde{v})\Vert_{\mathcal{X}(t)}^{q-1}\bigr),\\
&\Vert H_2(u)(t,\cdot)-H_2(\tilde{u})(t,\cdot)\Vert_{\dot{H}^{\sigma}}\lesssim t^{-(m+1)(\sigma+\frac{n}{2})+\frac{\sqrt{\delta_2}-\mu_2+1}{2}}\Vert (u,v)-(\tilde{u},\tilde{v})\Vert_{\mathcal{X}(t)}\bigl(\Vert (u,v)\Vert_{\mathcal{X}(t)}^{q-1}+\Vert (\tilde{u},\tilde{v})\Vert_{\mathcal{X}(t)}^{q-1}\bigr),\\
&\Vert \partial_t\big(H_2(u)(t,\cdot)-H_2(\tilde{u})(t,\cdot)\big)\Vert_{\dot{H}^{\sigma-1}}\lesssim t^{m-(m+1)(\sigma+\frac{n}{2})+\frac{\sqrt{\delta_2}-\mu_2+1}{2}}\Vert (u,v)-(\tilde{u},\tilde{v})\Vert_{\mathcal{X}(t)}\bigl(\Vert (u,v)\Vert_{\mathcal{X}(t)}^{q-1}+\Vert (\tilde{u},\tilde{v})\Vert_{\mathcal{X}(t)}^{q-1}\bigr).
\end{aligned}$}
\end{equation*}
Then we derive
\begin{equation}\label{2casecon0}
M_2(t,H_2(u)-H_2(\tilde{u}))\lesssim \Vert (u,v)-(\tilde{u},\tilde{v})\Vert_{\mathcal{X}(t)}\bigl(\Vert (u,v)\Vert_{\mathcal{X}(t)}^{q-1}+\Vert (\tilde{u},\tilde{v})\Vert_{\mathcal{X}(t)}^{q-1}\bigr).
\end{equation}

Consequently, combining (\ref{1casecon}) and (\ref{2casecon0}) gives
\begin{equation*}
\resizebox{0.99\hsize}{!}{$
\begin{aligned}
&\Vert \mathcal{N}(u,v)-\mathcal{N}(\tilde{u},\tilde{v})\Vert_{\mathcal{X}(T)}=\big\Vert\big(H_1(v)(t,\cdot)-H_1(\tilde{v})(t,\cdot),H_2(u)(t,\cdot)-H_2(\tilde{u})
(t,\cdot)\big)\big\Vert_{\mathcal{X}(T)}\\
&\lesssim  \Vert (u,v)-(\tilde{u},\tilde{v})\Vert_{\mathcal{X}(T)}
\bigl(\Vert (u,v)\Vert_{\mathcal{X}(T)}^{p-1}+\Vert (u,v)\Vert_{\mathcal{X}(T)}^{q-1}+\Vert (\tilde{u},\tilde{v})\Vert_{\mathcal{X}(T)}^{p-1}+\Vert (\tilde{u},\tilde{v})\Vert_{\mathcal{X}(T)}^{q-1}\bigr).
\end{aligned}$}
\end{equation*}

Case 2: $\delta_1>(m+1)^2(n+2\sigma-1)^2$, $\delta_2=(m+1)^2(n+2\sigma-1)^2$, $\sigma>1$. 
In this case, $\mathcal{l}_1(\tau)=1$ and $\mathcal{l}_2(\tau)=(1+\log\tau)^{\frac{1}{2}}$. By Proposition \ref{unlinearestimate} and Lemma \ref{8cases}, we get
\begin{equation*}\label{2case11}
\resizebox{0.98\hsize}{!}{$
\begin{aligned}
&\Vert H_1(v)(t,\cdot)-H_1(\tilde{v})(t,\cdot)\Vert_{L^2}\lesssim t^{-(m+1)\frac{n}{2}+\frac{\sqrt{\delta_1}-\mu_1+1}{2}}
\int_1^t\tau^{(m+1)n+\beta_1+(-(m+1)n-\beta_2+1)p}\\
&\quad\times(1+\log \tau)^{\frac{p\theta_1(2p)}{2}} M_2(\tau,v-\tilde{v})\big(M_2(\tau,v)^{p-1}+M_2(\tau,\tilde{v})^{p-1}\big)d\tau\\
&\lesssim t^{-(m+1)\frac{n}{2}+\frac{\sqrt{\delta_1}-\mu_1+1}{2}}\Vert (u,v)-(\tilde{u},\tilde{v})\Vert_{\mathcal{X}(t)}\bigl(\Vert (u,v)\Vert_{\mathcal{X}(t)}^{p-1}+\Vert (\tilde{u},\tilde{v})\Vert_{\mathcal{X}(t)}^{p-1}\bigr),
\end{aligned}$}
\end{equation*}
here we used $
\int_1^t\tau^{(m+1)n+\beta_1+(-(m+1)n-\beta_2+1)p}(1+\log \tau)^{\frac{p\theta_1(2p)}{2}}<\infty$, one can see (\ref{ingre}).
With the same method, we conclude
\begin{equation*}\label{2case2}
\resizebox{1.03\hsize}{!}{$
\begin{aligned}
&\Vert \partial_t\big(H_1(v)(t,\cdot)-H_1(\tilde{v})(t,\cdot)\big)\Vert_{L^2}\lesssim t^{m-(m+1)(1+\frac{n}{2})+\frac{\sqrt{\delta_1}-\mu_1+1}{2}}\Vert (u,v)-(\tilde{u},\tilde{v})\Vert_{\mathcal{X}(t)}\bigl(\Vert (u,v)\Vert_{\mathcal{X}(t)}^{p-1}+\Vert (\tilde{u},\tilde{v})\Vert_{\mathcal{X}(t)}^{p-1}\bigr),\\
&\Vert H_1(v)(t,\cdot)-H_1(\tilde{v})(t,\cdot)\Vert_{\dot{H}^{\sigma}}\lesssim t^{-(m+1)
(\sigma+\frac{n}{2})+\frac{\sqrt{\delta_1}-\mu_1+1}{2}}\Vert (u,v)-(\tilde{u},\tilde{v})\Vert_{\mathcal{X}(t)}\notag\bigl(\Vert (u,v)\Vert_{\mathcal{X}(t)}^{p-1}+\Vert (\tilde{u},\tilde{v})\Vert_{\mathcal{X}(t)}^{p-1}\bigr),\\
&\Vert\partial_t(H_1(v)(t,\cdot)-H_1(\tilde{v})(t,\cdot))\Vert_{\dot{H}^{\sigma-1}}\lesssim t^{m-(m+1)
(\sigma+\frac{n}{2})+\frac{\sqrt{\delta_1}-\mu_1+1}{2}}\Vert (u,v)-(\tilde{u},\tilde{v})\Vert_{\mathcal{X}(t)}\notag\bigl(\Vert (u,v)\Vert_{\mathcal{X}(t)}^{p-1}+\Vert (\tilde{u},\tilde{v})\Vert_{\mathcal{X}(t)}^{p-1}\bigr).
\end{aligned}$}
\end{equation*}
Thus
\begin{equation}\label{2casecon}
M_1(t,H_1(v)-H_1(\tilde{v}))\lesssim \Vert (u,v)-(\tilde{u},\tilde{v})\Vert_{\mathcal{X}(t)}\bigl(\Vert (u,v)\Vert_{\mathcal{X}(t)}^{p-1}+\Vert (\tilde{u},\tilde{v})\Vert_{\mathcal{X}(t)}^{p-1}\bigr).
\end{equation}

Note that $\frac{\sqrt{\delta_2}}{2(m+1)}+\frac{1}{2}-\frac{n}{2}=\sigma>1$, by Proposition \ref{unlinearestimate}, Lemma \ref{8cases} and the integrability condition (\ref{infite2}), we have
\begin{align}
&\Vert H_2(u)(t,\cdot)-H_2(\tilde{u})(t,\cdot)\Vert_{L^2}\lesssim t^{-(m+1)\frac{n}{2}+\frac{\sqrt{\delta_2}-\mu_2+1}{2}}\notag\\
&\quad\times\int_1^t\tau^
{(m+1)n+\beta_2+(-(m+1)n-\beta_1+1)q} M_1(\tau,u-\tilde{u})\big(M_1(\tau,u)^{q-1}+
M_1(\tau,\tilde{u})^{q-1}\big)d\tau\notag\\
&\lesssim t^{-(m+1)\frac{n}{2}+\frac{\sqrt{\delta_2}-\mu_2+1}{2}}\Vert (u,v)-(\tilde{u},\tilde{v})\Vert_{\mathcal{X}(t)}\bigl(\Vert (u,v)\Vert_{\mathcal{X}(t)}^{q-1}+\Vert (\tilde{u},\tilde{v})\Vert_{\mathcal{X}(t)}^{q-1}\bigr),\label{2case25}
\end{align}
\begin{align}
&\Vert \partial_t\big( H_2(u)(t,\cdot)-H_2(\tilde{u})(t,\cdot)\big)\Vert_{L^2}\notag\\
&\lesssim t^{m-(m+1)(1+\frac{n}{2})+\frac{\sqrt{\delta_2}-\mu_2+1}{2}}\Vert (u,v)-(\tilde{u},\tilde{v})\Vert_{\mathcal{X}(t)}\bigl(\Vert (u,v)\Vert_{\mathcal{X}(t)}^{q-1}+\Vert (\tilde{u},\tilde{v})\Vert_{\mathcal{X}(t)}^{q-1}\bigr),\label{2case26}
\end{align}
\begin{align}
&\Vert H_2(u)(t,\cdot)-H_2(\tilde{u})(t,\cdot)\Vert_{\dot{H}^\sigma}\lesssim t^{-\frac{\mu_2+m}{2}}\int_1^t\tau^{(m+1)n+\beta_2+(-(m+1)n-\beta_1+1)q} (1+\log \frac{t}{\tau})^{\frac{1}{2}}\notag\\
&\quad\quad\quad\times M_1(\tau,u-\tilde{u})\big(M_1(\tau,u)^{q-1}+
M_1(\tau,\tilde{u})^{q-1}\big)d\tau\notag\\
&\lesssim t^{-\frac{\mu_2+m}{2}}(1+\log t)^{\frac{1}{2}}\Vert (u,v)-(\tilde{u},\tilde{v})\Vert_{\mathcal{X}(t)}\bigl(\Vert (u,v)\Vert_{\mathcal{X}(t)}^{q-1}+\Vert (\tilde{u},\tilde{v})\Vert_{\mathcal{X}(t)}^{q-1}\bigr)\notag\\
&=t^{-(m+1)(\sigma+\frac{n}{2})+\frac{\sqrt{\delta_2}-\mu_2+1}{2}}(1+\log t)^{\frac{1}{2}}\Vert (u,v)-(\tilde{u},\tilde{v})\Vert_{\mathcal{X}(t)}\bigl(\Vert (u,v)\Vert_{\mathcal{X}(t)}^{q-1}+\Vert (\tilde{u},\tilde{v})\Vert_{\mathcal{X}(t)}^{q-1}\bigr)\label{2case27}
\end{align}
and
\begin{align}
&\Vert \partial_t\big( H_2(u)(t,\cdot)-H_2(\tilde{u})(t,\cdot)\big)\Vert_{\dot{H}^{\sigma-1}}\notag\\
&\lesssim t^{m-(m+1)(\sigma+\frac{n}{2})+\frac{\sqrt{\delta_2}-\mu_2+1}{2}}(1+\log t)^{\frac{1}{2}}\Vert (u,v)-(\tilde{u},\tilde{v})\Vert_{\mathcal{X}(t)}\bigl(\Vert (u,v)\Vert_{\mathcal{X}(t)}^{q-1}+\Vert (\tilde{u},\tilde{v})\Vert_{\mathcal{X}(t)}^{q-1}\bigr).\label{2case28}
\end{align}
Combining (\ref{2case25})-(\ref{2case28}) yields
\begin{equation}\label{21casecon}
M_2(t,H_2(u)-H_2(\tilde{u}))\lesssim \Vert (u,v)-(\tilde{u},\tilde{v})\Vert_{\mathcal{X}(t)}\bigl(\Vert (u,v)\Vert_{\mathcal{X}(t)}^{q-1}+\Vert (\tilde{u},\tilde{v})\Vert_{\mathcal{X}(t)}^{q-1}\bigr).
\end{equation}
Therefore, it follows from (\ref{2casecon}) and (\ref{21casecon}) that
\begin{equation*}
\begin{aligned}
&\Vert \mathcal{N}(u,v)-\mathcal{N}(\tilde{u},\tilde{v})\Vert_{\mathcal{X}(T)}\\
&\lesssim
\Vert (u,v)-(\tilde{u},\tilde{v})\Vert_{\mathcal{X}(T)}
\bigl(\Vert (u,v)\Vert_{\mathcal{X}(T)}^{p-1}+\Vert (u,v)\Vert_{\mathcal{X}(T)}^{q-1}+\Vert (\tilde{u},\tilde{v})\Vert_{\mathcal{X}(T)}^{p-1}+\Vert (\tilde{u},\tilde{v})\Vert_{\mathcal{X}(T)}^{q-1}\bigr).
\end{aligned}
\end{equation*}

For the remaining cases Case 3- Case 8 listed in (\ref{bigcase3})- (\ref{bigcase8}),
we can proceed in a similar way and will not provide further details.
\end{proof}
\subsection{The proof of Theorem \ref{theorem12}}\label{proofthm12}
In this section, we aim to prove the existence of global solutions for
$p\leq\tilde{p}$ and $q>\tilde{q}$, corresponding to Theorem \ref{theorem12}.  The proof relies on Propositions \ref{linearestimate}-\ref{unlinearestimate}, Lemma \ref{foures} and Lemma \ref{8cases}, so we will not repeatedly highlight these elements in the discussion.

Based on the integrability analysis in Section \ref{prooftheorem2}, specifically (\ref{infite1}) and (\ref{ingre}), it is clear that for different values of $\delta_1,\delta_2$, the presence of $(1+\log \tau)^{\frac{1}{2}}$  does not affect the integrability. However, when $p\leq\tilde{p}$,  the situation changes, which is why we exclude the case of $\delta_2=(m+1)^2(n+2\sigma-1)^2$. We  clarify the specific reasons for this at the end of this section, i.e., Remark \ref{re444}.

By the definition (\ref{normxt}) of $\mathcal{X}(T)$, we see
\begin{equation}\label{defdef41}
\Vert(u,v)\Vert_{\mathcal{X}(T)}=\sup\limits_{t\in[1,T]}\big(t^{-\alpha_1}M_1
(t,u)+M_2(t,v)\big),
\end{equation}
where $\alpha_1$ is given by (\ref{alpha1}), i.e.,
\begin{equation*}
\alpha_1=\left\{
\begin{aligned}
&\big((m+1)n+\beta_2-1\big)(\tilde{p}-p),&\quad& \text{if}\ \ p<\tilde{p}, \\
&\epsilon, &\quad&\text{if}\ \ p=\tilde{p}.
\end{aligned}
\right.
\end{equation*}
As for the operator $\mathcal{N}$ defined by (\ref{operator}),  our goal remains to establish
\begin{proposition}\label{keyyypro}
Under the conditions of Theorem \ref{theorem12},
there exists a constant $C>0$ such that for any $T>1$  and any $(u,v), (\tilde{u},\tilde{v})\in \mathcal{X}(T)$,
\begin{align}
&\Vert \mathcal{N}(u,v)\Vert_{\mathcal{X}(T)}\leq C(\Vert(u_0,u_1)\Vert_{D^\sigma}+\Vert(v_0,v_1)\Vert_{D^\sigma})+C(\Vert (u,v)\Vert_{\mathcal{X}{(T)}}^p+\Vert (u,v)\Vert_{\mathcal{X}{(T)}}^q),\label{key421}\\
&\Vert \mathcal{N}(u,v)-\mathcal{N}(\tilde{u},\tilde{v})\Vert_{\mathcal{X}(T)}\leq C \Vert (u,v)-(\tilde{u},\tilde{v})\Vert_{\mathcal{X}(T)}\notag\\
&\quad\quad\quad\times\bigl(\Vert (u,v)\Vert_{\mathcal{X}(T)}^{p-1}+\Vert (u,v)\Vert_{\mathcal{X}(T)}^{q-1}+\Vert (\tilde{u},\tilde{v})\Vert_{\mathcal{X}(T)}^{p-1}+\Vert (\tilde{u},\tilde{v})\Vert_{\mathcal{X}(T)}^{q-1}\bigr).\label{key422}
\end{align}
\end{proposition}
After establishing Proposition \ref{keyyypro}, Theorem \ref{theorem12} can be immediately obtained by following the proof method of Theorem \ref{theorem11} in Section \ref{prooftheorem2}. Therefore, we omit the details here.

It is worth mentioning that, since  the proof of Proposition \ref{keyyypro} is quite similar to that of Proposition \ref{keypro}, we will focus on outlining the main differences. Thus, we only present the details  on
Case 1: $\delta_1,\delta_2>(m+1)^2(n+2\sigma-1)^2, \sigma>1$.  
The other three cases
\begin{align*}
&\text{Case 2:}\ \delta_1=\bigl((m+1)(n+2\sigma-1)\bigr)^2, \ \delta_2>\bigl((m+1)(n+2\sigma-1)\bigr)^2,\ \sigma>1,\\
&\text{Case 3:}\ \delta_1>\bigl((m+1)(n+2\sigma-1)\bigr)^2,\ \delta_2>\bigl((m+1)(n+2\sigma-1)\bigr)^2,\ \sigma=1, \\
&\text{Case 4:}\ \delta_1=\bigl((m+1)(n+2\sigma-1)\bigr)^2, \ \delta_2>\bigl((m+1)(n+2\sigma-1)\bigr)^2,\ \sigma=1
\end{align*}
can be addressed analogously.

\begin{proof}[Proof of Proposition \ref{keyyypro}]
The estimates for $u^l(t,\cdot),v^l(t,\cdot)$ are exactly the same as (\ref{ull2hsigmaetc}) and (\ref{vll2hsigmaetc}), then
\begin{equation}\label{ulvllines}
\begin{aligned}
\Vert(u^l,v^l)\Vert_{\mathcal{X}(T)}&\lesssim \sup\limits_{t\in[1,T]}\big(t^{-\alpha_1}M_1
(t,u^l)+M_2(t,v^l)\big)\\
&\lesssim\sup\limits_{t\in[1,T]}\big(t^{-\alpha_1} \Vert (u_0,u_1)\Vert_{D^\sigma}+\Vert (v_0,v_1)\Vert_{D^\sigma}\big)\\
&\lesssim \Vert (u_0,u_1)\Vert_{D^\sigma}+\Vert (v_0,v_1)\Vert_{D^\sigma}.
\end{aligned}
\end{equation}
In view of $p\leq\tilde{p}$, we get
\begin{equation}\label{nonintegralindex}
(m+1)n+\beta_1+
(-(m+1)n-\beta_2+1)p\geq-1.
\end{equation}
By (\ref{defdef41}),  $M_2(\tau,v)\lesssim \Vert(u,v)\Vert_{\mathcal{X}(t)} (1\leq\tau< t)$ holds,
thus when $p<\tilde{p}$, we get
\begin{equation*}\label{th4case111}
\begin{aligned}
&\Vert H_1(v)(t,\cdot)\Vert_{L^2}\lesssim t^{-(m+1)\frac{n}{2}+\frac{\sqrt{\delta_1}-\mu_1+1}{2}}\int_1^t\tau^{(m+1)n+\beta_1+
(-(m+1)n-\beta_2+1)p}M_2(\tau,v)^pd\tau\\
&\lesssim t^{-(m+1)\frac{n}{2}+\frac{\sqrt{\delta_1}-\mu_1+1}{2}+(m+1)n+\beta_1+
(-(m+1)n-\beta_2+1)p+1}\Vert(u,v)\Vert_{\mathcal{X}(t)}^p.
\end{aligned}
\end{equation*}
Note $$(m+1)n+\beta_1+
(-(m+1)n-\beta_2+1)p+1=\alpha_1,$$ so
\begin{equation}\label{th4case111}
\begin{aligned}
&\Vert H_1(v)(t,\cdot)\Vert_{L^2}\lesssim t^{-(m+1)\frac{n}{2}+\frac{\sqrt{\delta_1}-\mu_1+1}{2}+\alpha_1}\Vert(u,v)\Vert_
{\mathcal{X}(t)}^p.
\end{aligned}
\end{equation}
And if $p=\tilde{p}$,
\begin{align}
\Vert H_1(v)(t,\cdot)\Vert_{L^2}&\lesssim t^{-(m+1)\frac{n}{2}+\frac{\sqrt{\delta_1}-\mu_1+1}{2}}\int_1^t\tau^{(m+1)n+\beta_1+
(-(m+1)n-\beta_2+1)p}M_2(\tau,v)^pd\tau\notag\\
&= t^{-(m+1)\frac{n}{2}+\frac{\sqrt{\delta_1}-\mu_1+1}{2}}\int_1^t\tau^{-1}M_2(\tau,v)^p
d\tau\notag\\
&\lesssim t^{-(m+1)\frac{n}{2}+\frac{\sqrt{\delta_1}-\mu_1+1}{2}}\log t\Vert(u,v)\Vert_
{\mathcal{X}(t)}^p\notag\\
&\lesssim t^{-(m+1)\frac{n}{2}+\frac{\sqrt{\delta_1}-\mu_1+1}{2}+\alpha_1}\Vert(u,v)\Vert_
{\mathcal{X}(t)}^p.\label{th4case112}
\end{align}
Continuing with the same method, we deduce  that if $p\leq\tilde{p}$,
\begin{align}
&\Vert\partial_tH_1(v)(t,\cdot)\Vert_{L^2}\lesssim t^{m-(m+1)(1+\frac{n}{2})+\frac{\sqrt{\delta_1}-\mu_1+1}{2}+\alpha_1}\Vert(u,v)\Vert_
{\mathcal{X}(t)}^p,\label{th4case12}\\
&\Vert H_1(v)(t,\cdot)\Vert_{\dot{H}^\sigma}\lesssim t^{-(m+1)(\sigma+\frac{n}{2})+\frac{\sqrt{\delta_1}-\mu_1+1}{2}+\alpha_1}\Vert(u,v)\Vert_
{\mathcal{X}(t)}^p,\label{th4case13}\\
&\Vert\partial_tH_1(v)(t,\cdot)\Vert_{\dot{H}^{\sigma-1}}\lesssim t^{m-(m+1)(\sigma+\frac{n}{2})+\frac{\sqrt{\delta_1}-\mu_1+1}{2}+\alpha_1}\Vert(u,v)\Vert_
{\mathcal{X}(t)}^p.\label{th4case14}
\end{align}
From (\ref{th4case111})-(\ref{th4case14}) and by the definition  of $M_1(t,H_1(v))$, we obtain
\begin{equation}\label{M111}
M_1(t,H_1(v))\lesssim t^{\alpha_1}\Vert(u,v)\Vert_
{\mathcal{X}(t)}^p.
\end{equation}

It is evident that if $p<\tilde{p}$,
\begin{equation*}
\begin{aligned}
&(m+1)n+\beta_2+(-(m+1)n-\beta_1+1+\alpha_1)q<-1\\
\iff &(m+1)n+\beta_2+\big(-(m+1)n-\beta_1+1+((m+1)n+\beta_2-1)(\tilde{p}-p)\big)q<-1\\
\iff &\frac{q+1}{pq-1}<\frac{(m+1)n+\beta_2-1}{2};
\end{aligned}
\end{equation*}
while if $p=\tilde{p}$, we have $(m+1)n+\beta_1+1=((m+1)n+\beta_2-1)p$, so
\begin{equation}\label{th4con1}
\begin{aligned}
&(m+1)n+\beta_2+(-(m+1)n-\beta_1+1+\alpha_1)q<-1\\
\iff &(m+1)n+\beta_2+(2+\epsilon-((m+1)n+\beta_2-1)p)q<-1\\
\iff &\frac{(1+\frac{\epsilon}{2})q+1}{pq-1}<\frac{(m+1)n+\beta_2-1}{2},
\end{aligned}
\end{equation}
owing to the condition $\frac{q+1}{pq-1}<\frac{(m+1)n+\beta_2-1}{2}$ in Theorem \ref{theorem12}, we can choose $\epsilon>0$ sufficiently small such that (\ref{th4con1}) holds. Based on the analysis above,  we conclude that if $p\leq\tilde{p}$, then
\begin{equation}\label{ingreinfty}
\resizebox{0.901\hsize}{!}{$
\begin{aligned}
\int_1^t\tau^{(m+1)n+
\beta_2+(-(m+1)n-\beta_1+1+\alpha_1)q}d\tau\leq\int_1^\infty\tau^{(m+1)n+
\beta_2+(-(m+1)n-\beta_1+1+\alpha_1)q}d\tau <\infty.
\end{aligned}$}
\end{equation}
Due to  $M_1(\tau,u)\lesssim\tau^{\alpha_1}\Vert(u,v)\Vert_{\mathcal{X}(t)} (1\leq\tau\leq t)$,  we get
\begin{equation}\label{th4case15}
\begin{aligned}
&\Vert H_2(u)(t,\cdot)\Vert_{L^2}\lesssim t^{-(m+1)\frac{n}{2}+\frac{\sqrt{\delta_2}-\mu_2+1}{2}}\int_1^t\tau^{(m+1)n+
\beta_2+(-(m+1)n-\beta_1+1)q}M_1(\tau,u)^qd\tau\\
&\lesssim t^{-(m+1)\frac{n}{2}+\frac{\sqrt{\delta_2}-\mu_2+1}{2}}\Vert(u,v)\Vert_
{\mathcal{X}(t)}^q\int_1^t\tau^{(m+1)n+
\beta_2+(-(m+1)n-\beta_1+1+\alpha_1)q}d\tau\\
&\lesssim t^{-(m+1)\frac{n}{2}+\frac{\sqrt{\delta_2}-\mu_2+1}{2}}\Vert(u,v)\Vert_
{\mathcal{X}(t)}^q.
\end{aligned}
\end{equation}
Analogously, we derive
\begin{align}
&\Vert \partial_tH_2(u)(t,\cdot)\Vert_{L^2}\lesssim t^{m-(m+1)(1+\frac{n}{2})+\frac{\sqrt{\delta_2}-\mu_2+1}{2}}\Vert(u,v)\Vert_
{\mathcal{X}(t)}^q,\label{th4case16}\\
&\Vert H_2(u)(t,\cdot)\Vert_{\dot{H}^\sigma}\lesssim t^{-(m+1)(\sigma+\frac{n}{2})+\frac{\sqrt{\delta_2}-\mu_2+1}{2}}\Vert(u,v)\Vert_
{\mathcal{X}(t)}^q,\label{th4case17}\\
&\Vert \partial_tH_2(u)(t,\cdot)\Vert_{\dot{H}^{\sigma-1}}\lesssim t^{m-(m+1)(\sigma+\frac{n}{2})+\frac{\sqrt{\delta_2}-\mu_2+1}{2}}\Vert(u,v)\Vert_
{\mathcal{X}(t)}^q.\label{th4case18}
\end{align}
As a result, (\ref{th4case15})-(\ref{th4case18}) give
\begin{equation}\label{M112}
M_2(t,H_2(u))\lesssim \Vert(u,v)\Vert_
{\mathcal{X}(t)}^p.
\end{equation}
Combining (\ref{M111}) and (\ref{M112}) yields
\begin{equation}\label{un2}
\Vert \big(H_1(v),H_2(u)\big)\Vert_{\mathcal{X}(T)}\lesssim \Vert(u,v)\Vert_
{\mathcal{X}(T)}^p+\Vert(u,v)\Vert_
{\mathcal{X}(T)}^q,
\end{equation}
together with (\ref{ulvllines}), we get the validness of (\ref{key421}).

In view of the definition of the norm of $\mathcal{X}(t)$, we have $M_2(\tau,v)\lesssim\Vert(u,v)\Vert_{\mathcal{X}(t)}$, $M_2(\tau,\tilde{v})\lesssim\Vert(\tilde{u},\tilde{v})\Vert_{\mathcal{X}(t)}$ and $M_2(\tau,v-\tilde{v})\lesssim\Vert(u,v)-(\tilde{u},\tilde{v})\Vert_{\mathcal{X}(t)}$. Following the analysis of (\ref{th4case111}) and (\ref{th4case112}), we obtain
\begin{equation}\label{u-vcase11}
\resizebox{0.88\hsize}{!}{$
\begin{aligned}
&\Vert H_1(v)(t,\cdot)-H_1(\tilde{v})(t,\cdot)\Vert_{L^2}\lesssim t^{-(m+1)\frac{n}{2}+\frac{\sqrt{\delta_1}-\mu_1+1}{2}}\\
&\quad\times\int_1^t\tau^{(m+1)n+\beta_1+
(-(m+1)n-\beta_2+1)p} M_2(\tau,v-\tilde{v})\big(M_2(\tau,v)^{p-1}+M_2(\tau,\tilde{v})^{p-1}\big)d\tau\\
&\lesssim t^{-(m+1)\frac{n}{2}+\frac{\sqrt{\delta_1}-\mu_1+1}{2}+\alpha_1}\Vert(u,v)-(\tilde{u},\tilde{v})\Vert_{\mathcal{X}(t)}
\big(\Vert(u,v)\Vert_{\mathcal{X}(t)}^{p-1}+\Vert(\tilde{u},\tilde{v})\Vert_{\mathcal{X}
(t)}^{p-1}\big),
\end{aligned}$}
\end{equation}
\begin{align}
&\Vert \partial_t\big(H_1(v)(t,\cdot)-H_1(\tilde{v})(t,\cdot)\big)\Vert_{L^2}\notag\\
&\lesssim t^{m-(m+1)(1+\frac{n}{2})+\frac{\sqrt{\delta_1}-\mu_1+1}{2}+\alpha_1}\Vert(u,v)-(\tilde{u},\tilde{v})\Vert_{\mathcal{X}(t)}
\big(\Vert(u,v)\Vert_{\mathcal{X}(t)}^{p-1}+\Vert(\tilde{u},\tilde{v})\Vert_{\mathcal{X}
(t)}^{p-1}\big),\label{u-vcase12}\\
&\Vert H_1(v)(t,\cdot)-H_1(\tilde{v})(t,\cdot)\Vert_{\dot{H}^\sigma}\notag\\
&\lesssim t^{-(m+1)(\sigma+\frac{n}{2})+\frac{\sqrt{\delta_1}-\mu_1+1}{2}+\alpha_1}\Vert(u,v)-(\tilde{u},\tilde{v})\Vert_{\mathcal{X}(t)}
\big(\Vert(u,v)\Vert_{\mathcal{X}(t)}^{p-1}+\Vert(\tilde{u},\tilde{v})\Vert_{\mathcal{X}
(t)}^{p-1}\big),\label{u-vcase13}\\
&\Vert \partial_t\big(H_1(v)(t,\cdot)-H_1(\tilde{v})(t,\cdot)\big)\Vert_{\dot{H}^{\sigma-1}}\notag\\
&\lesssim t^{m-(m+1)(\sigma+\frac{n}{2})+\frac{\sqrt{\delta_1}-\mu_1+1}{2}+\alpha_1}\Vert(u,v)-(\tilde{u},\tilde{v})\Vert_{\mathcal{X}(t)}
\big(\Vert(u,v)\Vert_{\mathcal{X}(t)}^{p-1}+\Vert(\tilde{u},\tilde{v})\Vert_{\mathcal{X}
(t)}^{p-1}\big).\label{u-vcase14}
\end{align}
Consequently, (\ref{u-vcase11})-(\ref{u-vcase14}) yield
\begin{equation}\label{1case01}
\resizebox{0.88\hsize}{!}{$
M_1(t,H_1(v)-H_1(\tilde{v}))\lesssim t^{\alpha_1}\Vert(u,v)-(\tilde{u},\tilde{v})\Vert_{\mathcal{X}(t)}
\big(\Vert(u,v)\Vert_{\mathcal{X}(t)}^{p-1}+\Vert(\tilde{u},\tilde{v})\Vert_{\mathcal{X}
(t)}^{p-1}\big).$}
\end{equation}
Note that \begin{align*}
&M_1(\tau,u-\tilde{u})\lesssim\tau^{\alpha_1}\Vert(u,v)-(\tilde{u},\tilde{v})
\Vert_{\mathcal{X}(t)}, \\
&M_1(\tau,u)\lesssim \tau^{\alpha_1}\Vert(u,v)\Vert_{\mathcal{X}(t)},\\
&M_1(\tau,\tilde{u})\lesssim \tau^{\alpha_1}\Vert(\tilde{u},\tilde{v})\Vert_{\mathcal{X}(t)},
\end{align*}
and by the analysis of (\ref{ingreinfty}), we have
\begin{align}
&\Vert H_2(u)(t,\cdot)-H_2(\tilde{u})(t,\cdot)\Vert_{L^2}\lesssim t^{-(m+1)\frac{n}{2}+\frac{\sqrt{\delta_2}-\mu_2-1}{2}}\int_1^t\tau^{(m+1)n+
\beta_2+(-(m+1)n-\beta_1+1)q}\notag\\
&\quad\quad\quad\quad\quad\quad\times M_1(\tau,u-\tilde{u})\big(M_1(\tau,u)^{q-1}+M_1(\tau,\tilde{u})^{q-1}\big)d\tau\notag\\
&\lesssim t^{-(m+1)\frac{n}{2}+\frac{\sqrt{\delta_2}-\mu_2-1}{2}}\int_1^t\tau^{(m+1)n+
\beta_2+(-(m+1)n-\beta_1+1+\alpha_1)q}d\tau\notag\\
&\quad\quad\quad\quad\quad\quad\quad\quad\times\Vert(u,v)-(\tilde{u},\tilde{v})
\Vert_{\mathcal{X}(t)}\big(\Vert(u,v)\Vert_{\mathcal{X}(t)}^{q-1}+
\Vert(\tilde{u},\tilde{v})\Vert_{\mathcal{X}(t)}^{q-1}\big)\notag\\
&\lesssim t^{-(m+1)\frac{n}{2}+\frac{\sqrt{\delta_2}-\mu_2-1}{2}}\Vert(u,v)-(\tilde{u},\tilde{v})
\Vert_{\mathcal{X}(t)}\big(\Vert(u,v)\Vert_{\mathcal{X}(t)}^{q-1}+
\Vert(\tilde{u},\tilde{v})\Vert_{\mathcal{X}(t)}^{q-1}\big),\label{1case115}\\
&\Vert \partial_t\big(H_2(u)(t,\cdot)-H_2(\tilde{u})(t,\cdot)\big)\Vert_{L^2}\notag\\
&\lesssim t^{m-(m+1)(1+\frac{n}{2})+\frac{\sqrt{\delta_2}-\mu_2+1}{2}}\Vert(u,v)-(\tilde{u},\tilde{v})
\Vert_{\mathcal{X}(t)}\big(\Vert(u,v)\Vert_{\mathcal{X}(t)}^{q-1}+
\Vert(\tilde{u},\tilde{v})\Vert_{\mathcal{X}(t)}^{q-1}\big),\label{1case116}\\
&\Vert H_2(u)(t,\cdot)-H_2(\tilde{u})(t,\cdot)\Vert_{\dot{H}^\sigma}\notag\\
&\lesssim t^{-(m+1)(\sigma+\frac{n}{2})+\frac{\sqrt{\delta_2}-\mu_2-1}{2}}\Vert(u,v)-(\tilde{u},\tilde{v})
\Vert_{\mathcal{X}(t)}\big(\Vert(u,v)\Vert_{\mathcal{X}(t)}^{q-1}+
\Vert(\tilde{u},\tilde{v})\Vert_{\mathcal{X}(t)}^{q-1}\big),\label{1case117}\\
&\Vert \partial_t\big(H_2(u)(t,\cdot)-H_2(\tilde{u})(t,\cdot)\big)\Vert_{\dot{H}^{\sigma-1}}\notag\\
&\lesssim t^{m-(m+1)(\sigma+\frac{n}{2})+\frac{\sqrt{\delta_2}-\mu_2-1}{2}}\Vert(u,v)-(\tilde{u},\tilde{v})
\Vert_{\mathcal{X}(t)}\big(\Vert(u,v)\Vert_{\mathcal{X}(t)}^{q-1}+
\Vert(\tilde{u},\tilde{v})\Vert_{\mathcal{X}(t)}^{q-1}\big).\label{1case118}
\end{align}
Then (\ref{1case115})-(\ref{1case118}) show that
\begin{equation}\label{1case1101}
M_2(t,H_2(u)-H_2(\tilde{u}))\lesssim \Vert(u,v)-(\tilde{u},\tilde{v})
\Vert_{\mathcal{X}(t)}\big(\Vert(u,v)\Vert_{\mathcal{X}(t)}^{q-1}+
\Vert(\tilde{u},\tilde{v})\Vert_{\mathcal{X}(t)}^{q-1}\big).
\end{equation}

It follows from (\ref{1case01}) and (\ref{1case1101}) that (\ref{key422}) holds.
\end{proof}

\begin{remark}\label{re444}
In the case of  $\delta_2=(m+1)^2(n+2\sigma-1)^2$, we see that $\mathcal{l}_2(\tau)=(1+\log \tau)^{\frac{1}{2}}$. Following the above process, we get
$$\Vert H_1(v)(t,\cdot)\Vert_{L^2}\lesssim t^{-(m+1)\frac{n}{2}+\frac{\sqrt{\delta_1}-\mu_1+1}{2}}\Vert(u,v)
\Vert_{\mathcal{X}(t)}^p\times(I),$$
where
\begin{equation}\label{I}
(I):=\int_1^t\tau^{(m+1)n+\beta_1+(-(m+1)n-\beta_2+1)p}(1+\log\tau)
^{\frac{p\theta_1(2p)}{2}}d\tau.
\end{equation}
In view of $(m+1)n+\beta_1+(-(m+1)n-\beta_2+1)p\geq-1$ under the condition $p\leq\tilde{p}$, the  occurrence  of $(1+\log\tau)
^{\frac{p\theta_1(2p)}{2}}$ in (I) prevents us from handling the integral in the manner outlined in (\ref{th4case111}). In the estimates of $\Vert H_1(v)(t,\cdot)\Vert_{\dot{H}^\sigma}$, $\Vert \partial_tH_1(v)(t,\cdot)\Vert_{L^{2}}$, $\Vert \partial_tH_1(v)(t,\cdot)\Vert_{\dot{H}^{\sigma-1}}$, $\Vert H_1(v)(t,\cdot)-H_1(\tilde{v})(t,\cdot)\Vert_{L^2}$, $\Vert H_1(v)(t,\cdot)-H_1(\tilde{v})(t,\cdot)\Vert_{\dot{H}^\sigma}$, $\Vert \partial_t\big(H_1(v)(t,\cdot)-H_1(\tilde{v})(t,\cdot)\big)\Vert_{L^2}$, $\Vert \partial_t\big(H_1(v)(t,\cdot)-H_1(\tilde{v})(t,\cdot)\big)\Vert_{\dot{H}^{\sigma-1}}$,  the factor $(1+\log\tau)$ will still appear, which is beyond the current technical capabilities of this paper. Therefore, in Theorem \ref{theorem12}, we do not consider the case of $\delta_2=(m+1)^2(n+2\sigma-1)^2$.

\end{remark}

\subsection{The proof of Theorem \ref{theorem13}}\label{proofthm13}
By exchanging the positions of $p$ and $q$ in the proof process of Theorem \ref{theorem12}, we can immediately obtain  Theorem  \ref{theorem13}, so we will not elaborate on the proof of Theorem \ref{theorem13} any further.

\section{Global existence of  solutions with low regularity of initial data}\label{section5}
In this section, we establish the global existence  of the solution  with low regularity of initial data, specifically Theorems \ref{theorem21}-\ref{theorem23}. The proofs  are similar to those of Theorems \ref{theorem11}-\ref{theorem13}. Here, we briefly outline the differences. For specific details, one can refer to Sections \ref{prooftheorem2}-\ref{proofthm13},  according to the correspondence between the three groups of theorems as in Table \ref{tab1}.
\begin{table}[ht]
\centering
\begin{tabular}{c}
\toprule
\ Theorem \ref{theorem11} \quad$\longleftrightarrow$\quad Theorem \ref{theorem21} \quad \\
\ Theorem \ref{theorem12} \quad$\longleftrightarrow$\quad Theorem \ref{theorem22} \quad  \\
\ Theorem \ref{theorem13} \quad$\longleftrightarrow$\quad Theorem \ref{theorem23} \quad  \\
\bottomrule
\end{tabular}
\vspace{2pt}
\caption{}\label{tab1}
\end{table}

We need to make the corresponding modifications to the function space and its norm defined in Section \ref{prooftheorem2}.
For $\sigma\in(0,1)$, let
\begin{align}
W_1(t,u)=&t^{-\frac{\sqrt{\delta_1}-\mu_1+1}{2}+(m+1)\frac{n}{2}}\Vert u(t,\cdot)\Vert_{L^2}\notag+t^{-\frac{\sqrt{\delta_1}-\mu_1+1}{2}+(m+1)(\sigma+\frac{n}{2})}\ell_1^{-1}(t)\Vert u(t,\cdot)\Vert_{\dot{H}^\sigma},\notag\\
W_2(t,v)=&t^{-\frac{\sqrt{\delta_2}-\mu_2+1}{2}+(m+1)\frac{n}{2}}\Vert v(t,\cdot)\Vert_{L^2}+t^{-\frac{\sqrt{\delta_2}-\mu_2+1}{2}+(m+1)(\sigma+\frac{n}{2})}\ell_2^{-1}(t)\Vert v(t,\cdot)\Vert_{\dot{H}^\sigma},\notag
\end{align}
where $\ell_i(t) (i=1,2)$ are shown in (\ref{l(t)}).
Define the function space
\begin{equation*}
\mathcal{Y}(T):=\Big\{(u,v)\in\Big(C\bigl([1,T];H^\sigma\bigr)\Big)^2\ \ \text{such that}\ \ \text{supp}\big(u(t,\cdot),v(t,\cdot)\big)\subset B_{\phi_m(t)-\phi_m(1)+M}\Big\}
\end{equation*}
equipped with the norm
\begin{equation}\label{def51functionspace}
\Vert(u,v)\Vert_{\mathcal{Y}(T)}=\sup\limits_{t\in[1,T]}\big(t^{-\alpha_1}W_1(t,u)+t^{-\alpha_2}W_2(t,v)\big),
\end{equation}
here $M>0$,  $\alpha_i ,i=1,2$ are the same as (\ref{alpha1}) and (\ref{alpha2}) and $\phi_m(t)$ is defined by (\ref{compact2}).

By following the proof  of Lemma \ref{foures} and Lemma \ref{8cases}, we obtain
\begin{lemma}\label{lem5.1}
In the case of $\delta_1,\delta_2\geq(m+1)^2(n+2\sigma-1)^2$, $p,q>1$, for any $1\leq\tau< T$ and any $(u,v),(\tilde{u},\tilde{v})\in \mathcal{Y}(T)$,   we have
\begin{equation*}
\begin{aligned}
&\Vert \vert v(\tau,\cdot)\vert^p\Vert_{L^1}\lesssim \tau^{(m+1)n+(-(m+1)n-\beta_2+1)p}\ell_2(\tau)^{p\theta_1(2p)}W_2(\tau,v)^p,\\
&\Vert \vert v(\tau,\cdot)\vert^p\Vert_{L^2}\lesssim\tau^{(m+1)\frac{n}{2}+(-(m+1)n-\beta_2+1)p}\ell_2
(\tau)^{p\theta_1(2p)}W_2(\tau,v)^p,\\
&\Vert \vert u(\tau,\cdot)\vert^q\Vert_{L^1}\lesssim \tau^{(m+1)n+(-(m+1)n-\beta_1+1)q}\ell_1(\tau)^{q\theta_1(2q)}W_1(\tau,u)^q,\label{uql1}\\
&\Vert \vert u(\tau,\cdot)\vert^q\Vert_{L^2}\lesssim\tau^{(m+1)\frac{n}{2}+(-(m+1)n-\beta_1+1)q}\ell_1
(\tau)^{q\theta_1(2q)}W_1(\tau,u)^q
\end{aligned}
\end{equation*}
and
\begin{equation*}
\begin{aligned}
&\Vert\vert v(\tau,\cdot)\vert^p-\vert \tilde{v}(\tau,\cdot)\vert^p\Vert_{L^1}\\
&\lesssim\tau^{(m+1)n+(-(m+1)n-\beta_2+1)p}
\ell_2(\tau)^{p\theta_1(2p)}W_2(\tau,v-\tilde{v})
\big(W_2(\tau,v)^{p-1}+W_2(\tau,\tilde{v})^{p-1}\big),\\
&\Vert\vert v(\tau,\cdot)\vert^p-\vert \tilde{v}(\tau,\cdot)\vert^p\Vert_{L^2}\\
&\lesssim\tau^{(m+1)\frac{n}{2}+(-(m+1)n-\beta_2+1)p}
\ell_2(\tau)^{p\theta_1(2p)}W_2(\tau,v-\tilde{v})
\big(W_2(\tau,v)^{p-1}+W_2(\tau,\tilde{v})^{p-1}\big),\\
&\Vert\vert u(\tau,\cdot)\vert^q-\vert \tilde{u}(\tau,\cdot)\vert^q\Vert_{L^1}\\
&\lesssim\tau^{(m+1)n+(-(m+1)n-\beta_1+1)q}
\ell_1(\tau)^{q\theta_1(2q)}W_1(\tau,u-\tilde{u})
\big(W_1(\tau,u)^{q-1}+W_1(\tau,\tilde{u})^{q-1}\big),\\
&\Vert\vert u(\tau,\cdot)\vert^q-\vert \tilde{u}(\tau,\cdot)\vert^q\Vert_{L^2}\\
&\lesssim\tau^{(m+1)\frac{n}{2}+(-(m+1)n-\beta_1+1)q}
\ell_1(\tau)^{q\theta_1(2q)}W_1(\tau,u-\tilde{u})
\big(W_1(\tau,u)^{q-1}+W_1(\tau,\tilde{u})^{q-1}\big),
\end{aligned}
\end{equation*}
where $\theta_1, \beta_i ,i=1,2$ are defined by (\ref{theta1}) and (\ref{E}), respectively.
\end{lemma}
\begin{remark}\label{pqexsit}
We point out that the conditions
\begin{equation}\label{ensure}
\left
\{
\begin{aligned}
&1< p,q\ &\quad& n\leq 2\sigma,\\
&1< p,q\leq \frac{n}{n-2\sigma},\ &\quad& n> 2\sigma
\end{aligned}
\right.
\end{equation}
in Theorems \ref{theorem21}-\ref{theorem23}  can ensure that $\theta_1(2p),\theta_1(2q)\in[0,1]$.
\end{remark}
Following Lemma \ref{lem5.1}  and  the proof techniques of Propositions \ref{keypro}-\ref{keyyypro}, we can demonstrate that the following conclusions hold under the conditions of Theorems \ref{theorem21}-\ref{theorem23}, respectively.
\begin{proposition}\label{sec5keyconsi}
There exists a constant $C>0$ such that for any $T>1$  and any $(u,v), (\tilde{u},\tilde{v})\in \mathcal{Y}(T)$, we have
\begin{align}
&\Vert \mathcal{N}(u,v)\Vert_{\mathcal{Y}(T)}\leq C\bigl(\Vert(u_0,u_1)\Vert_{D^\sigma}+\Vert(v_0,v_1)\Vert_{D^\sigma}\big)+C\big(\Vert (u,v)\Vert_{\mathcal{Y}{(T)}}^p+\Vert (\tilde{u},\tilde{v})\Vert_{\mathcal{Y}{(T)}}^q\bigr),\label{kkkkey1}\\
&\Vert \mathcal{N}(u,v)-\mathcal{N}(\tilde{u},\tilde{v})\Vert_{\mathcal{Y}(T)}\leq C \Vert (u,v)-(\tilde{u},\tilde{v})\Vert_{\mathcal{Y}(T)}\notag\\
&\quad\quad\quad\times\bigl(\Vert (u,v)\Vert_{\mathcal{Y}(T)}^{p-1}+\Vert (u,v)\Vert_{\mathcal{Y}(T)}^{q-1}+\Vert (\tilde{u},\tilde{v})\Vert_{\mathcal{Y}(T)}^{p-1}+\Vert (\tilde{u},\tilde{v})\Vert_{\mathcal{Y}(T)}^{q-1}\bigr),\label{kkkkey2}
\end{align}
where $\mathcal{N}$ is given by (\ref{operator}).
\end{proposition}

With proposition \ref{sec5keyconsi}, we can use the contraction mapping principle to give the global existence. 

\section*{Appendix}
\setcounter{theorem}{0}
\renewcommand{\thetheorem}{A.\arabic{theorem}}
\setcounter{lemma}{0}
\renewcommand{\thelemma}{A.\arabic{lemma}}
\setcounter{equation}{0}
\renewcommand{\theequation}{A.\arabic{equation}}
\renewcommand{\theremark}{A.\arabic{equation}}

The proof framework and methods for the local solutions, namely, Proposition \ref{localexistence}, are comparable to those for the global solutions, so we give a brief outline. For $\sigma>1$, let
\begin{equation*}
\begin{aligned}
&X(T):=\Big\{(u,v)\in(C\bigl([1,T];H^\sigma\bigr)\cap C^1\bigl([1,T];H^{\sigma-1}\bigr)\Big)^2\\
&\quad\quad\quad\quad\quad\quad \text{such that}\ \ \text{supp}\big(u(t,\cdot),v(t,\cdot)\big)\subset B_{\phi_m(t)-\phi_m(1)+M}\Big\},
\end{aligned}
\end{equation*}
and for $0<\sigma<1$,
\begin{equation*}
X(T):=\Big\{(u,v)\in\Big(C\bigl([1,T];H^\sigma\bigr)\Big)^2\ \ \text{such that}\ \ \text{supp}\big(u(t,\cdot),v(t,\cdot)\big)\subset B_{\phi_m(t)-\phi_m(1)+M}\Big\}
\end{equation*}
with the norm $\Vert (u,v)\Vert_{X(T)}=\max\limits_{t\in[1,T]} \big(Q[u](t)+Q[v](t)\big)$, where
\begin{equation*}
Q[w](t)=\left
\{
\resizebox{0.86\hsize}{!}{$
\begin{aligned}
&\Vert w(t,\cdot)\Vert_{L^2}+\Vert w(t,\cdot)\Vert_{\dot{H}^{\sigma}}+\Vert\partial_tw(t,\cdot)\Vert_{L^2}+\Vert\partial_tw(t,\cdot)\Vert_{\dot{H}^{\sigma-1}}, &\ & \sigma\geq1,\\
&\Vert w(t,\cdot)\Vert_{L^2}+\Vert w(t,\cdot)\Vert_{\dot{H}^{\sigma}},&\ & \sigma\in (0,1).
\end{aligned}$}
\right.
\end{equation*}
For $T>1$, introduce the space $X(T,K)=\{(u,v)\in X(T): \Vert (u,v)\Vert_{X(T)}\leq K\}$ and the operator $\mathcal{N}$ as in (\ref{operator}), where  and $K>0$ is determined later. Following the proof of Proposition \ref{keypro}, we can prove
\begin{lemma}
For any $(u,v)$ and $(\tilde{u},\tilde{v})\in X(T,K)$, the following estimates hold
\begin{align}
&\Vert\mathcal{N}(u,v)\Vert_{X(T)}\leq C_{1T}\big(\Vert u_0,u_1\Vert_{D^\sigma}+\Vert v_0,v_1\Vert_{D^\sigma}\big)+C_{2T}\big(\Vert (u,v)\Vert_{X(T)}^{p}+\Vert (u,v)\Vert_{X(T)}^{q}\big),\label{kkey1}\\
&\Vert\mathcal{N}(u,v)-\mathcal{N}(\tilde{u},\tilde{v})\Vert_{X(T)}\leq C_{3T}\Vert (u,v)-(\tilde{u},\tilde{v})\Vert_{X(T)}\notag\\
&\quad\quad\quad\quad\times\bigl(\Vert (u,v)\Vert_{X(T)}^{p-1}+\Vert (\tilde{u},\tilde{v})\Vert_{X(T)}^{p-1}+\Vert (u,v)\Vert_{X(T)}^{q-1}+\Vert (\tilde{u},\tilde{v})\Vert_{X(T)}^{q-1}\bigr),\label{kkey2}
\end{align}
where $C_{1T}$ is bounded,   and $C_{2T}, C_{3T}\rightarrow0$,  as $T\rightarrow 1^{+}$.
\end{lemma}
\begin{proof}[Proof of Proposition \ref{localexistence}]
By (\ref{kkey1}) and (\ref{kkey2}), for sufficiently large $K$,  we choose  $T$ sufficiently close to $1^+$ such that  $\mathcal{N}$ maps $X(T,K)$ into itself and  $\mathcal{N}$ is a contraction mapping. Thus, we establish the local existence and uniqueness of the solution to (\ref{eqs}) in $X(T)$.
\end{proof}

\section*{Acknowledgments}

This work is supported by the NSF of China (11731007), the Priority
Academic Program Development of Jiangsu Higher Education Institutions, the NSF of Jiangsu Province
(BK20221320).

\renewcommand{\theequation}{A\arabic{equation}}
\setcounter{equation}{0}
%
%

\end{document}